\newif\ifEWzJhvOnScreen\EWzJhvOnScreentrue 

\RequirePackage{iftex}
\newif\ifEWzJhvPDF\EWzJhvPDFfalse 
\ifpdf\EWzJhvPDFtrue\fi

\newif\ifEWzJhvPageBreak\EWzJhvPageBreakfalse
\ifEWzJhvPDF\EWzJhvPageBreaktrue\fi

\makeatletter
\@ifdefinable\EWzJhvProtectedDefNew{}\protected\def\EWzJhvProtectedDefNew#1{\@ifdefinable#1{}\protected\def#1}
\EWzJhvProtectedDefNew\EWzJhvDefNew#1{\@ifdefinable#1{}\def#1}
\EWzJhvProtectedDefNew\EWzJhvLetNew#1{\@ifdefinable#1{}\let#1}
\EWzJhvProtectedDefNew\EWzJhvProtectedLongDefNew#1{\@ifdefinable#1{}\protected\long\def#1}
\makeatother

\documentclass[11pt, a4paper]{article}

\makeatletter\EWzJhvProtectedDefNew\EWzJhvTocBeginGroupSectionNoPageNumber{\begingroup\EWzJhvLetNew\EWzJhvOriginalLAtSection=\l@section\protected\def\l@section##1##2{\EWzJhvOriginalLAtSection{##1}{}}}\makeatother
\EWzJhvProtectedDefNew\EWzJhvTocEndGroup{\endgroup}
\EWzJhvProtectedDefNew\EWzJhvTocNoIndentBeginMinipage{\noindent\begin{minipage}{\linewidth}}
\EWzJhvProtectedDefNew\EWzJhvTocEndMinipage{\end{minipage}}

\makeatletter
\EWzJhvProtectedDefNew\EWzJhvSubsectionStarWithSecondLine#1#2{\@startsection{subsection}{2}{0pt}{-3.25ex plus -1ex minus -.2ex}{1sp}{\normalfont\large\bfseries}*{#1}\begingroup\interlinepenalty=10000\relax#2\par\endgroup\nobreak\vspace{1.5ex plus .2ex}\@afterheading}
\EWzJhvProtectedDefNew\EWzJhvMiniSubsection{\@startsection{subsection}{2}{0pt}{-3.25ex plus -1ex minus -.2ex}{1.5ex plus .2ex}{\normalfont\normalsize\bfseries}}
\EWzJhvProtectedDefNew\EWzJhvParagraphNoBeforeskip{\@startsection{paragraph}{4}{0pt}{0pt}{-1em}{\normalfont\normalsize\bfseries}}
\makeatother

\ifEWzJhvPDF
\usepackage[OT1]{fontenc}
\makeatletter
\EWzJhvProtectedDefNew\EWzJhvSwitchToTOneLM@from@cmr{\fontencoding{T1}\fontfamily{lmr}\selectfont}
\EWzJhvProtectedDefNew\EWzJhvSwitchToTOneLM@from@cmss{\fontencoding{T1}\fontfamily{lmss}\selectfont}
\EWzJhvProtectedDefNew\EWzJhvSwitchToTOneLM@from@cmtt{\fontencoding{T1}\fontfamily{lmtt}\selectfont}
\EWzJhvProtectedDefNew\EWzJhvSwitchToTOneLM{\ifcsname EWzJhvSwitchToTOneLM@from@\f@family\endcsname\csname EWzJhvSwitchToTOneLM@from@\f@family\endcsname\else\errmessage{\string\EWzJhvSwitchToTOneLM: current value "\f@family" of \string\f@family\space is unsupported.}\fontencoding{T1}\selectfont\fi}
\DeclareTextCommand{\"}{OT1}[1]{\leavevmode\begingroup\EWzJhvSwitchToTOneLM\"{#1}\endgroup}
\DeclareTextCommand{\textasteriskcentered}{OT1}{\leavevmode\begingroup\EWzJhvSwitchToTOneLM\textasteriskcentered\endgroup}
\DeclareTextCommand{\textbullet}{OT1}{\leavevmode\begingroup\EWzJhvSwitchToTOneLM\textbullet\endgroup}
\EWzJhvProtectedDefNew\EWzJhvText@cmr@m@n@bab{ba\kern-0.02em\relax b}
\EWzJhvProtectedDefNew\EWzJhvText@cmr@m@n@cab{ca\kern-0.03em\relax b}
\EWzJhvProtectedDefNew\EWzJhvText@cmr@m@n@kab{k\kern-0.01em\relax a\kern-0.04em\relax b}
\EWzJhvProtectedDefNew\EWzJhvText@cmr@m@n@mab{ma\kern-0.02em\relax b}
\EWzJhvProtectedDefNew\EWzJhvText@cmr@m@n@nab{na\kern-0.02em\relax b}
\EWzJhvProtectedDefNew\EWzJhvText@cmr@m@n@rab{ra\kern-0.03em\relax b}
\EWzJhvProtectedDefNew\EWzJhvText@cmr@bx@n@rab{ra\kern-0.03em\relax b}
\EWzJhvProtectedDefNew\EWzJhvText@cmr@m@n@Tab{Ta\kern-0.03em\relax b}
\expandafter\EWzJhvProtectedDefNew\csname EWzJhvText@cmr@m@n@\detokenize{käh}\endcsname{kä\kern-0.03em\relax h}
\EWzJhvProtectedDefNew\EWzJhvText#1{\ifcsname EWzJhvText@\f@family @\f@series @\f@shape @\detokenize{#1}\endcsname\csname EWzJhvText@\f@family @\f@series @\f@shape @\detokenize{#1}\endcsname\else\errmessage{\string\EWzJhvText: EWzJhvText@\f@family @\f@series @\f@shape @\detokenize{#1} not found.}#1\fi}
\makeatother
\else
\EWzJhvProtectedDefNew\EWzJhvText#1{#1}
\fi

\usepackage[hmargin=1.25in, vmargin=1in]{geometry}
\usepackage{indentfirst}
\usepackage{enumitem}
\usepackage{amsmath}
\usepackage{amsfonts}
\usepackage{amssymb}
\usepackage{amsthm}
\usepackage{mathrsfs}
\usepackage{aliascnt}

\usepackage[activate={compatibility,false}, expansion={true,compatibility}, patch=none]{microtype}
\DeclareMicrotypeSet*[expansion]{EWzJhvAllTextNotT}{
	encoding = {OT1,T1,T2A,LY1,OT4,QX,T5,TS1,EU1,EU2,TU},
	family = {cmr,cmss,lmr,lmss}}

\usepackage{xcolor}

\ifEWzJhvOnScreen
\usepackage[psdextra, colorlinks, allcolors=blue!40!black]{hyperref}
\else
\usepackage[psdextra, hidelinks]{hyperref}
\fi

\ifEWzJhvPDF\pdfinterwordspaceon\fi

\newtheorem{EWzJhvDefinition}{Definition}[section]
\EWzJhvDefNew\EWzJhvDefinitionautorefname{Definition}

\newaliascnt{EWzJhvProposition}{EWzJhvDefinition}
\newtheorem{EWzJhvProposition}[EWzJhvProposition]{Proposition}
\aliascntresetthe{EWzJhvProposition}
\EWzJhvDefNew\EWzJhvPropositionautorefname{Proposition}

\newaliascnt{EWzJhvLemma}{EWzJhvDefinition}
\newtheorem{EWzJhvLemma}[EWzJhvLemma]{Lemma}
\aliascntresetthe{EWzJhvLemma}
\EWzJhvDefNew\EWzJhvLemmaautorefname{Lemma}

\newaliascnt{EWzJhvTheorem}{EWzJhvDefinition}
\newtheorem{EWzJhvTheorem}[EWzJhvTheorem]{Theorem}
\aliascntresetthe{EWzJhvTheorem}
\EWzJhvDefNew\EWzJhvTheoremautorefname{Theorem}

\newaliascnt{EWzJhvCorollary}{EWzJhvDefinition}
\newtheorem{EWzJhvCorollary}[EWzJhvCorollary]{Corollary}
\aliascntresetthe{EWzJhvCorollary}
\EWzJhvDefNew\EWzJhvCorollaryautorefname{Corollary}

\newtheorem*{EWzJhvDefinitionUnnumbered}{Definition}
\newtheorem*{EWzJhvPropositionUnnumbered}{Proposition}
\newtheorem*{EWzJhvTheoremUnnumbered}{Theorem}

\theoremstyle{remark}
\newtheorem*{EWzJhvRemark}{Remark}

\ExplSyntaxOn

\cs_new_nopar:Nn \EWzJhv_bibcite_sanitize_char:n {
	\bool_lazy_any:nT
	{
		{ !( \str_compare_p:nNn { #1 } < { 0 } ) && !( \str_compare_p:nNn { #1 } > { 9 } ) }
		{ !( \str_compare_p:nNn { #1 } < { A } ) && !( \str_compare_p:nNn { #1 } > { Z } ) }
		{ !( \str_compare_p:nNn { #1 } < { a } ) && !( \str_compare_p:nNn { #1 } > { z } ) }
	}
	{ #1 }
}

\str_new:N \l_EWzJhv_bibcite_temp_str
\EWzJhvLetNew\EWzJhvOriginalBibcite=
\protected\def\bibcite#1#2{
	\str_gclear_new:c { EWzJhv@b@ #1 }
	\str_gset:cn { EWzJhv@b@ #1 } { #2 }
	\str_greplace_all:cnn { EWzJhv@b@ #1 } { ä } { ae }
	\group_begin:
		\str_set:Ne \l_EWzJhv_bibcite_temp_str { \str_map_function:cN { EWzJhv@b@ #1 } \EWzJhv_bibcite_sanitize_char:n }
		\str_if_eq:NcF \l_EWzJhv_bibcite_temp_str { EWzJhv@b@ #1 }
		{
			\PackageWarning { EWzJhv } { bibcite~sanitized:~" \str_use:c { EWzJhv@b@ #1 } " }
			\str_gset_eq:cN { EWzJhv@b@ #1 } \l_EWzJhv_bibcite_temp_str
		}
	\group_end:
	\EWzJhvOriginalBibcite{#1}{#2}
}

\int_const:Nn \c_EWzJhv_prefix_length_int { \str_count:n { EWzJhv } }
\int_const:Nn \c_EWzJhv_prefix_length_plus_one_int { \c_EWzJhv_prefix_length_int + 1 }

\cs_new_nopar:Nn \EWzJhv_hyper_dest_name_filter:n {
	\str_if_eq:eeTF
	{ \str_range:nnn { #1 } { 1 } { \c_EWzJhv_prefix_length_int } }
	{ EWzJhv }
	{ \str_range:nnn { #1 } { \c_EWzJhv_prefix_length_plus_one_int } { -1 } }
	{
		\str_if_eq:eeTF
		{ \str_range:nnn { #1 } { 1 } { 5 } }
		{ cite. }
		{
			cite.
			\cs_if_exist_use:cF
			{ EWzJhv@b@ \str_range:nnn { #1 } { 6 } { -1 } }
			{ missing. \str_range:nnn { #1 } { 6 } { -1 } }
		}
		{ #1 }
	}
}
\cs_generate_variant:Nn \EWzJhv_hyper_dest_name_filter:n { e }

\EWzJhvLetNew\EWzJhvOriginalHyperDestNameFilter=\HyperDestNameFilter
\def\HyperDestNameFilter#1{ \EWzJhv_hyper_dest_name_filter:e { \EWzJhvOriginalHyperDestNameFilter{#1} } }

\makeatletter
\EWzJhvProtectedDefNew\EWzJhvTextNormal@nocorr#1{\textnormal{\nocorr#1\nocorr}}
\EWzJhvProtectedDefNew\EWzJhvTextNormal{
	\mode_leave_vertical:
	\bool_lazy_all:nTF
	{
		{ \str_if_eq_p:ee { \f@family } { \familydefault } }
		{ \str_if_eq_p:ee { \f@series } { \seriesdefault } }
		{ \str_if_eq_p:ee { \f@shape } { \shapedefault } }
	}
	{ \EWzJhvTextNormal@nocorr }
	{ \textnormal }
}
\makeatother

\ExplSyntaxOff

\makeatletter
\ifx\@cite@ofmt\hbox\else\errmessage{Unexpected definition of \string\@cite@ofmt.}\fi
\protected\def\@cite@ofmt{}
\makeatother

\makeatletter
\EWzJhvProtectedDefNew\EWzJhvMathOP#1{\mathop{\kern\z@#1}\nolimits}
\EWzJhvProtectedDefNew\EWzJhvMathVCenter@#1#2{\vcenter{\hbox{$#1#2$}}}
\EWzJhvProtectedDefNew\EWzJhvMathVCenter{\mathpalette\EWzJhvMathVCenter@}
\EWzJhvProtectedDefNew\EWzJhvMathFracVCentered#1#2{\frac{\EWzJhvMathVCenter{#1}}{\EWzJhvMathVCenter{#2}}}
\makeatother

\EWzJhvProtectedDefNew\EWzJhvDisplayMathPeriod#1{\makebox[0pt][l]{\raisebox{-0.2ex}[0pt][0pt]{\:#1}}}
\ifEWzJhvPDF
\EWzJhvProtectedDefNew\EWzJhvDisplayMathPeriodOverfull#1{\makebox[0pt][l]{\raisebox{-0.2ex}[0pt][0pt]{#1}}}
\else
\EWzJhvLetNew\EWzJhvDisplayMathPeriodOverfull=\EWzJhvDisplayMathPeriod
\fi

\EWzJhvProtectedDefNew\EWzJhvMathSepColon{\mskip2mu\relax:\mskip2mu\relax}

\EWzJhvProtectedDefNew\EWzJhvMathTimes{\mathrm{\mathchar45\relax times}\mskip4mu\relax}

\ifEWzJhvPDF
\EWzJhvProtectedDefNew\EWzJhvTpStar#1#2{\makebox[0pt][l]{$\displaystyle#1_{#2}$}\makebox[0pt][l]{$\displaystyle\phantom{#1}^*$}\phantom{\displaystyle#1_{#2}}}
\else
\EWzJhvProtectedDefNew\EWzJhvTpStar#1#2{{\displaystyle#1_{#2}^*}}
\fi

\makeatletter
\EWzJhvProtectedDefNew\EWzJhvMathRaiseABit@#1#2{\raisebox{0.15ex}{$#1#2$}}
\EWzJhvProtectedDefNew\EWzJhvMathRaiseABit{\mathpalette\EWzJhvMathRaiseABit@}
\EWzJhvProtectedDefNew\EWzJhvlvertRaised{\mathopen{\EWzJhvMathRaiseABit\lvert}}
\EWzJhvProtectedDefNew\EWzJhvrvertRaised{\mathclose{\EWzJhvMathRaiseABit\rvert}}
\makeatother

\EWzJhvProtectedDefNew\EWzJhvCDotInd{{\,\cdot\,}}

\EWzJhvProtectedDefNew\EWzJhvPar{\par}

\EWzJhvProtectedDefNew\EWzJhvEqRef#1{\EWzJhvTextNormal{\hyperref[{#1}]{(\ref*{#1})}}}

\ifEWzJhvPDF
\EWzJhvProtectedDefNew\EWzJhvSlantedRightarrow{\mbox{\pdfsave\pdfsetmatrix{1 0 0.2679492 1}\makebox[0pt][l]{$\Rightarrow$}\pdfrestore\phantom{$\Rightarrow$}}}
\else
\EWzJhvProtectedDefNew\EWzJhvSlantedRightarrow{\mbox{\normalfont\slshape ⇒}}
\fi

\EWzJhvProtectedDefNew\EWzJhvOptionallyFixedWidth#1#2{#2}
\EWzJhvProtectedDefNew\EWzJhvOptionallyFixedWidthEnabled#1#2{\makebox[#1]{#2}}

\makeatletter
\EWzJhvProtectedDefNew\EWzJhv@sw@slant{\leavevmode\begingroup
	\ifdim\lastskip=\z@
		\EWzJhv@fix@penalty
	\else
		\skip@\lastskip
		\unskip
		\EWzJhv@fix@penalty
		\hskip\skip@
	\fi
	\endgroup}

\EWzJhvProtectedDefNew\EWzJhv@fix@penalty{%
	\ifnum\lastpenalty=\z@
		\pdfprimitive\/%
	\else
		\count@\lastpenalty
		\unpenalty
		\pdfprimitive\/%
		\penalty\count@
	\fi
}
\makeatother

\makeatletter
\EWzJhvProtectedLongDefNew\EWzJhvFootnote{\EWzJhv@sw@slant\footnote}

\EWzJhvProtectedLongDefNew\EWzJhvPrefixFootnote#1{\EWzJhv@sw@slant\begingroup
	\EWzJhvLetNew\EWzJhvPrefixFootnote@Original@makefnmark=\@makefnmark
	\EWzJhvDefNew\EWzJhvPrefixFootnote@@makefnmark{\hbox{\hspace*{\scriptspace}\scriptspace=0pt\relax\textsuperscript{\normalfont\@thefnmark}}}%
	\def\@makefnmark{\EWzJhvPrefixFootnote@@makefnmark}%
	\footnotemark
	\let\EWzJhvPrefixFootnote@@makefnmark=\EWzJhvPrefixFootnote@Original@makefnmark
	\footnotetext{#1}%
	\endgroup}
\EWzJhvProtectedLongDefNew\EWzJhvInfixFootnote#1{\EWzJhv@sw@slant\begingroup
	\EWzJhvLetNew\EWzJhvInfixFootnote@Original@makefnmark=\@makefnmark
	\EWzJhvDefNew\EWzJhvInfixFootnote@@makefnmark{\hbox{\scriptspace=0pt\relax\textsuperscript{\normalfont\@thefnmark}}}%
	\def\@makefnmark{\EWzJhvInfixFootnote@@makefnmark}%
	\footnotemark
	\let\EWzJhvInfixFootnote@@makefnmark=\EWzJhvInfixFootnote@Original@makefnmark
	\footnotetext{#1}%
	\endgroup}
\makeatother

\makeatletter
\newcount\EWzJhvFootRef@spacefactor
\EWzJhvProtectedDefNew\EWzJhvFootRef#1{\EWzJhv@sw@slant\begingroup
	\ifhmode
		\EWzJhvFootRef@spacefactor\spacefactor
		\nobreak
	\fi
	\hyperref[{#1}]{\mbox{\textsuperscript{\normalfont\ref*{#1}}}}%
	\ifhmode\spacefactor\EWzJhvFootRef@spacefactor\fi
	\endgroup}
\EWzJhvProtectedDefNew\EWzJhvPrefixFootRef#1{\EWzJhv@sw@slant\begingroup
	\ifhmode
		\EWzJhvFootRef@spacefactor\spacefactor
		\nobreak
	\fi
	\hyperref[{#1}]{\mbox{\hspace*{\scriptspace}\scriptspace=0pt\relax\textsuperscript{\normalfont\ref*{#1}}}}%
	\ifhmode\spacefactor\EWzJhvFootRef@spacefactor\fi
	\endgroup}
\makeatother

\makeatletter
\EWzJhvProtectedDefNew\EWzJhvLimItCorr#1{\leavevmode\begingroup
	\dimen@=#1\relax
	\ifdim\lastkern=\z@
		\pdfprimitive\/%
		\ifdim\lastkern>\dimen@
			\unkern
			\kern\dimen@
		\fi
	\fi
	\endgroup}
\makeatother

\hypersetup{
	pdftitle={Recovering complex structure from (𝑛,0)-form with implications for Vafa--Witten equation},
	pdfauthor={Zhengxiong Cao}}

\begin{document}
	\title{Recovering complex structure from $(n,0)$-form with implications for Vafa--Witten equation}
	\author{Zhengxiong Cao\thanks{Institut für Mathematik, Humboldt-Universität zu Berlin, Berlin, Germany}}
	\date{}
	\maketitle
	
	\begin{abstract}
		Given an even-dimensional smooth manifold and a (co)(closed) bundle-valued $n$-form $a$ that ``resembles'' a scalar $(n,0)$-form, we may (or may not) recover a ``$\mathbb{Z}_2$-(almost) complex structure'' under certain assumptions. In the case of a $4$-manifold or a quaternion-Kähler manifold, we can recover a compatible $\mathbb{Z}_2$-complex structure on the whole manifold including the zero locus of $a$. In general, we can recover a $\mathbb{Z}_2$-(almost) complex structure wherever $a$ satisfies a certain nondegeneracy condition, but not necessarily on the whole manifold, with a counterexample identified.
		
		The above condition of ``resembling'' a scalar $(2,0)$-form is satisfied by any $\mathrm{U}(1)$-invariant solution to the Vafa--Witten equation on a compact Kähler surface. This hints at a possibility to generalise the notion of $\mathrm{U}(1)$-invariant solutions to oriented Riemannian $4$-manifolds. However, our results above imply any solution satisfying the condition, and thus seemingly any potential ``generalised $\mathrm{U}(1)$-invariant solution'', must be trivial if the $4$-manifold admits no compatible $\mathbb{Z}_2$-complex structure.
	\end{abstract}
	
	\tableofcontents
	
	\ifEWzJhvPageBreak\pagebreak\fi
	
	\section{Introduction}
	
	This article discusses, roughly speaking, given an even-dimensional smooth manifold $M$, a vector bundle $V \rightarrow M$, and a (co)(closed) $V$\hspace*{-0.03em}-valued $n$-form $a$ that pointwise ``resembles'' a scalar $(n,0)$-form, how we may (or may not) (locally) recover an (almost) complex structure on $M$, possibly under suitable additional assumptions. This article also discusses what this might imply about the Vafa--Witten equation.
	
	\EWzJhvSubsectionStarWithSecondLine{Four dimensional case, or quaternion-Kähler case (\hspace*{-0.08em}``$\mathrm{Hol}(g) \subset \mathrm{Sp}(N)\hspace{0.0625em}\mathrm{Sp}(1)$''\hspace*{-0.08em})}
	{\noindent\hspace*{0.5\parindent}{\small\slshape (\hspace*{0.07em}\EWzJhvSlantedRightarrow\hspace{0.2em}\autoref{secSpNSp1})}}
	
	Let $(M,g)$ be a smooth oriented Riemannian $4$-manifold (or a quaternion-Kähler manifold\EWzJhvFootnote{For our purposes, the setting of quaternion-Kähler manifolds appears as a natural generalisation of the setting of smooth oriented Riemannian $4$-manifolds. Originally, the author developed the main result with only smooth oriented Riemannian $4$-manifolds in mind. It was then noticed that it generalises straightforwardly to quaternion-Kähler manifolds. It should be noted that it is not clear what implications the generalised result may have for quaternion-Kähler manifolds.}). Let $V \rightarrow M$ be a smooth real vector bundle with inner product, endowed with a connection $A$. Let $a \in \Omega^{2,+}(V)$ be a $V$\hspace*{-0.03em}-valued self-dual $2$-form (or in higher dimensions, ``coefficient'' $2$-form; note that $(2,0)$-forms are self-dual in real dimension~$4$) that pointwise ``resembles'' a scalar $(2,0)$-form in the following sense:\vspace{-0.4ex plus -0.1ex minus -0.3ex}
	\begin{EWzJhvDefinitionUnnumbered}
		$a \in \Omega^{2,+}(V)$ is said to be a \emph{parascalar $(2,0)$-form}, if at each\/~$p \in M$, $a\rvert_p = \lambda\hspace{0.0625em}\bigl(v_1 \otimes \omega^1 + v_2 \otimes \omega^2\bigr)$ where\/ $\lambda$ is some (nonnegative) real number, $(v_1\,\,v_2)$ are some orthonormal vectors in\/ $V\rvert_p$, and\/ $(\omega^1\,\,\omega^2)$ are some orthonormal vectors in\/ $\Omega^{2,+}\rvert_p$.
		
		\noindent\hspace*{0.5\parindent}{\small\slshape (\hspace*{0.07em}\EWzJhvSlantedRightarrow\hspace{0.17em}\autoref{defn2plusParascalar})}
		
		{\small\slshape\noindent\hspace*{0.5\parindent}(This definition is related to ``\/$\mathrm{U}(1)$-invariant'' or ``nilpotent'' solutions to the Vafa--Witten equation. We will explain this in a moment.)\par}\vspace{-0ex plus -0.5ex}
	\end{EWzJhvDefinitionUnnumbered}
	
	We shall easily see that $a$ induces a so-called ``$\mathbb{Z}_2$-almost complex structure'' on the complement of the zero locus of $a$, compatible with the metric and the orientation (or with the quaternion-Kähler structure) {\small (see the description following \autoref{defn2plusParascalar})}, and if $a$ is coclosed, then this $\mathbb{Z}_2$-almost complex structure is integ\EWzJhvText{rab}le {\small (\autoref{lmaJaIntegrability})}.
	
	\medskip
	
	Our main result is:\vspace{-0.8ex plus -0.2ex minus -0.4ex}
	\begin{EWzJhvTheoremUnnumbered}
		\hspace{-0.333333em}(informal) If parascalar\/ $(2,0)$-form\/ $a$ is coclosed and is not trivially zero, then its induced\/ $\mathbb{Z}_2$-almost complex structure can be smoothly extended to an integrable\/ $\mathbb{Z}_2$-(almost) complex structure over the whole manifold including the zero locus of\/\hspace{-0.08em} $a$.
		
		\noindent\hspace*{0.5\parindent}{\small\slshape (\hspace*{0.07em}\EWzJhvSlantedRightarrow\hspace{0.17em}\autoref{thm2})}\vspace{-0ex plus -0.5ex}
	\end{EWzJhvTheoremUnnumbered}
	
	The proof of this smooth extension, including all those lemmas used, shall take up a large number of pages. To prove this smooth extension, we shall first do some pointwise calculation involving derivatives of arbitrarily high order\EWzJhvFootnote{The exact order of derivatives needed pro\EWzJhvText{bab}ly depends on the dimension of $M$ and the maximal order of the zeros of $a$ (basically the least $m$ for which $\sum_{j=0}^m{}\bigl|(\nabla^A)^j a\bigr| >0$). A curious effect of relying on derivatives of arbitrarily high order is that, seemingly, our method would fail if things were only assumed to be, say, $C^{100}$, instead of being $C^\infty$, for $a$ that possibly has zeros of order higher than, say, $1000$.} to obtain a bound on the derivative of this $\mathbb{Z}_2$-complex structure {\small (\autoref{corHigherDer})}. After some further technical work including working around a subtle problem with \cite{ZeroSetsSolutionsElliptic} {\small (\autoref{subsecA20ZeroLocus})}, we shall prove that this $\mathbb{Z}_2$-complex structure extends to the whole manifold in $W^{1,p}_{\mathrm{loc}}$ {\small (\autoref{lmaJaW1p})}. With the help of an elliptic PDE satisfied by such $\mathbb{Z}_2$-complex structure compatible with the metric (or the quaternion-Kähler structure) {\small (\autoref{lmaPdeJa})}, we shall finally prove the smooth extension. For details, see \autoref{secSpNSp1}.
	
	Note that a general smooth connected oriented Riemannian $4$-manifold (or quaternion-Kähler manifold) may admit no compatible $\mathbb{Z}_2$-complex structure at all, in which case our result simply implies that any such coclosed $a$ must be zero. In particular, it is known that a compact quaternion-Kähler manifold of dimension at least~$8$ that admits a compatible complex structure must be finitely covered by a hyper\EWzJhvText{käh}ler manifold, and any such compatible complex structure must be Kähler (\cite{CplxStrsOnQuarternionicManifolds}, see also \cite{CompatibleAlmostCplxStrsOnQKManifolds}).
	
	\EWzJhvSubsectionStarWithSecondLine{Case of general even-dimensional manifolds}
	{\noindent\hspace*{0.5\parindent}{\small\slshape (\hspace*{0.07em}\EWzJhvSlantedRightarrow\hspace{0.2em}\autoref{secGeneral})}}
	
	On general even-dimensional smooth manifolds, we shall observe similarly that a so-called ``parascalar $(n,0)$-form'' {\small (\autoref{defnParascalarN})} induces a $\mathbb{Z}_2$-almost complex structure wherever the parascalar $(n,0)$-form satisfies a certain nondegeneracy condition \ifEWzJhvPDF\linebreak\fi {\small (\autoref{propParascalarNInducesCplxStr})}, and this $\mathbb{Z}_2$-almost complex structure is integ\EWzJhvText{rab}le if the parascalar $(n,0)$-form is closed {\small (\autoref{propParascalarNClosedToIntegrab})}. In particular, these results hold for a form that pointwise ``resembles'' a complex volume form {\small (\autoref{corParaVolumeForm})}. However, in this case, smooth extension of the integ\EWzJhvText{rab}le $\mathbb{Z}_2$-(almost) complex structure does not necessarily hold, as we can easily give a counterexample in complex dimension~$4$ {\small (\autoref{subsecCounterexampleSmoothExtensib})}. Our understanding of this case is still rudimentary. For details, see \autoref{secGeneral}.
	
	\EWzJhvSubsectionStarWithSecondLine{Motivation, Implications for Vafa--Witten equation}
	{\noindent\hspace*{0.5\parindent}{\small\slshape (\hspace*{0.07em}\EWzJhvSlantedRightarrow\hspace{0.2em}\autoref{secVW})}}
	
	The author's motivation to consider the question of recovering a $\mathbb{Z}_2$-(almost) complex structure from a parascalar $(n,0)$-form comes from the study of the Vafa--Witten equation. Let $(M,g)$ be a smooth oriented Riemannian $4$-manifold. Let $G$ be a compact Lie group, $\mathrm{SU}(2)$ and $\mathrm{SO}(3)$ being two simple cases. Let $P \rightarrow M$ be a $G$-principal bundle. Denote the space of principal connections on $P$ by $\mathscr{A}(P)$. Denote the space of self-dual \ifEWzJhvPDF\linebreak\fi $2$-forms valued in the adjoint bundle by $\Omega^{2,+}(\mathrm{ad}(P))$. The Vafa--Witten equation \cite{VAFA19943}, following \ifEWzJhvPDF\hspace{0pt minus 1.3pt}\fi \cite{MaresPhD}, \ifEWzJhvPDF\hspace{0pt minus 0.9pt}\fi is the following equation system on a pair\EWzJhvFootnote{We ignore the seemingly less significant $\Gamma(\mathrm{ad}(P))$ component of a solution, which vanishes in many interesting cases anyway \cite[Remark~2.1.2]{MaresPhD}.} $(A,a) \ifEWzJhvPDF\hspace*{0pt minus 2pt}\fi\in\ifEWzJhvPDF\hspace*{0pt minus 2pt}\fi \mathscr{A}(P) \times \Omega^{2,+}(\mathrm{ad}(P))$
	\begin{align*}
		\mathrm{d}_A^* a &= 0\\
		F_A^+ + \EWzJhvMathFracVCentered{1}{8}\hspace{0.0625em}[a \centerdot a] &= 0
	\end{align*}
	where $F_A^+ \in \Omega^{2,+}(\mathrm{ad}(P))$ is the self-dual part of the curvature of $A$, $[\EWzJhvCDotInd\centerdot\EWzJhvCDotInd]$ is a certain antisymmetric linear map $\mskip2mu\relax \EWzJhvCDotInd\centerdot\EWzJhvCDotInd \mskip2mu\relax\EWzJhvMathSepColon\mskip2mu\relax \Omega^{2,+}\rvert_p \otimes \Omega^{2,+}\rvert_p \mskip2mu\relax\rightarrow\mskip2mu\relax \Omega^{2,+}\rvert_p \mskip2mu\relax$ tensored with the Lie bracket on the adjoint bundle.
	
	In case $(M,g)$ is a compact Kähler surface, it is known that $a \in \Omega^{2,+}(\mathrm{ad}(P))$ splits into $a = \beta - \beta^* + \gamma \otimes_{\mathbb{R}} \omega$, where $\beta \in \Omega^{2,0}(\mathrm{ad}(P) \otimes_{\mathbb{R}} \mathbb{C})$, $\gamma \in \Gamma(\mathrm{ad}(P))$, $\omega$ is the Kähler form, and the Vafa--Witten equation can be rewritten in terms of $(A,\beta,\gamma)$ \cite[\begingroup\small Theorem~7.1.2\endgroup]{MaresPhD} {\small (\autoref{propVWequivalenceKaehler})}. In many interesting cases $\gamma = 0$, for example if $A$ is an irreducible $\mathrm{SU}(2)$-connection. The space of solutions $(A,\beta,0)$ modulo the gauge group action admits a $\mathrm{U}(1)$-action\EWzJhvFootnote{If one considers the other side of the Kobayashi--Hitchin correspondence, it is actually a $\mathbb{C}^*$-action. More precisely, it is actually this $\mathbb{C}^*$-action that was considered in \cite{Tanaka_2019}\cite{Tanaka_2017}\cite{chen2024VWeqKaehler}.}: $\beta \mapsto e^{i\theta}\beta$, whose invariant locus is used to define Vafa--Witten invariants for projective surfaces \cite{Tanaka_2019}\cite{Tanaka_2017}. Previously, it has been observed\EWzJhvFootnote{These results from \cite{chen2024VWeqKaehler} were actually proved for the Vafa--Witten equations over compact Kähler manifolds of arbitrary dimension, not just compact Kähler surfaces. Our \autoref{propU1invToParascalar},~\ref{propNilpotentParascalar} hold for almost complex manifolds of arbitrary (higher) dimension. On the other hand, as noted above, our result of smooth extension does not seem to generalise to general even-dimensional smooth manifolds.} that for~$G = \mathrm{U}(n)$, if $(A,\beta,0)$ is a $\mathrm{U}(1)$-invariant solution, then $\beta$ must be pointwise nilpotent \cite[\begingroup\small Lemma~2.12\endgroup]{chen2024VWeqKaehler}, and solutions $(A,\beta,0)$ where $\beta$ is pointwise nilpotent\EWzJhvFootnote{or more generally, solutions with uniformly bounded spectral covers} satisfy a $C^0$ a priori estimate and, as a consequence, a version of Uhlenbeck compactness \cite[\begingroup\small section~5\endgroup]{chen2024VWeqKaehler}.
	
	\medskip
	
	We shall observe:\vspace{-0.8ex plus -0.2ex minus -0.4ex}
	\begin{EWzJhvPropositionUnnumbered}
		\hspace{-0.333333em}(informal) If\/\hspace{-0.08em} $(A,\beta,0)$ is a\/ $\mathrm{U}(1)$-invariant solution, then\/ $a = \beta - \beta^*$ must be a parascalar\/ $(2,0)$-form in the sense defined above.
		
		\noindent\hspace*{0.5\parindent}{\small\slshape (\hspace*{0.07em}\EWzJhvSlantedRightarrow\hspace{0.17em}\autoref{propU1invToParascalar})}\vspace{-0.5ex plus -0.1ex minus -0.3ex}
	\end{EWzJhvPropositionUnnumbered}
	\begin{EWzJhvPropositionUnnumbered}
		\hspace{-0.333333em}(informal) For\/~$\mathfrak{g} = \mathfrak{u}(n)$, if\/\hspace{-0.08em} $\beta$ is pointwise nilpotent, then\/ $a = \beta - \beta^*$ must be a parascalar\/ $(2,0)$-form. For\/~$\mathfrak{g} = \mathfrak{su}(2)$, the converse also holds.
		
		\noindent\hspace*{0.5\parindent}{\small\slshape (\hspace*{0.07em}\EWzJhvSlantedRightarrow\hspace{0.17em}\autoref{propNilpotentParascalar})}\vspace{-0ex plus -0.5ex}
	\end{EWzJhvPropositionUnnumbered}
	
	In addition, for~$\mathfrak{g} = \mathfrak{su}(2)$, on a compact smooth oriented Riemannian $4$-manifold, solutions in the form of $(A,a)$ where $a$ is a parascalar $(2,0)$-form also satisfy a $C^0$ a priori estimate {\small (\autoref{propParascalarC0Estimate})}, and, as a consequence, a version of Uhlenbeck compactness.
	
	\medskip
	
	This is a close relation between $\mathrm{U}(1)$-invariant (or nilpotent) solutions and the notion \ifEWzJhvPDF\linebreak\fi of parascalar $(2,0)$-forms. Note that the condition of being a parascalar $(2,0)$-form \ifEWzJhvPDF\nolinebreak\fi makes sense on general smooth $4$-manifolds with no Kähler structure specified. This hints at a possibility to generalise the notion of $\mathrm{U}(1)$-invariant solutions from the Kähler setting \ifEWzJhvPDF\linebreak\fi to general smooth oriented Riemannian $4$-manifolds. Here, the study of parascalar $(2,0)$-forms is motivated. The study, however, turns out to indicate that for a nontrivial parascalar $(2,0)$-form forming part of a solution to exist, it is necessary that the smooth oriented Riemannian $4$-manifold admits a compatible $\mathbb{Z}_2$-complex structure. Therefore, it rather seems that the notion of $\mathrm{U}(1)$-invariant solutions of the Kähler setting cannot be generalised to very general smooth oriented Riemannian $4$-manifolds in this seemingly quite natural manner. Some further questions may be asked. For details, see \autoref{secVW}.
	
	\section{\texorpdfstring{$\mathbb{Z}_2$}{ℤ₂}-complex structure}
	
	In this short section, we give the definition of a $\mathbb{Z}_2$-(almost) complex structure.
	
	\medskip
	
	Let $M$ be an even-dimensional smooth manifold.
	
	Let $\mathcal{J}(M)\rvert_{p \in M} := \{\,J_p \in \operatorname{End}(T_p M) \mid J_p^2 = -\mathrm{id}\,\}$ be the fibre bundle whose (local) sections are the (local) almost complex structures. $\{1, -1\}$ acts on $\mathcal{J}(M)$ fibrewise by scalar multiplication. Denote the quotient fibre bundle by $\widetilde{\mathcal{J}}(M)$.
	
	\begin{EWzJhvDefinition}
		A \emph{$\mathbb{Z}_2$-almost complex structure} $\tilde{J}$ on\/ $M$ is a smooth section of\/\hspace{-0.08em} $\widetilde{\mathcal{J}}(M)$. $\tilde{J}$ is said to be integrable, if locally\/ $\tilde{J}$ comes from an integrable almost complex structure, in which case\/ $\tilde{J}$ is also called a \emph{$\mathbb{Z}_2$-complex structure}. $(M, \tilde{J})$ is called a \emph{$\mathbb{Z}_2$-(almost) complex manifold}.
		
		In case\/ $M$ is endowed with a Riemannian metric, $\tilde{J}$ is said to be compatible with the metric (alias ``orthogonal with respect to the metric''), if locally\/ $\tilde{J}$ comes from an almost complex structure compatible with the metric.
		
		In case\/ $M$ is oriented and\/ $\dim_{\mathbb{R}} M$ is a multiple of\/\hspace*{0em plus -0.08em}~$4$, $\tilde{J}$ is said to be compatible with the orientation, if locally\/ $\tilde{J}$ comes from an almost complex structure compatible with the orientation.
	\end{EWzJhvDefinition}
	
	\paragraph*{Constructing $\mathbb{Z}_2$-complex manifold} Here we describe a way to construct a $\mathbb{Z}_2$-complex manifold that is not a complex manifold, in the sense that its $\mathbb{Z}_2$-complex structure does not \emph{globally} come from a complex structure. Take a homogeneous polynomial with real coefficients (or more generally, not just one polynomial, but an ideal). Denote the polynomial by $f$. The complex conjugation is an antiholomorphic involution on the zero locus of $f$ in $\mathbb{C}\mathrm{P}^N$. If the zero locus of $f$ in $\mathbb{C}\mathrm{P}^N$ is smooth, and the zero locus of $f$ in $\mathbb{R}\mathrm{P}^N$ is empty (e.g.\ $f(x_0,\ldots,x_N) = x_0^{2k} + \cdots + x_N^{2k}$), then the quotient of the zero locus of $f$ in $\mathbb{C}\mathrm{P}^N$ by the complex conjugation is a $\mathbb{Z}_2$-complex manifold. \ifEWzJhvPDF\linebreak\fi If furthermore, the zero locus of $f$ in $\mathbb{C}\mathrm{P}^N$ is connected, then the quotient is not a complex manifold.
	
	\section{Case of \texorpdfstring{$\mathrm{Hol}(g) \subset \mathrm{Sp}(N)\hspace{0.0625em}\mathrm{Sp}(1)$}{Hol(𝑔)\unichar{"2009}⊂\unichar{"2009}Sp(𝑁)\unichar{"200A}Sp(1)}}\label{secSpNSp1}
	
	\subsection{Introduction}
	
	\subsubsection*{Preliminaries: The setting of $\mathrm{Hol}(g) \subset \mathrm{Sp}(N)\hspace{0.0625em}\mathrm{Sp}(1)$}
	
	For the purpose of this section, the setting of $\mathrm{Hol}(g) \subset \mathrm{Sp}(N)\hspace{0.0625em}\mathrm{Sp}(1)$ appears as a natural generalisation\EWzJhvFootnote{Originally, the author developed the main results of this section with only smooth oriented Riemannian $4$-manifolds in mind, but it was then noticed that these results generalise straightforwardly to the setting of $\mathrm{Hol}(g) \subset \mathrm{Sp}(N)\hspace{0.0625em}\mathrm{Sp}(1)$, so they shall be presented in this general setting. It should be noted that the setting of $\mathrm{Hol}(g) \subset \mathrm{Sp}(N)\hspace{0.0625em}\mathrm{Sp}(1)$ with $N \geqslant 2$ is a kind of \emph{special holonomy} and, for general considerations, is of rather different nature than that of smooth oriented Riemannian $4$-manifolds. It should also be noted that it is not clear what implications the generalised results may have for $N \geqslant 2$.} of the setting of smooth oriented Riemannian $4$-manifolds. The setting is described as below:
	
	\smallskip
	
	Let $(M,g)$ be a smooth $4N$\hspace*{-0.03em}-dimensional Riemannian manifold\EWzJhvFootnote{$(M,g)$ is not assumed to be complete, so the discussions can work locally.} ($N \geqslant 1$). Let $Q$ be a rank-$3$ subbundle of $\operatorname{Hom}(TM,TM)$ satisfying the following conditions:\vspace{-\parskip}
	\begin{enumerate}[topsep=1ex plus 0.25ex minus 0.25ex, itemsep=1ex plus 0.25ex minus 0.25ex, parsep=0pt]
		\item\label{eiQCond1} For any~$p \in M$, there exists an orthogonal basis $(J_1\,\,J_2\,\,J_3)$ of the $3$-dimensional vector space $Q\rvert_p \subset \operatorname{Hom}(T_p M,T_p M)$, such that $J_1,J_2,J_3$ are isometries on $T_p M$, and moreover $J_1^2=J_2^2=J_3^2=-\mathrm{id}$, $J_1J_2=-J_2J_1=J_3$, $J_2J_3=-J_3J_2=J_1$, $J_3J_1=-J_1J_3=J_2$.
		\item\label{eiQCond2} For any section $q \in \Gamma(Q)$ and vector field $X$, the covariant derivative $\nabla_{\!X} q$, which a priori is a section of $\operatorname{Hom}(TM,TM)$, is actually a section of\/ $Q$.
	\end{enumerate}
	
	If such $Q$ exists, $(M,g)$ or $(M,g,Q)$ is what we call a smooth $4N$\hspace*{-0.03em}-dimensional \ifEWzJhvPDF\linebreak\fi Riemannian manifold with holonomy $\mathrm{Hol}(g) \ifEWzJhvPDF\hspace{0em minus 0.2em}\fi\subset\ifEWzJhvPDF\hspace{0em minus 0.2em}\fi \mathrm{Sp}(N)\hspace{0.0625em}\mathrm{Sp}(1)$, and $Q$ is called a ``coefficient \ifEWzJhvPDF\linebreak\fi bundle''. This includes any smooth oriented Riemannian $4$-manifold $(M,g)$, in which case $Q$ is locally spanned by local almost complex structures compatible with $g$ and the orientation.\EWzJhvFootnote{Note that $\mathrm{Sp}(N)\hspace{0.0625em}\mathrm{Sp}(1)$ coincides with $\mathrm{SO}(4N)$ when $N=1$.} In general, $(M,g)$ or $(M,g,Q)$ is often said to be quaternion-Kähler\EWzJhvInfixFootnote{Note that in general, a quaternion-Kähler manifold is not necessarily Kähler. In addition, note that one often imposes some extra condition on $(M,g,Q)$ in the definition of a quaternion-Kähler manifold when $N=1$.}. For references, see e.g.\ \cite{QKManifoldsIshihara, QKManifoldsSalamon}.
	
	\bigskip
	
	\EWzJhvParagraphNoBeforeskip*{The bundle\hspace{0.5em plus 0.166667em minus 0.066667em}$\Omega^{2,+} \subset \Omega^2$.} For $(M,g,Q)$ with $\mathrm{Hol}(g) \subset \mathrm{Sp}(N)\hspace{0.0625em}\mathrm{Sp}(1)$, let $\Omega^{2,+} \subset \Omega^2$ be the rank-$3$ subbundle of\/ $2$-forms whose sections are given by $g(q\EWzJhvCDotInd,\EWzJhvCDotInd)$ where $q \in \Gamma(Q)$. For a smooth oriented Riemannian $4$-manifold $(M,g)$, this is just the bundle of self-dual $2$-forms\EWzJhvInfixFootnote{i.e.\ the eigenvectors with eigenvalue~$1$ of the Hodge star operator $*$ on $\Omega^2$}. As any rank-$1$ subbundle of\/ $Q$ gives a compatible\EWzJhvFootnote{A ``compatible'' ($\mathbb{Z}_2$-)(almost) complex structure on $(M,g,Q)$ with $\mathrm{Hol}(g) \subset \mathrm{Sp}(N)\hspace{0.0625em}\mathrm{Sp}(1)$ is so defined. See \autoref{subsecCplxStr} for more on compatible complex structures.} $\mathbb{Z}_2$-almost complex structure, any rank-$1$ subbundle of\/ $\Omega^{2,+}$ also gives a compatible $\mathbb{Z}_2$-almost \ifEWzJhvPageBreak\pagebreak\fi complex structure. For any~$p \in M$ and any oriented\EWzJhvFootnote{The orientation of\/ $\Omega^{2,+}$ is defined to be the one that allows \EWzJhvEqRef{eqOrthonBasis2plus}.} orthonormal basis $(\omega^1\,\,\omega^2\,\,\omega^3)$ of\/ $\Omega^{2,+}\rvert_p$, we may choose an oriented\EWzJhvFootnote{The orientation of an $(M,g,Q)$ with $\mathrm{Hol}(g) \subset \mathrm{Sp}(N)\hspace{0.0625em}\mathrm{Sp}(1)$ is so defined.} orthonormal basis $(e^{j,k})_{1 \leqslant j \leqslant N,1 \leqslant k \leqslant 4}$ of\/ $\EWzJhvTpStar Tp M$ and express $(\omega^1\,\,\omega^2\,\,\omega^3)$ explicitly as follows:\EWzJhvFootnote{\label{ftnInnPdExtPdTenPd}Note how our inner product for exterior products differs from the one for tensor products, i.e.\ $\bigl| e^{1,1} \otimes e^{1,2} - e^{1,2} \otimes e^{1,1}\bigr|=\sqrt{2}$, but $\bigl| e^{1,1} \wedge e^{1,2}\bigr|=1$, etc. This is different from those conventions used by e.g.\ \cite{CompatibleAlmostCplxStrsOnQKManifolds}.}
	\begin{equation}\label{eqOrthonBasis2plus}
		\begin{gathered}
			\omega^1=\EWzJhvMathFracVCentered{1}{\sqrt{2N}\,}\sum_{j=1}^N\bigl(e^{j,2} \wedge e^{j,3} + e^{j,1} \wedge e^{j,4}\bigr)\\
			\omega^2=\EWzJhvMathFracVCentered{1}{\sqrt{2N}\,}\sum_{j=1}^N\bigl(e^{j,3} \wedge e^{j,1} + e^{j,2} \wedge e^{j,4}\bigr)\\
			\omega^3=\EWzJhvMathFracVCentered{1}{\sqrt{2N}\,}\sum_{j=1}^N\bigl(e^{j,1} \wedge e^{j,2} + e^{j,3} \wedge e^{j,4}\bigr)
		\end{gathered}
	\end{equation}
	
	\subsubsection*{Main results}
	
	Let $(M,g,Q)$ be with $\mathrm{Hol}(g) \subset \mathrm{Sp}(N)\hspace{0.0625em}\mathrm{Sp}(1)$. Let $V \rightarrow M$ be a smooth real vector bundle with inner product. Let $A$ be a connection on $V$ that is compatible with the inner product on $V$. Let $\Omega^{2,+}(V) := V \otimes \Omega^{2,+}$.
	
	\begin{EWzJhvDefinition}\label{defn2plusParascalar}
		$a \in \Omega^{2,+}(V)$ is said to be a \emph{parascalar $(2,0)$-form}, if at each\/~$p \in M$, $a\rvert_p = \lambda\hspace{0.0625em}\bigl(v_1 \otimes \omega^1 + v_2 \otimes \omega^2\bigr)$ where\/ $\lambda$ is some (nonnegative) real number, $(v_1\,\,v_2)$ are some orthonormal vectors in\/ $V\rvert_p$, and\/ $(\omega^1\,\,\omega^2)$ are some orthonormal vectors in\/ $\Omega^{2,+}\rvert_p$.
	\end{EWzJhvDefinition}
	
	\noindent\emph{Induced\/ $\mathbb{Z}_2$-almost complex structure:}\enskip Let $a \in \Omega^{2,+}(V)$ be a parascalar $(2,0)$-form. For~$p \in M$ with $a\rvert_p \neq 0$, write $a\rvert_p = \lambda\hspace{0.0625em}\bigl(v_1 \otimes \omega^1 + v_2 \otimes \omega^2\bigr)$ where $\lambda$ is some positive real number, $(v_1\,\,v_2)$ are some orthonormal vectors in $V\rvert_p$, and $(\omega^1\,\,\omega^2)$ are some orthonormal vectors in $\Omega^{2,+}\rvert_p$. The orthogonal complement of\/ $\operatorname{span}_{\mathbb{R}}\{\omega^1,\omega^2\}$ in $\Omega^{2,+}\rvert_p$ defines a compatible $\mathbb{Z}_2$-almost complex structure on the complement of the zero locus of $a$. This $\mathbb{Z}_2$-almost complex structure is said to be the one \emph{induced} by $a$. We denote this compatible $\mathbb{Z}_2$-almost complex structure by $\tilde{J}_a$.
	
	\medskip
	
	The main result of this section is
	
	\begin{EWzJhvTheorem}\label{thm2}
		Let\/ $(M,g,Q)$ be a smooth\/ $4N$\hspace*{-0.08em}-dimensional Riemannian manifold with holonomy\/ $\mathrm{Hol}(g) \subset \mathrm{Sp}(N)\hspace{0.0625em}\mathrm{Sp}(1)$. Let\/ $V \rightarrow M$ be a smooth real vector bundle with inner product. Let\/ $A$ be a connection on\/ $V$ that is compatible with the inner product on\/ $V$. \ifEWzJhvPDF\linebreak\fi Let\/ $a \in \Omega^{2,+}(V)$ be a parascalar\/ $(2,0)$-form. Suppose\/ $\displaystyle \mathrm{d}_A^* a = 0$ and\/ $a$ is not the zero section on any connected component of\/\hspace{-0.08em} $M$, then the\/ $\mathbb{Z}_2$-almost complex structure\/ $\tilde{J}_a$ defined on the complement of the zero locus of\/\hspace{-0.08em} $a$ can be uniquely extended to an integrable\/ $\mathbb{Z}_2$-(almost) complex structure on\/ $M$ (also denoted by\/ $\tilde{J}_a$).
	\end{EWzJhvTheorem}
	
	In addition, we can also prove a continuity result:
	\begin{EWzJhvTheorem}\label{thm2Continuity}
		Let\/ $M$ be a smooth\/ $4N$\hspace*{-0.08em}-dimensional manifold. Let\/ $\{(g_j,Q_j)\}_{j\in\mathbb{N}\cup\{\infty\}}$ be a sequence where\/ $(g_j,Q_j)$ makes\/ $(M,g_j,Q_j)$ a Riemannian manifold with holonomy\/ $\mathrm{Hol}(g_j) \subset \mathrm{Sp}(N)\hspace{0.0625em}\mathrm{Sp}(1)$ for any\/~$j\in\mathbb{N}\cup\{\infty\}$. Let\/ $V \rightarrow M$ be a smooth real vector bundle with inner product. Let\/ $\{(A_j,a_j)\}_{j\in\mathbb{N}\cup\{\infty\}}$ be a sequence where\/ $A_j$ is a connection on\/ $V$ \ifEWzJhvPDF\linebreak\fi compatible with the inner product on\/ $V$ and\/ $a_j \in \Omega^{2,+}(M,g_j,Q_j;V)$ is a parascalar\/ $(2,0)$-form for any\/~$j\in\mathbb{N}\cup\{\infty\}$. Suppose\/ $\displaystyle \mathrm{d}_{A_j,g_j}^* a_j = 0$, $a_j$ is not the zero section on any connected component of\/\hspace{-0.08em} $M$ for any\/~$j\in\mathbb{N}\cup\{\infty\}$, and\/ $(g_j,A_j,a_j)$ converges\EWzJhvFootnote{It is not necessary to assume $Q_j \rightarrow Q_\infty$. In fact, given the assumptions of \autoref{thm2Continuity}, using \autoref{defn2plusParascalar} and \hyperref[eiQCond1]{the first condition satisfied by $Q$ under ``Preliminaries: \ldots'' \ifEWzJhvPDF on \autopageref*{eiQCond1}\else of this subsection\fi}, we easily see $Q_j \rightarrow Q_\infty$ in $\displaystyle C^\infty_{\mathrm{loc}}$ over the complement of the zero locus of $a_\infty$, which then leads to $Q_j \rightarrow Q_\infty$ in $\displaystyle C^\infty_{\mathrm{loc}}$ over the whole $M$ due to \hyperref[eiQCond2]{the second condition satisfied by $Q$\ifEWzJhvPDF\space on \autopageref*{eiQCond2}\fi}.} to\/ $(g_\infty,A_\infty,a_\infty)$ in\/ $\displaystyle C^\infty_{\mathrm{loc}}$ over\/ $M$ as\/ $j\rightarrow\infty$, then the\/ $\mathbb{Z}_2$-complex structure\/ $\tilde{J}_{a_j}$ also converges to\/ $\tilde{J}_{a_\infty}$ in\/ $\displaystyle C^\infty_{\mathrm{loc}}$ over\/ $M$ as\/ $j\rightarrow\infty$.
	\end{EWzJhvTheorem}
	
	We shall see a ``converse'' to \autoref{thm2} in the form of \autoref{prop2n0plusHolm} and its remark, which justifies the assumption of a parascalar $(2,0)$-form $a \in \Omega^{2,+}(V)$ with $\displaystyle \mathrm{d}_A^* a = 0$ by showing that roughly speaking, given a compatible $\mathbb{Z}_2$-complex structure, such $a$ (for trivial $V,A$) exists locally in the form of a holomorphic section of a certain holomorphic line bundle.
	
	On the other hand, due to \autoref{propDDStar2plus}, when $N \geqslant 2$, the condition $\mathrm{d}_A a = 0$ is too strong to be a suitable assumption.
	
	The conclusion of \autoref{thm2} can be divided into two parts: the integ\EWzJhvText{rab}ility of $\tilde{J}_a$ on the complement of the zero locus of $a$, and the smooth extensibility of $\tilde{J}_a$ to the whole $M$. Proving the first part is easy and straightforward {\small (\autoref{lmaJaIntegrability})}. The major difficulty is about proving the second part.
	
	Note that in general, a compatible ($\mathbb{Z}_2$-)complex structure defined on a dense open set (not in the form of $\tilde{J}_a$ for some suitable $a$) does not necessarily extend smoothly to the whole $M$. For example, for~$(M,g) = (\mathbb{R}^4,g_{\mathbb{R}^{4\mathstrut}})$, the singular set of a compatible complex structure can be a point (in which case the compatible complex structure is related to a constant one by a conformal transformation), a round circle or a straight line \cite{OrthogonalCplxStrR4}.
	
	In addition, note that a general $(M,g,Q)$ with $\mathrm{Hol}(g) \subset \mathrm{Sp}(N)\hspace{0.0625em}\mathrm{Sp}(1)$ may admit no compatible $\mathbb{Z}_2$-complex structure at all, in which case our result simply implies that no $a$ satisfies the assumptions of \autoref{thm2}. In particular, it is known that a compact quaternion-Kähler manifold of dimension at least~$8$ that admits a compatible complex structure must be finitely covered by a hyper\EWzJhvText{käh}ler manifold, and any such compatible complex structure must be Kähler (\cite{CplxStrsOnQuarternionicManifolds}, see also \cite{CompatibleAlmostCplxStrsOnQKManifolds};{\small\space also brief\EWzJhvLimItCorr{0.05em}ly discussed after \autoref{propIntegrabEq}}).
	
	\paragraph*{Organisation of section} Before actually proving \autoref{thm2},~\ref{thm2Continuity}, we shall first do some preparation in \autoref{subsecWeitzen},~\ref{subsecCplxStr}. \hyperref[subsecWeitzen]{Subsection~\ref*{subsecWeitzen}} is about the operators $\displaystyle \mathrm{d},\mathrm{d}^*$ on $\Omega^{2,+}$ and the relevant Weitzenböck formula. \hyperref[subsecCplxStr]{Subsection~\ref*{subsecCplxStr}} is about some properties of $(M,g,Q)$ endowed with a compatible complex structure, notably the ``converse'' mentioned above. We shall start with the first step of the proof of \autoref{thm2} in \autoref{subsecIntegrability}, where it is shown without much difficulty that $\tilde{J}_a$ is integ\EWzJhvText{rab}le on the complement of the zero locus of $a$. The hard part of the proof mostly resides in \autoref{subsecHigherDerivatives}, where we shall do some pointwise calculation involving derivatives of arbitrarily high order to obtain a bound on $\bigl|\nabla \tilde{J}_a\bigr|$. Before actually using this bound to prove that $\tilde{J}_a$ extends to the whole $M$ in $W^{1,p}_{\mathrm{loc}}$, we shall first do some further technical work in \autoref{subsecA20ZeroLocus},~\ref{subsecFMinus1n}. The exact purposes of \autoref{subsecA20ZeroLocus},~\ref{subsecFMinus1n} are described at the end of \autoref{subsecHigherDerivatives}. In \autoref{subsecA20ZeroLocus},\EWzJhvFootnote{Somehow \autoref{subsecA20ZeroLocus} also requires smoothness, in addition to \EWzJhvPrefixFootRef{ftnSecSpNSp1OrgReqDer}below.} we shall work near the zero locus of $a$, also working around a subtle problem with \cite{ZeroSetsSolutionsElliptic}. \hyperref[subsecFMinus1n]{Subsection~\ref*{subsecFMinus1n}} is about some rather general properties of (smooth) functions, related to the Chebyshev polynomials. In \autoref{subsecFinishUp}, we shall finally prove the $W^{1,p}_{\mathrm{loc}}$ extension\EWzJhvInfixFootnote{\label{ftnSecSpNSp1OrgReqDer}This step (\autoref{lmaJaW1p}), which uses \autoref{corHigherDer}, requires the existence of derivatives of arbitrarily high order. The exact order of derivatives needed in this step depends on $N$ and the maximal order of the zeros of $a$.}, and then with the help of an elliptic PDE \ifEWzJhvPDF\linebreak\\*[-\baselineskip]\fi satisfied by $\tilde{J}_a$ due to its compatibility and integ\EWzJhvText{rab}ility, \autoref{thm2},~\ref{thm2Continuity}.
	
	\subsection{Operators $\displaystyle \mathrm{d},\mathrm{d}^*$ on \texorpdfstring{$\Omega^{2,+}$}{Ω²ʼ⁺}, Weitzenböck formula}\label{subsecWeitzen}
	
	On $\Omega^{2,+}(V)$, the operators $\displaystyle \mathrm{d}_A,\mathrm{d}_A^*$ are remar\EWzJhvText{kab}ly related in the following manner:
	\begin{EWzJhvProposition}\label{propDDStar2plus}
		Let\/ $(M,g,Q)$ be a smooth\/ $4N$\hspace*{-0.08em}-dimensional Riemannian manifold with holonomy\/ $\mathrm{Hol}(g) \subset \mathrm{Sp}(N)\hspace{0.0625em}\mathrm{Sp}(1)$. Let\/ $V \rightarrow M$ be a smooth real vector bundle with inner product. Let\/ $A$ be a connection on\/ $V$ that is compatible with the inner product on\/ $V$. For any\/~$a,a' \in \Omega^{2,+}(V)$, we have
		\[\bigl\langle \mathrm{d}_A a, \mathrm{d}_A a' \bigr\rangle = \bigl\langle \mathrm{d}_A^* a, \mathrm{d}_A^* a' \bigr\rangle + \EWzJhvMathFracVCentered{N-1}{N}\bigl\langle \nabla^A a, \nabla^A a' \bigr\rangle \EWzJhvDisplayMathPeriod.\]
	\end{EWzJhvProposition}
	As a corollary, if $N \geqslant 2$ and $a \in \Omega^{2,+}(V)$ satisfies $\mathrm{d}_A a =0$, then $\nabla^A a=0$, which was already proved in \cite{CompatibleAlmostCplxStrsOnQKManifolds} (for~$a \in \Omega^{2,+}$).
	\begin{proof}
		Following the expressions~\EWzJhvEqRef{eqOrthonBasis2plus}, we easily compute $\bigl|e^{1,1} \wedge \omega^1\bigr|\ifEWzJhvPDF{\rule{0pt}{1.88ex}}\fi^2=\frac{2N-1}{2N}$ and also $\bigl\langle e^{1,1} \wedge \omega^1,\: e^{j,k} \wedge \omega^1 \bigr\rangle=0$ for~$(j,k)\neq(1,1)$. Therefore, for any~$\alpha,\beta \in \EWzJhvTpStar Tp M$ and any~$j\in\{1,2,3\}$, we have \ifEWzJhvPDF\vspace{1pt}\fi $\bigl\langle \alpha \wedge \omega^j,\: \beta \wedge \omega^j\bigr\rangle = \frac{2N-1}{2N}\langle \alpha, \beta \rangle = (2N-1) \bigl\langle \iota_{\alpha^\sharp} \omega^j, \iota_{\beta^\sharp} \omega^j\bigr\rangle$, where $\sharp$ is the musical isomorphism. For~$a_{j,k},b_{j,k} \in \mathbb{R}$, we compute
		\begin{align*}
			&\Bigl\langle \sum_{j=1}^N\sum_{k=1}^4 a_{j,k}\,e^{j,k} \wedge \omega^1,\: \sum_{j=1}^N\sum_{k=1}^4 b_{j,k}\,e^{j,k} \wedge \omega^2 \Bigr\rangle\\ ={}& \EWzJhvMathFracVCentered{1}{2N}\sum_{j=1}^N \bigl(a_{j,1}b_{j,2}-a_{j,2}b_{j,1}+a_{j,3}b_{j,4}-a_{j,4}b_{j,3}\bigr)\\ ={}& \Bigl\langle \sum_{j=1}^N\sum_{k=1}^4 a_{j,k}\,\iota_{e_{j,k}} \omega^1,\: \sum_{j=1}^N\sum_{k=1}^4 b_{j,k}\,\iota_{e_{j,k}} \omega^2 \Bigr\rangle \EWzJhvDisplayMathPeriod.
		\end{align*}
		Then we have for~$\alpha_j,\beta_j \in \EWzJhvTpStar Tp M$,
		\[\Bigl\langle \sum_{j=1}^3 \alpha_j \wedge \omega^j,\: \sum_{j=1}^3 \beta_j \wedge \omega^j \Bigr\rangle = \Bigl\langle \sum_{j=1}^3 \iota_{\alpha_j^\sharp} \omega^j,\: \sum_{j=1}^3 \iota_{\beta_j^\sharp} \omega^j \Bigr\rangle + \EWzJhvMathFracVCentered{N-1}{N}\sum_{j=1}^3\langle \alpha_j, \beta_j \rangle \EWzJhvDisplayMathPeriod.\]
		
		For~$a \in \Omega^{2,+}(V)$, we have locally
		\[\mathrm{d}_A a = \sum_{j=1}^{4N} e^j \wedge \nabla_{e_j}^A a = \sum_{j=1}^{4N}\sum_{k=1}^3{}\bigl\langle \nabla_{e_j}^A a, \omega^k \bigr\rangle \otimes e^j \wedge \omega^k\]
		where $(e_j)_{j=1}^{4N}$ is any local orthonormal basis of the tangent bundle and $(e^j)_{j=1}^{4N}$ is its dual. Note locally $\nabla_{e_j}^A a \in \Omega^{2,+}(V)$. For~$a,a' \in \Omega^{2,+}(V)$, we then have locally
		\begin{align*}
			\bigl\langle \mathrm{d}_A a, \mathrm{d}_A a' \bigr\rangle ={}& \Bigl\langle \sum_{k=1}^3\sum_{j=1}^{4N}{}\bigl\langle \nabla_{e_j}^A a, \omega^k \bigr\rangle \otimes \iota_{e_j} \omega^k,\: \sum_{k=1}^3\sum_{j=1}^{4N}{}\bigl\langle \nabla_{e_j}^A a', \omega^k \bigr\rangle \otimes \iota_{e_j} \omega^k \Bigr\rangle \\&+ \EWzJhvMathFracVCentered{N-1}{N}\sum_{k=1}^3 \Bigl\langle \sum_{j=1}^{4N}{}\bigl\langle \nabla_{e_j}^A a, \omega^k \bigr\rangle \otimes e^j,\: \sum_{j=1}^{4N}{}\bigl\langle \nabla_{e_j}^A a', \omega^k \bigr\rangle \otimes e^j \Bigr\rangle \EWzJhvDisplayMathPeriod.
		\end{align*}
		For $\displaystyle \mathrm{d}_A^*$, we have locally
		\begin{equation}\label{eqDAStar}
			\mathrm{d}_A^* a = -\sum_{j=1}^{4N} \iota_{e_j} \nabla_{e_j}^A a = -\sum_{j=1}^{4N}\sum_{k=1}^3{}\bigl\langle \nabla_{e_j}^A a, \omega^k \bigr\rangle \otimes \iota_{e_j} \omega^k \EWzJhvDisplayMathPeriod.
		\end{equation}
		Combine the two equations above to get the conclusion.
	\end{proof}
	
	\noindent\emph{Notation of\/\hspace{-0.08em} $F_A,R,\EWzJhvCDotInd\circledast\EWzJhvCDotInd$\hspace*{0.1em}:}\enskip In the following, we denote the curvature of $A$ by $F_A$, and the Riemann curvature tensor by $R$. For convenience, we use $\EWzJhvCDotInd\circledast\EWzJhvCDotInd$ to denote an unspecified sum of contractions, raising and lowering indices, and multiplications with scalar constants of the tensor product of two tensors. Notably, $\EWzJhvlvertRaised X \mathbin{\EWzJhvMathRaiseABit\circledast} Y\EWzJhvrvertRaised \mathrel{\EWzJhvMathRaiseABit\leqslant} C\hspace{0.0625em}\EWzJhvlvertRaised X\EWzJhvrvertRaised\EWzJhvlvertRaised Y\EWzJhvrvertRaised$, where $C$ is some unspecified real constant.
	
	\begin{EWzJhvProposition}[Weitzenböck formula on $\Omega^{2,+}(V)$]\label{propWeitzen2plus}
		Let\/ $(M,g,Q)$ be a smooth\/ $4N$\hspace*{-0.08em}-dimensional Riemannian manifold with holonomy\/ $\mathrm{Hol}(g) \subset \mathrm{Sp}(N)\hspace{0.0625em}\mathrm{Sp}(1)$. Let\/ $V \rightarrow M$ be a smooth real vector bundle with inner product. Let\/ $A$ be a connection on\/ $V$ that is compatible with the inner product on\/ $V$. For\/~$a \in \Omega^{2,+}(V)$,
		\[\bigl(\mathrm{d}_A\mathrm{d}_A^* a\bigr)^{\mskip-2mu\relax 2,+}=\EWzJhvMathFracVCentered{1}{2N}\bigl(\nabla^A\bigr)^{\mskip-5mu\relax *} \nabla^A a + (F_A \oplus R) \circledast a\]
		where\/ $(\EWzJhvCDotInd)^{2,+}$ denotes the orthogonal projection onto\/ $\Omega^{2,+}(V)$, and the operator\EWzJhvFootnote{\label{ftnPropWeitzen2plus}The operator $\EWzJhvCDotInd\circledast\EWzJhvCDotInd$ can be written out explicitly (see e.g.\ \cite[Theorem~B.2.1]{MaresPhD}). In particular, the part $R \circledast a$ of $(F_A \oplus R) \circledast a$ can actually be written as $R^+ \circledast a$ where $R^+$ is the orthogonal projection of $R$ onto $\operatorname{Sym}^2(\Omega^{2,+})$. In the case of $N \geqslant 2$, to show this, we can use \cite[Theorem~3.1]{QKManifoldsSalamon} (see also the proof of \autoref{propCurvQK1304}), which even implies $R \circledast a$ is actually proportional to $ta$ where $t$ is the scalar curvature.} $\EWzJhvCDotInd\circledast\EWzJhvCDotInd$ only depends on\/ $N$.
	\end{EWzJhvProposition}
	\begin{proof}
		The Weitzenböck formula for bundle-valued differential forms (see e.g.\ \cite[Theorem~B.2.1]{MaresPhD}) implies for any~$a,a' \in \Omega^{2,+}(V)$ with $a'$ compactly supported,
		\[\int_M{}\Bigl(\bigl\langle \nabla^A a, \nabla^A a' \bigr\rangle - \bigl\langle \mathrm{d}_A a, \mathrm{d}_A a' \bigr\rangle - \bigl\langle \mathrm{d}_A^* a, \mathrm{d}_A^* a' \bigr\rangle + (F_A \oplus R) \circledast a \circledast a'\Bigr)\,\mathrm{d}\mathrm{Vol}_g = 0\]
		where the operators $\EWzJhvCDotInd\circledast\EWzJhvCDotInd$ only depend on $N$. This together with \autoref{propDDStar2plus} implies the conclusion.
	\end{proof}
	
	See \autoref{propIntegrabEq} for the context of the following proposition. In particular, it is basically part of \cite[Remark~2.6~(1)]{CompatibleAlmostCplxStrsOnQKManifolds}. We put it here as we prefer to prove it using the notation of this subsection.
	
	\begin{EWzJhvProposition}\label{propIntegrabIneq}
		Let\/ \ifEWzJhvPDF\vspace{0.5pt}\fi $(M,g,Q)$ be a smooth\/ $4N$\hspace*{-0.08em}-dimensional Riemannian manifold with holonomy\/ $\mathrm{Hol}(g) \subset \mathrm{Sp}(N)\hspace{0.0625em}\mathrm{Sp}(1)$. Let\/ $\omega \in \Omega^{2,+}$. If\/\hspace{-0.08em} $\lvert\omega\rvert$ is a constant function over\/ \ifEWzJhvPDF\vspace{1pt}\fi $M$, \ifEWzJhvPDF\linebreak\fi then\/ $\displaystyle N\hspace{0.0625em}\bigl|\mathrm{d}^* \omega\bigr|\ifEWzJhvPDF{\rule{0pt}{1.88ex}}\fi^2 \leqslant \bigl|\nabla \omega\bigr|\ifEWzJhvPDF{\rule{0pt}{1.88ex}}\fi^2$.\vspace{-1pt}
	\end{EWzJhvProposition}
	\begin{proof}
		\linespread{1.11}\selectfont
		
		WLOG, we assume $\lvert\omega\rvert=1$. Let $\omega^1 := \omega$. Locally, we choose $\omega^2,\omega^3$ such that $(\omega^1\,\,\omega^2\,\,\omega^3)$ forms a local orthonormal basis of\/ $\Omega^{2,+}$. For $\omega=\omega^1$, as $\nabla_{\!e_j}\langle \omega^1, \omega^1 \rangle=0$, $\bigl\langle \nabla_{\!e_j}\omega, \omega^1 \bigr\rangle=0$, so \EWzJhvEqRef{eqDAStar} becomes ${\displaystyle \mathrm{d}^* \omega} = -\sum_{k=2}^3\sum_{j=1}^{4N}{}\bigl\langle \nabla_{\!e_j} \omega, \omega^k \bigr\rangle\,\iota_{e_j} \omega^k$. From \EWzJhvEqRef{eqOrthonBasis2plus} we can see $\bigl|\sum_{j=1}^{4N}{}\bigl\langle \nabla_{\!e_j} \omega, \omega^k \bigr\rangle\,\iota_{e_j} \omega^k\bigr|\ifEWzJhvPDF{\rule{0pt}{1.88ex}}\fi^2 = \frac{1}{2N}\sum_{j=1}^{4N} \bigl\langle \nabla_{\!e_j} \omega, \omega^k \bigr\rangle\ifEWzJhvPDF{\rule{0pt}{1.88ex}}\fi^{\mskip-1mu\relax 2}$. Therefore,
		\[\bigl|\mathrm{d}^* \omega\bigr|^2 \leqslant 2\hspace{0.0625em}\sum_{k=2}^3\left|\sum_{j=1}^{4N}{}\bigl\langle \nabla_{\!e_j} \omega, \omega^k \bigr\rangle\,\iota_{e_j} \omega^k\right|^2 = \EWzJhvMathFracVCentered{1}{N}\sum_{k=2}^3\sum_{j=1}^{4N} \bigl\langle \nabla_{\!e_j} \omega, \omega^k \bigr\rangle^{\mskip-1mu\relax 2} = \EWzJhvMathFracVCentered{1}{N}\hspace{0.0625em}\bigl|\nabla \omega\bigr|^2 \EWzJhvDisplayMathPeriod.\qedhere\]\par
	\end{proof}
	
	\subsection{Compatible complex structure, the bundles \texorpdfstring{$\Omega^{2n,0,+}$}{Ω²ⁿʼ⁰ʼ⁺}}\label{subsecCplxStr}
	
	Let $(M,g,Q)$ be a smooth Riemannian manifold with $\mathrm{Hol}(g) \subset \mathrm{Sp}(N)\hspace{0.0625em}\mathrm{Sp}(1)$. Let $J$ be a compatible complex structure on $(M,g,Q)$, that is, an integ\EWzJhvText{rab}le (almost) complex structure on $M$ that is also a section of\/ $Q$. Let $\Omega^{2,0,+} := \bigl(\Omega^{2,+}\otimes_{\mathbb{R}}\mathbb{C}\bigr) \cap \Omega^{2,0}$. In this subsection, after a brief introduction, we shall prove some simple results about $(M,g,Q,J)$, $\Omega^{2,0,+}$ (and its generalisations $\Omega^{2n,0,+}$), and some $\bar{\partial}$-operators. In particular, we shall see that the operator $\displaystyle \mathrm{d}^*$ on $\Omega^{2,0,+}$ can be seen as the $\bar{\partial}$-operator of some holomorphic structure on $\Omega^{2,0,+}$ (\autoref{prop2n0plusHolm}).
	
	\ifEWzJhvPageBreak\pagebreak\fi
	
	\bigskip\vspace{-1pt}
	
	\begingroup
	\linespread{1.08}\selectfont
	\EWzJhvParagraphNoBeforeskip*{Notation of $e_{j,k}, \bar{e}_{j,k}, e^{j,k}, \bar{e}^{j,k}, \langle\EWzJhvCDotInd,\EWzJhvCDotInd\rangle$:} Let $(e_{j,k})_{1 \leqslant j \leqslant N,1 \leqslant k \leqslant 2}$ be local smooth sections of\/ $T^{1,0}\mskip-2mu\relax M$ satisfying $\langle e_{j,k}, \bar{e}_{l,m} \rangle = \delta_{jl}\delta_{km}$, where $\bar{e}_{l,m} \in \Gamma(T^{0,1}\mskip-2mu\relax M)$ is the complex conjugate, and $\langle\EWzJhvCDotInd,\EWzJhvCDotInd\rangle \in \Gamma\bigl({\displaystyle (\EWzJhvTpStar T{\hphantom{p}} M)^{\otimes\EWzJhvMathVCenter 2}}\otimes_{\mathbb{R}}\mathbb{C}\bigr)$ is \emph{the\/ $\mathbb{C}$-linear extension} of the inner product on $TM$. Let $e^{j,k} \in \Omega^{1,0}$ and $\bar{e}^{j,k} \in \Omega^{0,1}$ be the dual covector fields. Note that $(e_{j,k})^\flat=\bar{e}^{j,k}$ and $(\bar{e}_{j,k})^\flat=e^{j,k}$, where the musical isomorphisms are also considered to be $\mathbb{C}$-linear. \ifEWzJhvPDF\linebreak\fi In addition, we require the choice of $(e_{j,k})_{1 \leqslant j \leqslant N,1 \leqslant k \leqslant 2}$ to make $\sum_{j=1}^N e^{j,1} \wedge e^{j,2}$ a local section of the complex line bundle\EWzJhvFootnote{For us, a ``complex line bundle'' is just a smooth rank-$1$ complex vector bundle. It is not endowed with a holomorphic structure, otherwise we would call it a \emph{holomorphic} line bundle.} $\Omega^{2,0,+}$. To see why we can make this requirement, just note the local description \EWzJhvEqRef{eqOrthonBasis2plus} of\/ $\Omega^{2,+}$ and that the complex structure $J$ is compatible.\par
	\endgroup
	
	\bigskip
	
	Now that the complex line bundle $\Omega^{2,0,+}$ can be described through the local section $\sum_{j=1}^N e^{j,1} \wedge e^{j,2}$, we easily see:\vspace{-\parskip}
	\begin{itemize}[topsep=1ex plus 0.25ex minus 0.25ex, itemsep=1ex plus 0.25ex minus 0.25ex, parsep=0pt]
		\item When $N=1$, $\Omega^{2,0,+}$ is simply $\Omega^{2,0}$.
		\item Through the splitting into real and imaginary parts, any~$a \in \Omega^{2,0,+}$ is a parascalar $(2,0)$-form. Its induced $\mathbb{Z}_2$-almost complex structure is simply given by $\pm J$.
	\end{itemize}
	
	\subsubsection*{Covariant derivatives on $\Omega^{s,t}$}
	
	\begin{EWzJhvProposition}
		$(\nabla_{\!e_{j,k}} e_{l,m})^{0,1} = \overline{(\nabla_{\!\bar{e}_{j,k}} \bar{e}_{l,m})^{1,0}} = 0$, where\/ $(\EWzJhvCDotInd)^{s,t}$ denotes the\/ $(s,t)$-part.
	\end{EWzJhvProposition}
	\begin{proof}
		This is a direct consequence of
		\[\bigl\langle\nabla_{\!e_{j,k}} e_{l,m}, e_{p,q}\bigr\rangle = \EWzJhvMathFracVCentered{1}{2}\Bigl(\bigl\langle[e_{j,k},e_{l,m}],e_{p,q}\bigr\rangle + \bigl\langle[e_{p,q},e_{j,k}],e_{l,m}\bigr\rangle - \bigl\langle[e_{l,m},e_{p,q}],e_{j,k}\bigr\rangle\Bigr)\]
		and the fact that because of the integ\EWzJhvText{rab}ility of $J$, the Lie bracket of two $(1,0)$-vector fields is still of type~$(1,0)$.
	\end{proof}
	
	Using the musical isomorphisms, we get
	\begin{EWzJhvCorollary}\label{corSTFormDer}
		For any\/ $(1,0)$-vector field\/ $X$ and any\/~$b \in \Omega^{s,t}$, the covariant derivative\/ $\nabla_{\!X} b$ which a priori is a section of\/\hspace{0em plus -0.08em} $\Omega^{s+t} \otimes_{\mathbb{R}}\mathbb{C}$, is actually a section of\/\hspace{0em plus -0.08em} $\Omega^{s,t} \oplus \Omega^{s-1,t+1}$, and\/ $\nabla_{\!\overline{X}} b$ is a section of\/\hspace{0em plus -0.08em} $\Omega^{s,t} \oplus \Omega^{s+1,t-1}$.
	\end{EWzJhvCorollary}
	
	\subsubsection*{Covariant derivatives on $\Omega^{2,+}$, ``$\,\omega^J\,$''``$\,\omega^+\,$''``$\,\phi\,$''``$\,\psi\,$''}
	
	Let $\omega^J := \langle J\EWzJhvCDotInd,\EWzJhvCDotInd \rangle = i\sum_{j=1}^N\bigl(e^{j,1} \wedge \bar{e}^{j,1} + e^{j,2} \wedge \bar{e}^{j,2}\bigr) \in \Omega^{1,1} \cap \Omega^{2,+}$ be the ``Kähler form''. Locally, let $\omega^+ := \sum_{j=1}^N e^{j,1} \wedge e^{j,2}$ be a section of\/ $\Omega^{2,0,+}$. For any $(0,1)$-vector field $X$, \ifEWzJhvPDF\vspace{1pt}\fi $\nabla_{\!X} \omega^J \in \bigl(\Omega^{1,1} \oplus \Omega^{2,0}\bigr) \cap \bigl(\Omega^{2,+} \otimes_{\mathbb{R}}\mathbb{C}\bigr)$. Note that $\omega^J$ is a section of the real line bundle $\Omega^{1,1} \cap \Omega^{2,+}$ with constant norm, \hypertarget{nablaOmegaJ}{therefore} \ifEWzJhvPDF\vspace{1pt}\fi $\nabla_{\!X} \omega^J \in \Omega^{2,0,+}$. Locally, let $\phi \in \Omega^{0,1}$ satisfy for any $(0,1)$-vector field $X$, \ifEWzJhvPDF\vspace{1pt}\fi $\nabla_{\!X} \omega^J = (\iota_X \phi)\,\omega^+$. This implies for any $(1,0)$-vector field $X$, $\nabla_{\!X} \omega^J = \overline{(\iota_{\overline{X}} \phi)\,\omega^+}$ and
	\[\bigl(\nabla_{\!X} \omega^+\bigr)^{\mskip-2mu\relax 1,1} = \EWzJhvMathFracVCentered{1}{2N}\,\bigl\langle \nabla_{\!X} \omega^+, {\textstyle \omega^J} \bigr\rangle\,{\textstyle \omega^J} = -\EWzJhvMathFracVCentered{1}{2N}\,\bigl\langle \nabla_{\!X} {\textstyle \omega^J}, \omega^+ \bigr\rangle\,{\textstyle \omega^J} = -\EWzJhvMathFracVCentered{1}{2}\,(\overline{\iota_{\overline{X}} \phi})\,{\textstyle \omega^J} \EWzJhvDisplayMathPeriod.\]
	In addition, locally, let $\psi \in \Omega^{0,1}$ satisfy for any $(0,1)$-vector field $X$, $\nabla_{\!X} \omega^+ = (\iota_X \psi)\,\omega^+$. Then for any $(0,1)$-vector field $X$,
	\[\bigl(\nabla_{\!X} \overline{\omega^+}\bigr)^{\mskip-2mu\relax 0,2} = \EWzJhvMathFracVCentered{1}{N}\,\bigl\langle \nabla_{\!X} \overline{\omega^+}, \omega^+ \bigr\rangle\,\overline{\omega^+} = -\EWzJhvMathFracVCentered{1}{N}\,\bigl\langle \nabla_{\!X} \omega^+, \overline{\omega^+} \bigr\rangle\,\overline{\omega^+} = -\,(\iota_X \psi)\,\overline{\omega^+} \EWzJhvDisplayMathPeriod.\]
	
	\ifEWzJhvPageBreak\pagebreak\fi
	
	\subsubsection*{The $\bar{\partial}$-operators\EWzJhvFootnote{An additional $\bar{\partial}$-operator, $\bar{\partial}^{\mathtt{tw}}$, is discussed in \autoref{subsecCplxStrAddResults}.} $\bar{\partial}^{\mathtt{CM}},\bar{\partial}^{\mathtt{LC}},\bar{\partial}^{\mathtt{D}}$, the bundles $\Omega^{2n,0,+}$}
	
	We shall discuss several $\bar{\partial}$-operators\EWzJhvInfixFootnote{that is, the $(0,1)$-part of some connection. The $(0,2)$-part of the curvature is not required to vanish.}. We denote the standard $\bar{\partial}$-operators on the bundle of $(1,0)$-(co)tangent vectors, and also bundles of their various tensor products, in particular $\Omega^{s,0}$, of the complex manifold $(M,J)$ by $\bar{\partial}^{\mathtt{CM}}$. Due to \autoref{corSTFormDer}, the $(0,1)$-part \ifEWzJhvPDF\vspace{0.5pt}\fi of the Levi-Civita connection $\nabla$ is a $\bar{\partial}$-operator on $\Omega^{s,0}$ and in particular $\Omega^{2,0,+}$, denoted by $\bar{\partial}^{\mathtt{LC}}$. Note $\nabla_{\!\bar{e}_{j,k}}(u\,\omega^+) = (\partial_{\bar{e}_{j,k}}u+u\,\iota_{\bar{e}_{j,k}}\psi)\,\omega^+$ for~$u\,\omega^+ \in \Omega^{2,0,+}$ where $u$ is some complex-valued function.
	
	For integer~$n$ satisfying $0 \leqslant n \leqslant N$, let $\Omega^{2n,0,+}$ be the image of the embedding between complex vector bundles $(\Omega^{2,0,+})^{\otimes_{\mathbb{C}}\mskip1.5mu\relax\EWzJhvMathVCenter n} \hookrightarrow \Omega^{2n,0}$ given by wedge product. Note that $\Omega^{2N,0,+}$ is just $\Omega^{2N,0}$.
	
	\begin{EWzJhvProposition}\label{propDBarD2n0plus}
		For any integer\/~$n$ satisfying\/ $1 \leqslant n \leqslant N$, there is a unique\/ $\bar{\partial}$-operator on\/ $\Omega^{2n,0,+}$, denoted by\/ $\bar{\partial}^{\mathtt{D}}$, such that for any\/~$b \in \Omega^{2n,0,+}$ and any\/~$p \in M$, $\displaystyle \mathrm{d}^*b$ vanishes at\/ $p$ if and only if\/\hspace{-0.08em} $\bar{\partial}^{\mathtt{D}}b$ vanishes at\/ $p$.
		
		Moreover, for any integer\/~$n$ satisfying\/ $1 \leqslant n \leqslant N$, we have an isomorphism of complex line bundles with\/ $\bar{\partial}$-operator\EWzJhvFootnote{Later we shall see that both sides of the isomorphism are actually holomorphic line bundles.} $(\Omega^{2,0,+}, \bar{\partial}^{\mathtt{D}}) \otimes_{\mathbb{C}} (\Omega^{2,0,+}, \bar{\partial}^{\mathtt{LC}})^{\otimes_{\mathbb{C}}(n-1)} \cong (\Omega^{2n,0,+}, \bar{\partial}^{\mathtt{D}})$, where the underlying isomorphism of complex line bundles is given by wedge product.
	\end{EWzJhvProposition}
	\begin{proof}
		Locally, let $u\,(\omega^+)^n \in \Omega^{2n,0,+}$ where $u$ is some complex-valued function.\\
		Noting $\nabla_{\!\bar{e}_{j,k}}(\omega^+)^n = n\,(\iota_{\bar{e}_{j,k}}\psi)\,(\omega^+)^n$, we calculate
		\begin{align*}
			&{-\sum_{j=1}^N{}\bigl(\iota_{e_{j,1}}\nabla_{\!\bar{e}_{j,1}} + \iota_{e_{j,2}}\nabla_{\!\bar{e}_{j,2}}\bigr)\bigl(u\,(\omega^+)^n\bigr)}\\
			={}& n\,(\omega^+)^{n-1} \wedge \sum_{j=1}^N\Bigl(\bigl(\partial_{\bar{e}_{j,2}}u+n\,u\,\iota_{\bar{e}_{j,2}}\psi\bigr)\,e^{j,1} - \bigl(\partial_{\bar{e}_{j,1}}u+n\,u\,\iota_{\bar{e}_{j,1}}\psi\bigr)\,e^{j,2}\Bigr) \EWzJhvDisplayMathPeriod.
		\end{align*}
		For $\iota_{\bar{e}_{j,k}}\nabla_{\!e_{j,k}}$, we calculate
		\begin{align*}
			&{-\sum_{j=1}^N{}\bigl(\iota_{\bar{e}_{j,1}}\nabla_{\!e_{j,1}} + \iota_{\bar{e}_{j,2}}\nabla_{\!e_{j,2}}\bigr)\bigl(u\,(\omega^+)^n\bigr)}\\
			={}& {-}\,n\,u\,(\omega^+)^{n-1} \wedge \sum_{j=1}^N{}\bigl(\iota_{\bar{e}_{j,1}}\nabla_{\!e_{j,1}} + \iota_{\bar{e}_{j,2}}\nabla_{\!e_{j,2}}\bigr)\,\omega^+\\
			={}& {-\EWzJhvMathFracVCentered{i}{2}}\,n\,u\,(\omega^+)^{n-1} \wedge \sum_{j=1}^N\bigl((\overline{\iota_{\bar{e}_{j,1}} \phi})\,e^{j,1} + (\overline{\iota_{\bar{e}_{j,2}} \phi})\,e^{j,2}\bigr) \EWzJhvDisplayMathPeriod.
		\end{align*}
		The sum of the two equations equals to ${\displaystyle \mathrm{d}^*}\bigl(u\,(\omega^+)^n\bigr)$, so $\bar{\partial}^{\mathtt{D}}$ is locally given by
		\begin{equation}\label{eqDBarD2n0plus}
		\begin{aligned}
			\bar{\partial}^{\mathtt{D}}\bigl(u\,(\omega^+)^n\bigr) = (\omega^+)^n \otimes_{\mathbb{C}} \sum_{j=1}^N\Bigl(&\bigl(\partial_{\bar{e}_{j,1}}u+n\,u\,\iota_{\bar{e}_{j,1}}\psi + \EWzJhvMathFracVCentered{i}{2}\,u\,\overline{\iota_{\bar{e}_{j,2}} \phi}\bigr)\,\bar{e}^{j,1} \\{}+{}& \bigl(\partial_{\bar{e}_{j,2}}u+n\,u\,\iota_{\bar{e}_{j,2}}\psi - \EWzJhvMathFracVCentered{i}{2}\,u\,\overline{\iota_{\bar{e}_{j,1}} \phi}\bigr)\,\bar{e}^{j,2}\Bigr) \EWzJhvDisplayMathPeriod.
		\end{aligned}
		\end{equation}
		As a consequence, we get $(\Omega^{2,0,+}, \bar{\partial}^{\mathtt{D}}) \otimes_{\mathbb{C}} (\Omega^{2,0,+}, \bar{\partial}^{\mathtt{LC}})^{\otimes_{\mathbb{C}}(n-1)} \cong (\Omega^{2n,0,+}, \bar{\partial}^{\mathtt{D}})$.\ifEWzJhvPageBreak\pagebreak\fi
	\end{proof}
	
	\begin{EWzJhvProposition}\label{propCurvQK1304}
		If\/\hspace{-0.08em} $N \geqslant 2$, then\/ $R(\bar{e}_{j_1,k_1},\bar{e}_{j_2,k_2})\bar{e}_{j_3,k_3}=0$, that is, the\/ $(1,3)$- and\/ $(0,4)$-parts of the Riemann curvature tensor vanish.
	\end{EWzJhvProposition}
	\begin{proof}
		Theorem~3.1 of \cite{QKManifoldsSalamon} decomposes the Riemann curvature tensor as $R = tR_0 + R_1$ \ifEWzJhvPDF\linebreak\fi where $t$ is the scalar curvature of $(M,g)$, $R_0$ is the Riemann curvature tensor of the quaternionic projective space $\mathbb{H}P^N$, and $R_1$ is a section of the symmetric square \ifEWzJhvPDF\linebreak\fi $\operatorname{Sym}^2({\displaystyle \mathfrak{sp}(N)^*}) \subset \operatorname{Sym}^2\bigl(\bigl(\mathfrak{sp}(N)\oplus\mathfrak{sp}(1)\bigr)^{\mskip-2mu\relax *}\bigr) \subset \operatorname{Sym}^2({\displaystyle \mathfrak{so}(4N)^*})$. Moreover, \EWzJhvTextNormal{(3.3)} of \cite{QKManifoldsSalamon} writes $R_0 = -c'(B_1 + B_2)$ where $c'$ is a real constant, and $B_1, B_2$ are the Killing forms on $\mathfrak{sp}(N), \mathfrak{sp}(1)$ respectively. Therefore, $R = -c'tB_2 + R_1'$ where $R_1'$ is a section of\/ $\operatorname{Sym}^2({\displaystyle \mathfrak{sp}(N)^*})$. We shall check $B_2$ and $R_1'$ individually.
		
		For $B_2$, after raising or lowering indices, locally $B_2$ can be written in our notation as $-c''\hspace{0.0625em}\bigl(\mskip2mu\relax\omega^J\otimes_{\mathbb{C}}\omega^J \mskip2mu\relax+\mskip2mu\relax 2\mskip5mu\relax\omega^+\otimes_{\mathbb{C}}\overline{\omega^+} \mskip2mu\relax+\mskip2mu\relax 2\mskip5mu\relax\overline{\omega^+}\otimes_{\mathbb{C}}\omega^+\mskip2mu\relax\bigr)$, where there is no $(1,3)$- or $(0,4)$-part.
		
		For $R_1'$, note that $R_1'$ is a section of\/ $\operatorname{Sym}^2({\displaystyle \mathfrak{sp}(N)^*}) \subset \operatorname{Sym}^2({\displaystyle \mathfrak{u}(2N)^*})$. Therefore, $R_1'(\bar{e}_{j_1,k_1},\bar{e}_{j_2,k_2})=0$. One can say that $R_1'$ behaves like the Riemann curvature tensor of a hyper\EWzJhvText{käh}ler manifold.
	\end{proof}
	
	\begin{EWzJhvCorollary}\label{cor20plusLCHolo}
		If\/\hspace{-0.08em} $N \geqslant 2$, then\/ $(\Omega^{2,0,+}, \bar{\partial}^{\mathtt{LC}})$ is a holomorphic line bundle.
	\end{EWzJhvCorollary}
	
	\begin{EWzJhvProposition}\label{prop2N0plusDCanonical}
		$(\Omega^{2N,0,+}, \bar{\partial}^{\mathtt{D}}) = (\Omega^{2N,0}, \bar{\partial}^{\mathtt{CM}})$
	\end{EWzJhvProposition}
	\begin{proof}
		Any~$b \in \Omega^{2N,0,+} = \Omega^{2N,0}$ is self-dual, so (locally) $\displaystyle \mathrm{d}^*b=0$ if and only if\/ $\mathrm{d}b=0$.
	\end{proof}
	
	\begin{EWzJhvProposition}\label{prop2n0plusHolm}
		For any integer\/~$n$ satisfying\/ $1 \leqslant n \leqslant N$, $(\Omega^{2n,0,+}, \bar{\partial}^{\mathtt{D}})$ is a holomorphic line bundle.
	\end{EWzJhvProposition}
	\begin{proof}
		For~$n=N$, use \autoref{prop2N0plusDCanonical}.
		
		For~$n \ifEWzJhvPDF\hspace*{0pt minus 0.82pt}\fi=\ifEWzJhvPDF\hspace*{0pt minus 0.82pt}\fi 1 \ifEWzJhvPDF\hspace*{0pt minus 0.82pt}\fi<\ifEWzJhvPDF\hspace*{0pt minus 0.82pt}\fi N$, use the isomorphism $(\Omega^{2,0,+}\ifEWzJhvPDF\hspace*{0pt minus 0.82pt}\fi,\ifEWzJhvPDF\hspace*{0pt minus 0.82pt}\fi \bar{\partial}^{\mathtt{D}}) \otimes_{\mathbb{C}} (\Omega^{2,0,+}\ifEWzJhvPDF\hspace*{0pt minus 0.82pt}\fi,\ifEWzJhvPDF\hspace*{0pt minus 0.82pt}\fi \bar{\partial}^{\mathtt{LC}})^{\otimes_{\mathbb{C}}(N-1)} \ifEWzJhvPDF\hspace*{0pt minus 0.82pt}\fi\cong\ifEWzJhvPDF\hspace*{0pt minus 0.82pt}\fi (\Omega^{2N,0,+}\ifEWzJhvPDF\hspace*{0pt minus 0.82pt}\fi,\ifEWzJhvPDF\hspace*{0pt minus 0.82pt}\fi \bar{\partial}^{\mathtt{D}})$ from \autoref{propDBarD2n0plus} together with \autoref{cor20plusLCHolo}.
		
		For~$1 \ifEWzJhvPDF\hspace*{0pt minus 0.45pt}\fi<\ifEWzJhvPDF\hspace*{0pt minus 0.45pt}\fi n \ifEWzJhvPDF\hspace*{0pt minus 0.45pt}\fi<\ifEWzJhvPDF\hspace*{0pt minus 0.45pt}\fi N$, use the isomorphism $(\Omega^{2,0,+}\ifEWzJhvPDF\hspace*{0pt minus 0.45pt}\fi,\ifEWzJhvPDF\hspace*{0pt minus 0.45pt}\fi \bar{\partial}^{\mathtt{D}}) \otimes_{\mathbb{C}} (\Omega^{2,0,+}\ifEWzJhvPDF\hspace*{0pt minus 0.45pt}\fi,\ifEWzJhvPDF\hspace*{0pt minus 0.45pt}\fi \bar{\partial}^{\mathtt{LC}})^{\otimes_{\mathbb{C}}(n-1)} \ifEWzJhvPDF\hspace*{0pt minus 0.45pt}\fi\cong\ifEWzJhvPDF\hspace*{0pt minus 0.45pt}\fi (\Omega^{2n,0,+}\ifEWzJhvPDF\hspace*{0pt minus 0.45pt}\fi,\ifEWzJhvPDF\hspace*{0pt minus 0.45pt}\fi \bar{\partial}^{\mathtt{D}})$ from \autoref{propDBarD2n0plus} together with \autoref{cor20plusLCHolo}.
	\end{proof}
	\begin{EWzJhvRemark}
		As a consequence, we have a ``converse'' to \autoref{thm2} in the following sense: given $(M,g,Q)$ with $\mathrm{Hol}(g) \subset \mathrm{Sp}(N)\hspace{0.0625em}\mathrm{Sp}(1)$ and given a compatible complex structure $J$, \ifEWzJhvPDF\linebreak\fi \emph{locally} there always exist nonzero coclosed sections of\/ $\Omega^{2,0,+}$, which are parascalar $(2,0)$-forms with induced $\mathbb{Z}_2$-almost complex structure given by $\pm J$.
	\end{EWzJhvRemark}
	
	\subsubsection*{Extremal property of compatible complex structure}
	
	The following proposition indicates that $\omega^J$ attends the equality in \autoref{propIntegrabIneq}.\EWzJhvFootnote{\label{ftnPrePropIntegrabEq}%
		It is also easy to prove the converse, i.e.\ for a compatible almost complex structure $J$, if\/ $\omega^J$ attends the equality in \autoref{propIntegrabIneq}, then $J$ is integ\EWzJhvText{rab}le \cite[Remark~2.6~(1)]{CompatibleAlmostCplxStrsOnQKManifolds}.\ifEWzJhvPDF\vspace{0.4pt}\fi
		
		Due to \autoref{propDDStar2plus}, \autoref{propIntegrabEq} is equivalent to $\bigl|\mathrm{d} \omega^J\bigr|\ifEWzJhvPDF{\rule{0pt}{1.88ex}}\fi^2=\bigl|\nabla \omega^J\bigr|\ifEWzJhvPDF{\rule{0pt}{1.88ex}}\fi^2$, which actually holds for general hermitian manifolds and reflects the basic fact that given a hermitian manifold, if its ``Kähler form'' is closed, then it is Kähler. However, the inequality of \autoref{propIntegrabIneq} does not generalise to general almost hermitian manifolds in this manner---both $\bigl|\mathrm{d} \omega^J\bigr|\ifEWzJhvPDF{\rule{0pt}{1.88ex}}\fi^2>\bigl|\nabla \omega^J\bigr|\ifEWzJhvPDF{\rule{0pt}{1.88ex}}\fi^2$ and $\bigl|\mathrm{d} \omega^J\bigr|\ifEWzJhvPDF{\rule{0pt}{1.88ex}}\fi^2<\bigl|\nabla \omega^J\bigr|\ifEWzJhvPDF{\rule{0pt}{1.88ex}}\fi^2$ \ifEWzJhvPDF\vspace{-1.2pt}\fi are possible. See \cite[Remark~2]{GauduchonCplxStrCConfMNegT}, which also gives a general formula for $\bigl|\nabla \omega^J\bigr|\ifEWzJhvPDF{\rule{0pt}{1.88ex}}\fi^2-\bigl|\mathrm{d} \omega^J\bigr|\ifEWzJhvPDF{\rule{0pt}{1.88ex}}\fi^2$.\ifEWzJhvPDF\vspace{0.4pt}\fi}
	It is basically \cite[Remark~2.6~(1)]{CompatibleAlmostCplxStrsOnQKManifolds}.
	
	\begin{EWzJhvProposition}\label{propIntegrabEq}\hspace*{-0.5em}\EWzJhvFootnote{Keep in mind\EWzJhvFootRef{ftnInnPdExtPdTenPd} that our inner product for exterior products is different from the one for tensor products, i.e.\ $\bigl|\nabla J\bigr|=\sqrt{2}\mskip2mu\relax\bigl|\nabla \omega^J\bigr|$, etc.}
		$N\hspace{0.0625em}\bigl|{\displaystyle \mathrm{d}^*} \omega^J\bigr|\ifEWzJhvPDF{\rule{0pt}{1.88ex}}\fi^2=\bigl|\nabla \omega^J\bigr|\ifEWzJhvPDF{\rule{0pt}{1.88ex}}\fi^2$
	\end{EWzJhvProposition}
	\begin{proof}
		Locally, for ${\displaystyle \mathrm{d}^*} \omega^J$, we have
		\[-\sum_{j=1}^N{}\bigl(\iota_{e_{j,1}}\nabla_{\!\bar{e}_{j,1}} + \iota_{e_{j,2}}\nabla_{\!\bar{e}_{j,2}}\bigr)\,\omega^J = \sum_{j=1}^N\bigl((\iota_{\bar{e}_{j,2}} \phi)\,e^{j,1} - (\iota_{\bar{e}_{j,1}} \phi)\,e^{j,2}\bigr) \EWzJhvDisplayMathPeriod.\]\ifEWzJhvPageBreak\par\pagebreak\noindent\fi
		Then \ifEWzJhvPDF\vspace{-2pt}\fi ${\displaystyle \mathrm{d}^*}\omega^J = ({\displaystyle \mathrm{d}^*}\omega^J)^{1,0} + \overline{({\displaystyle \mathrm{d}^*}\omega^J)^{1,0}}$ and $({\displaystyle \mathrm{d}^*}\omega^J)^{1,0} = \sum_{j=1}^N\bigl((\iota_{\bar{e}_{j,2}} \phi)\,e^{j,1} - (\iota_{\bar{e}_{j,1}} \phi)\,e^{j,2}\bigr)$, \ifEWzJhvPDF\\\fi so $\displaystyle \bigl|\mathrm{d}^* {\textstyle \omega^J}\bigr|\ifEWzJhvPDF{\rule{0pt}{1.88ex}}\fi^2 = 2\hspace{0.0625em}\lvert\phi\rvert^2$. On the other hand,
		\[\bigl|\nabla {\textstyle \omega^J}\bigr|^2 = 2\hspace{0.0625em}\sum_{j=1}^N\sum_{k=1}^2{}\bigl|\nabla_{\!\bar{e}_{j,k}} {\textstyle \omega^J}\bigr|^2 = 2\hspace{0.0625em}\sum_{j=1}^N\sum_{k=1}^2{}\bigl|(\iota_{\bar{e}_{j,k}} \phi)\,\omega^+\bigr|^2 = 2N\hspace{0.0625em}\lvert\phi\rvert^2 \EWzJhvDisplayMathPeriod.\]
		Therefore, $N\hspace{0.0625em}\bigl|{\displaystyle \mathrm{d}^*} \omega^J\bigr|\ifEWzJhvPDF{\rule{0pt}{1.88ex}}\fi^2=\bigl|\nabla \omega^J\bigr|\ifEWzJhvPDF{\rule{0pt}{1.88ex}}\fi^2$.
	\end{proof}
	
	\begin{EWzJhvRemark}
		When $M$ is compact, we observe two relations between $\bigl\|{\displaystyle \mathrm{d}^*} \omega^J\bigr\|_{L^2}$ and $\bigl\|\nabla \omega^J\bigr\|_{L^2}$, one deduced from \autoref{propIntegrabEq} and one deduced from the Weitzenböck formula {\small (\autoref{propWeitzen2plus})}. The interplay of these two relations leads to a quick proof that if $M$ is compact, $N \geqslant 2$ (or $N = 1$ and the Weyl curvature tensor of $(M,g,Q)$ is anti-self-dual), and the scalar curvature is nonpositive\EWzJhvInfixFootnote{When $N \geqslant 2$, a compact $(M,g,Q)$ with $\mathrm{Hol}(g) \subset \mathrm{Sp}(N)\hspace{0.0625em}\mathrm{Sp}(1)$ and \emph{positive} scalar curvature \ifEWzJhvPDF\linebreak\fi (necessarily simply-connected by \cite[Theorem~6.6]{QKManifoldsSalamon}) admits no compatible \emph{almost} complex structure \ifEWzJhvPDF\linebreak\fi at all \cite[Theorem~3.8]{CompatibleAlmostCplxStrsOnQKManifolds}. When $N = 1$, a compact anti-self-dual $(M,g,Q)$ with positive scalar curvature and \emph{even first Betti number} admits no compatible almost complex structure by a topological \ifEWzJhvPDF\linebreak\fi argument (see e.g.\ \cite[subsection~1.1.7]{Geometry4Manifolds}), as its intersection form must be negative definite due to the Weitzenböck formula {\scriptsize (\autoref{propWeitzen2plus})}. On the other hand, a Hopf surface can be an example of a compact anti-self-dual $(M,g,Q)$ with positive scalar curvature that admits a compatible complex structure.}, then $(M,J)$ must be Kähler and the scalar curvature must actually vanish. This is basically part of \cite[Theorem~4.3]{CompatibleAlmostCplxStrsOnQKManifolds}.
	\end{EWzJhvRemark}
	
	The interplay of \autoref{propIntegrabIneq},~\ref{propIntegrabEq} and the Weitzenböck formula {\small (\autoref{propWeitzen2plus})} also leads to an elliptic PDE satisfied by $J$, \autoref{lmaPdeJa}, as we will see. For that purpose, \ifEWzJhvPDF\linebreak\fi we prove a lemma here as preparation.
	\begin{EWzJhvLemma}\label{lmaJaEqPre}\hspace*{-0.5em}\EWzJhvFootnote{\label{ftnLmaJaEqPre}Even though we do not have a generalisation of \autoref{propIntegrabIneq} to general almost hermitian manifolds (see \EWzJhvPrefixFootRef{ftnPrePropIntegrabEq}above), \autoref{lmaJaEqPre} still generalises to general hermitian manifolds: over a hermitian manifold, for~$b \in \Omega^{2,0}\oplus\Omega^{0,2}$, $\bigl\langle \mathrm{d} \omega^J, \mathrm{d} b \bigr\rangle = \bigl\langle \nabla\omega^J, \nabla b \bigr\rangle$ (\autoref{lmaJaEqPre} is a special case of this due to \autoref{propDDStar2plus}). This generalised statement follows from a direct calculation. This also follows from a variational argument like that of \autoref{lmaJaEqPre}, but instead of using \autoref{propIntegrabIneq}, we use \EWzJhvTextNormal{(23)} of \cite[Remark~2]{GauduchonCplxStrCConfMNegT}, which expresses $\bigl|\nabla \omega^J\bigr|\ifEWzJhvPDF{\rule{0pt}{1.88ex}}\fi^2 - \bigl|\mathrm{d} \omega^J\bigr|\ifEWzJhvPDF{\rule{0pt}{1.88ex}}\fi^2$ as the difference between two nonnegative quantities, both of which vanish when $J$ is integ\EWzJhvText{rab}le.}
		For\/~$b \in \Omega^{2,+}$ with\/ $\langle b, \omega^J \rangle = 0$, $N \bigl\langle {\displaystyle \mathrm{d}^*} \omega^J, {\displaystyle \mathrm{d}^*} b \bigr\rangle = \bigl\langle \nabla\omega^J, \nabla b \bigr\rangle$.\vspace{-1pt}
	\end{EWzJhvLemma}
	\begin{proof}
		\linespread{1.11}\selectfont
		
		Let $\omega_t\ (t \in \mathbb{R})$ be a smooth variation of smooth sections of\/ $\Omega^{2,+}$ with constant norm satisfying $\omega_0 \ifEWzJhvPDF\hspace*{0pt minus 1pt}\fi=\ifEWzJhvPDF\hspace*{0pt minus 1pt}\fi \omega^J$ and $\frac{\mathrm{d}}{\mathrm{d}t}\omega_t\bigr|_{t=0} \ifEWzJhvPDF\hspace*{0pt minus 1pt}\fi=\ifEWzJhvPDF\hspace*{0pt minus 1pt}\fi b$. \autoref{propIntegrabIneq} implies \ifEWzJhvPDF\vspace{1pt}\fi $\bigl|\nabla \omega_t\bigr|\ifEWzJhvPDF{\rule{0pt}{1.88ex}}\fi^2 - N\hspace{0.0625em}\bigl|{\displaystyle \mathrm{d}^*} \omega_t\bigr|\ifEWzJhvPDF{\rule{0pt}{1.88ex}}\fi^2 \ifEWzJhvPDF\hspace*{0pt minus 1pt}\fi\geqslant\ifEWzJhvPDF\hspace*{0pt minus 1pt}\fi 0$, and \autoref{propIntegrabEq} implies $\bigl|\nabla \omega_0\bigr|\ifEWzJhvPDF{\rule{0pt}{1.88ex}}\fi^2 - N\hspace{0.0625em}\bigl|{\displaystyle \mathrm{d}^*} \omega_0\bigr|\ifEWzJhvPDF{\rule{0pt}{1.88ex}}\fi^2 = 0$, so $\frac{\mbox{\footnotesize$\mathrm{d}$}}{\mbox{\footnotesize$\mathrm{d}t$}}\bigl(\bigl|\nabla \omega_t\bigr|\ifEWzJhvPDF{\rule{0pt}{1.88ex}}\fi^2 - N\hspace{0.0625em}\bigl|{\displaystyle \mathrm{d}^*} \omega_t\bigr|\ifEWzJhvPDF{\rule{0pt}{1.88ex}}\fi^2\bigr)\bigr|_{t=0} = 0$, which when rearranged is just the conclusion.\qedhere\par
	\end{proof}
	
	\addvspace{2.75ex plus 1ex minus 0.2ex}
	{\addtolength{\leftskip}{2cm}\addtolength{\rightskip}{2cm}\noindent\emph{In addition, some more results about\/ $\Omega^{2n,0,+}$ and the several\/ $\bar{\partial}$-operators are presented in \autoref{subsecCplxStrAddResults}. Those results are not needed elsewhere in this article, but still interesting.}\par}
	\addvspace{2.75ex plus 1ex minus 0.2ex}
	
	\subsection{First step of proof: Integ\EWzJhvText{rab}ility}\label{subsecIntegrability}
	
	We now work towards \autoref{thm2}. Assume the assumptions of \autoref{thm2}. \ifEWzJhvPDF\linebreak\fi Let $p_{\oldstylenums{0}} \in M$ satisfy $a\rvert_{p_{\oldstylenums{0}}} \neq 0$. Let $U \subset M$ be an open neighbourhood of $p_{\oldstylenums{0}}$ such that $a$ is nowhere zero on $U$, the $\mathbb{Z}_2$-almost complex structure $\tilde{J}_a$ comes from an almost complex structure denoted by $J_a$ over $U$, and the tangent bundle $TM$ together with $Q$ and $J_a$ can be trivialised over $U$. In this subsection, we shall only work over $U$. \ifEWzJhvPDF\linebreak\fi We shall do some pointwise calculation and prove that $J_a$ is integ\EWzJhvText{rab}le. We reuse the notation of\/ $\Omega^{2,0,+}, e_{j,k}, \bar{e}_{j,k}, e^{j,k}, \bar{e}^{j,k}, \langle\EWzJhvCDotInd,\EWzJhvCDotInd\rangle, \omega^{J_a}, \omega^+$ from \autoref{subsecCplxStr}, which here are defined over $U$ with respect to $J_a$. Note that we must prove the integ\EWzJhvText{rab}ility of $J_a$ first before we can apply most of the results from \autoref{subsecCplxStr}.
	
	$\omega^+ + \overline{\omega^+}$ and $i(\omega^+ - \overline{\omega^+})$ are everywhere orthogonal with equal norm in $\Omega^{2,+}(U)$, both of which are everywhere orthogonal with $\omega^{J_a}$. Because $a \in \Omega^{2,+}(V)$ is assumed to be a parascalar $(2,0)$-form and nowhere zero over $U$, let $(v_1\,\,v_2)$ be smooth sections of\/ $V$ \ifEWzJhvPDF\linebreak\fi over $U$ that are everywhere orthonormal such that
	\begin{equation}\label{eqACoord}
		\begin{aligned}
			a&=\lambda\,\bigl(v_1\otimes_{\mathbb{R}} (\omega^+ + \overline{\omega^+}) + v_2\otimes_{\mathbb{R}} i(\omega^+ - \overline{\omega^+})\bigr)\\&=\lambda\,(v_1+iv_2)\otimes_{\mathbb{C}}\omega^+ + \lambda\,(v_1-iv_2)\otimes_{\mathbb{C}}\overline{\omega^+}
		\end{aligned}
	\end{equation}
	where $\lambda$ is some smooth positive real-valued function on $U$.
	
	Let
	\begin{equation}\label{eqA20}
		a^{2,0}=a\biggl(\EWzJhvMathFracVCentered{\mathrm{id}-iJ_a}{2}\EWzJhvCDotInd,\EWzJhvMathFracVCentered{\mathrm{id}-iJ_a}{2}\EWzJhvCDotInd\biggr)=\lambda\,(v_1+iv_2)\otimes_{\mathbb{C}}\omega^+ \EWzJhvDisplayMathPeriod.
	\end{equation}
	Then $a=a^{2,0}+\overline{a^{2,0}}$ and $\lvert a\rvert=\sqrt{2}\mskip2mu\relax\bigl|a^{2,0}\bigr|$.
	
	\medskip
	
	Recall we use $\langle\EWzJhvCDotInd,\EWzJhvCDotInd\rangle$ to denote \emph{the\/ $\mathbb{C}$-linear extension} of some real inner product.
	\begin{EWzJhvLemma}
		For any (complex)\EWzJhvFootnote{For us, a ``complex'' vector field (over $M$) just means a smooth section of\/ $TM \otimes_{\mathbb{R}} \mathbb{C}$.} vector field $X$, $\bigl\langle \nabla^A_{\!X} a^{2,0}, \nabla^A_{\!X} a^{2,0} \bigr\rangle = 0$.
	\end{EWzJhvLemma}
	\begin{proof}
		$\nabla^A_{\!X} a^{2,0}$ has no $(0,2)$-part, that is,\\[\glueexpr-\baselineskip/2\relax]\mbox{}
		\[\nabla^A_{\!X} a^{2,0} = \bigl(\nabla^A_{\!X} a^{2,0}\bigr)^{\mskip-2mu\relax 2,0} + \bigl(\nabla^A_{\!X} a^{2,0}\bigr)^{\mskip-2mu\relax 1,1} \EWzJhvDisplayMathPeriod.\]
		Then
		\[\bigl\langle \nabla^A_{\!X} a^{2,0}, \nabla^A_{\!X} a^{2,0} \bigr\rangle = \bigl\langle \bigl(\nabla^A_{\!X} a^{2,0}\bigr)^{\mskip-2mu\relax 1,1}, \bigl(\nabla^A_{\!X} a^{2,0}\bigr)^{\mskip-2mu\relax 1,1} \bigr\rangle \EWzJhvDisplayMathPeriod.\]
		However,
		\[\bigl(\nabla^A_{\!X} a^{2,0}\bigr)^{\mskip-2mu\relax 1,1} = \lambda\,(v_1+iv_2)\otimes_{\mathbb{C}}\bigl(\nabla_{\!X} \omega^+\bigr)^{\mskip-2mu\relax 1,1} \EWzJhvDisplayMathPeriod.\]
		Because \ifEWzJhvPDF\vspace{1pt}\fi $(v_1\,\,v_2)$ are everywhere orthonormal, $\bigl\langle v_1+iv_2,\: v_1+iv_2\bigr\rangle=0$. As a consequence, $\bigl\langle \nabla^A_{\!X} a^{2,0}, \nabla^A_{\!X} a^{2,0} \bigr\rangle = 0$.
	\end{proof}
	
	\begin{EWzJhvCorollary}\label{corNablaAA20}
		$\bigl|\nabla^A a\bigr| = \sqrt{2}\mskip2mu\relax\bigl|\nabla^A a^{2,0}\bigr|$
	\end{EWzJhvCorollary}
	\begin{proof}
		For any real vector field $X$,
		\[\bigl|\nabla^A_{\!X} a\bigr|^2 = \bigl|\nabla^A_{\!X} a^{2,0}\bigr|^2 + \bigl\langle \nabla^A_{\!X} a^{2,0}, \overline{\nabla^A_{\!X} \overline{a^{2,0}}} \bigr\rangle + \bigl\langle \nabla^A_{\!X} \overline{a^{2,0}}, \overline{\nabla^A_{\!X} a^{2,0}} \bigr\rangle + \bigl|\nabla^A_{\!X} \overline{a^{2,0}}\bigr|^2 \EWzJhvDisplayMathPeriod.\]
		Among the terms, note \ifEWzJhvPDF\vspace{1pt}\fi $\bigl\langle \nabla^A_{\!X} a^{2,0}, \overline{\nabla^A_{\!X} \overline{a^{2,0}}} \bigr\rangle = \overline{\bigl\langle \nabla^A_{\!X} \overline{a^{2,0}}, \overline{\nabla^A_{\!X} a^{2,0}} \bigr\rangle} = \bigl\langle \nabla^A_{\!X} a^{2,0}, \nabla^A_{\!X} a^{2,0} \bigr\rangle = 0$ and $\bigl|\nabla^A_{\!X} a^{2,0}\bigr|\ifEWzJhvPDF{\rule{0pt}{1.88ex}}\fi^2 = \bigl|\nabla^A_{\!X} \overline{a^{2,0}}\bigr|\ifEWzJhvPDF{\rule{0pt}{1.88ex}}\fi^2$, therefore $\bigl|\nabla^A_{\!X} a\bigr|\ifEWzJhvPDF{\rule{0pt}{1.88ex}}\fi^2 = 2\hspace{0.0625em}\bigl|\nabla^A_{\!X} a^{2,0}\bigr|\ifEWzJhvPDF{\rule{0pt}{1.88ex}}\fi^2$, so $\bigl|\nabla^A a\bigr| = \sqrt{2}\mskip2mu\relax\bigl|\nabla^A a^{2,0}\bigr|$.
	\end{proof}
	
	In addition, we use the formula
	\begin{equation}\label{eq8}
		\mathrm{d}_A^*=-\sum_{j=1}^N\sum_{k=1}^2\bigl(\iota_{e_{j,k}}\nabla^A_{\bar{e}_{j,k}} + \iota_{\bar{e}_{j,k}}\nabla^A_{e_{j,k}}\bigr) \EWzJhvDisplayMathPeriod.
	\end{equation}
	
	\begin{EWzJhvLemma}\label{lmaA20eq}
		${\displaystyle \mathrm{d}_A^*}a^{2,0}=0$
	\end{EWzJhvLemma}
	\begin{proof}
		Because $(v_1\,\,v_2)$ are everywhere orthonormal, we have $\bigl\langle v_1+iv_2,\: v_1+iv_2\bigr\rangle=0$, where $\langle\EWzJhvCDotInd,\EWzJhvCDotInd\rangle$ is the $\mathbb{C}$-linear extension of the inner product on $V$. As a consequence, $\bigl\langle\nabla^A(v_1+iv_2),\: v_1+iv_2\bigr\rangle=\frac{1}{2}\nabla\bigl\langle v_1+iv_2,\: v_1+iv_2\bigr\rangle=0$. Using \EWzJhvEqRef{eq8}, we get
		\begin{equation}\label{eq9}
			\bigl\langle\mathrm{d}_A^*\bigl(\lambda\,(v_1+iv_2)\otimes_{\mathbb{C}}\omega^+\bigr),\: v_1+iv_2\bigr\rangle=0 \EWzJhvDisplayMathPeriod.
		\end{equation}\ifEWzJhvPageBreak\par\pagebreak\noindent\fi
		Similarly,\ifEWzJhvPDF\makebox[0.5\linewidth]{}\vspace{\glueexpr-\baselineskip/2\relax}\fi
		\begin{equation}\label{eq10}
			\bigl\langle\mathrm{d}_A^*\bigl(\lambda\,(v_1-iv_2)\otimes_{\mathbb{C}}\overline{\omega^+}\bigr),\: v_1-iv_2\bigr\rangle=0 \EWzJhvDisplayMathPeriod.
		\end{equation}
		However, $\displaystyle \bigl\langle\mathrm{d}_A^*a,\: v_1-iv_2\bigr\rangle=0$, whose difference with \EWzJhvEqRef{eq10} gives
		\begin{equation}\label{eq11}
			\bigl\langle\mathrm{d}_A^*\bigl(\lambda\,(v_1+iv_2)\otimes_{\mathbb{C}}\omega^+\bigr),\: v_1-iv_2\bigr\rangle=0 \EWzJhvDisplayMathPeriod.
		\end{equation}
		\EWzJhvEqRef{eq9} together with \EWzJhvEqRef{eq11} gives $\bigl\langle{\displaystyle \mathrm{d}_A^*}a^{2,0},v_1\bigr\rangle = \bigl\langle{\displaystyle \mathrm{d}_A^*}a^{2,0},v_2\bigr\rangle = 0$.
		
		\smallskip
		
		We still need to show $\bigl\langle{\displaystyle \mathrm{d}_A^*}a^{2,0},v_\perp\bigr\rangle=0$ for any section $v_\perp$ of\/ $V$ on $U$ satisfying $\langle v_\perp,v_1\rangle = \langle v_\perp,v_2\rangle = 0$. Note that in ${\displaystyle \mathrm{d}_A^*}a^{2,0}$, the derivation must hit $(v_1+iv_2)$ to make the inner product with $v_\perp$ nonzero, thus $\bigl\langle{\displaystyle \mathrm{d}_A^*}a^{2,0},v_\perp\bigr\rangle \in \Omega^{1,0}(U)$ and similarly $\bigl\langle{\displaystyle \mathrm{d}_A^*}\overline{a^{2,0}},v_\perp\bigr\rangle \in \Omega^{0,1}(U)$, which together with $\bigl\langle{\displaystyle \mathrm{d}_A^*}a,v_\perp\bigr\rangle=0$ imply $\bigl\langle{\displaystyle \mathrm{d}_A^*}a^{2,0},v_\perp\bigr\rangle=0$.
	\end{proof}
	
	\begin{EWzJhvLemma}\label{lmaJaIntegrability}
		$J_a$ is integrable over\/ $U$. Therefore, $\tilde{J}_a$ is a\/ $\mathbb{Z}_2$-complex structure on the complement of the zero locus of\/\hspace{-0.08em} $a$.\vspace{-1pt}
	\end{EWzJhvLemma}
	\begin{proof}
		\linespread{1.08}\selectfont
		
		Let $\phi' \in \Omega^{0,1}(U)$ satisfy $\bigl(\nabla_{\!X} \omega^{J_a}\bigr)\ifEWzJhvPDF{\rule{0pt}{1.88ex}}\fi^{\mskip-2mu\relax 0,2} = (\iota_X\phi')\,\overline{\omega^+}$ for any $(0,1)$-vector field $X$ over $U$. Then for any $(0,1)$-vector field $X$ over $U$, we have
		\[\bigl(\nabla_{\!X} \omega^+\bigr)^{\mskip-2mu\relax 1,1} = \EWzJhvMathFracVCentered{1}{2N}\,\bigl\langle \nabla_{\!X} \omega^+, {\textstyle \omega^{J_a}} \bigr\rangle\,{\textstyle \omega^{J_a}} = -\EWzJhvMathFracVCentered{1}{2N}\,\bigl\langle \nabla_{\!X} {\textstyle \omega^{J_a}}, \omega^+ \bigr\rangle\,{\textstyle \omega^{J_a}} = -\EWzJhvMathFracVCentered{1}{2}\,(\iota_X\phi')\,{\textstyle \omega^{J_a}} \EWzJhvDisplayMathPeriod.\]
		Using \EWzJhvEqRef{eq8}, we compute the $(0,1)$-part of\/ ${\displaystyle \mathrm{d}_A^*}a^{2,0}$:
		\begin{align*}
			0&=-\bigl(\mathrm{d}_A^*a^{2,0}\bigr)^{\mskip-2mu\relax 0,1}\\
			&=\lambda\,(v_1+iv_2)\otimes_{\mathbb{C}}\sum_{j=1}^N\sum_{k=1}^2\iota_{e_{j,k}}\bigl(\nabla_{\!\bar{e}_{j,k}}\omega^+\bigr)^{\mskip-2mu\relax 1,1}\\
			&=-\EWzJhvMathFracVCentered{i\lambda}{2}\,(v_1+iv_2)\otimes_{\mathbb{C}}\sum_{j=1}^N\sum_{k=1}^2{}(\iota_{\bar{e}_{j,k}}\phi')\,\bar{e}^{j,k}
		\end{align*}
		Then $\phi'=0$. For any $(0,1)$-vector fields $X,Y$ over $U$, $\nabla_{\!X} \omega^{J_a}$ has no $(0,2)$-part and $\bigl\langle \nabla_{\!X} \omega^{J_a}, \omega^{J_a} \bigr\rangle =0$, so $\nabla_{\!X} \omega^{J_a}$ is of type~$(2,0)$ and thus $\bigl(\nabla_{\!X} J_a\bigr)Y = \bigl(\iota_Y \nabla_{\!X} \omega^{J_a}\bigr)\ifEWzJhvPDF{\rule{0pt}{1.79ex}}\fi^{\mskip-2mu\relax \sharp} =0$. Then $\bigl(\nabla_{\!X} Y\bigr){}^{\mskip-2mu\relax 1,0} = \frac{\mathrm{id}-iJ_a}{2}\nabla_{\!X} Y = \nabla_{\!X}\bigl(\frac{\mathrm{id}-iJ_a}{2}Y\bigr) =0$ and similarly $\bigl(\nabla_{\mskip-1mu\relax Y} X\bigr){}^{\mskip-2mu\relax 1,0}=0$, so $[X,Y]^{1,0} = \bigl(\nabla_{\!X} Y\bigr){}^{\mskip-2mu\relax 1,0} - \bigl(\nabla_{\mskip-1mu\relax Y} X\bigr){}^{\mskip-2mu\relax 1,0} =0$, that is, the Lie bracket of any two $(0,1)$-vector fields over $U$ is still of type~$(0,1)$. Then we can apply the Newlander--Nirenberg theorem to conclude the integ\EWzJhvText{rab}ility of $J_a$ over $U$.\qedhere\par
	\end{proof}
	
	The difficulty remains to prove that $\tilde{J}_a$ extends smoothly to the whole $M$.
	
	\subsection{Bounding \texorpdfstring{$\bigl|\nabla J_a\bigr|$}{|∇𝐽ₐ|} with higher-order derivatives}\label{subsecHigherDerivatives}
	
	We keep the assumptions and notation from the last subsection, and continue with our pointwise calculation. Recall $\phi$ from \autoref{subsecCplxStr}: there is $\phi \in \Omega^{0,1}(U)$ satisfying $\nabla_{\!X} \omega^{J_a} = (\iota_X \phi)\,\omega^+$ for any $(0,1)$-vector field $X$ over $U$. We now work towards an upper bound\EWzJhvFootnote{Jump to \autoref{corHigherDer} for the bound, also the main result of this subsection.} on $\bigl|\nabla J_a\bigr|$, or equivalently, on $\lvert\phi\rvert$. Formally, we shall bound $\bigl|\iota_{\bar{e}_{1,1}}\phi\bigr|$, which actually leads to a bound on $\lvert\phi\rvert$, as the choice of local frame $(e_{j,k})_{1 \leqslant j \leqslant N,1 \leqslant k \leqslant 2}$ is rather arbitrary. The starting point is
	\begin{EWzJhvLemma}\label{lmaIotaNabla}
		For\/~$b \in \Omega^{2,0,+}(U,V \otimes_{\mathbb{R}} \mathbb{C})$, $\iota_{\bar{e}_{1,1}}\nabla^A_{e_{1,1}} b = -\frac{i}{2}\,(\overline{\iota_{\bar{e}_{1,1}}\phi})\,\iota_{e_{1,2}}b$.
	\end{EWzJhvLemma}
	\begin{proof}
		Write $b = u \otimes_{\mathbb{C}} \omega^+$ where $u \in \Gamma(U,V \otimes_{\mathbb{R}} \mathbb{C})$ and compute:
		\[\iota_{\bar{e}_{1,1}}\nabla^A_{e_{1,1}} b = \iota_{\bar{e}_{1,1}}\nabla^A_{e_{1,1}} \bigl(u \otimes_{\mathbb{C}} \omega^+\bigr) = u \otimes_{\mathbb{C}} \iota_{\bar{e}_{1,1}}\bigl(\nabla_{\!e_{1,1}}\omega^+\bigr)^{\mskip-2mu\relax 1,1}\]\ifEWzJhvPageBreak\par\pagebreak\noindent\fi
		We compute $\bigl(\nabla_{\!e_{1,1}}\omega^+\bigr)\ifEWzJhvPDF{\rule{0pt}{1.88ex}}\fi^{\mskip-2mu\relax 1,1}$ as follows:
		\[\bigl(\nabla_{\!e_{1,1}}\omega^+\bigr)^{\mskip-2mu\relax 1,1} = \EWzJhvMathFracVCentered{1}{2N}\,\bigl\langle \nabla_{\!e_{1,1}}\omega^+, {\textstyle \omega^{J_a}} \bigr\rangle\,{\textstyle \omega^{J_a}} = -\EWzJhvMathFracVCentered{1}{2N}\,\bigl\langle \nabla_{\!e_{1,1}}{\textstyle \omega^{J_a}}, \omega^+ \bigr\rangle\,{\textstyle \omega^{J_a}} = -\EWzJhvMathFracVCentered{1}{2}\,(\overline{\iota_{\bar{e}_{1,1}}\phi})\,{\textstyle \omega^{J_a}}\]
		Then $\iota_{\bar{e}_{1,1}}\nabla^A_{e_{1,1}} b = \frac{i}{2}\,(\overline{\iota_{\bar{e}_{1,1}}\phi})\,u \otimes_{\mathbb{C}} e^{1,1}$. We then conclude.
	\end{proof}
	
	By letting $b=a^{2,0}$ in the lemma above and taking advantage of \autoref{corNablaAA20}, it would be possible to bound $\bigl|\iota_{\bar{e}_{1,1}}\phi\bigr|$ and thus $\bigl|\nabla J_a\bigr|$ with $\frac{\text{$\bigl|\nabla^A a\bigr|$}}{\lvert a\rvert}$. However, this would not be enough, \ifEWzJhvPDF\vspace{1pt}\fi as $\lvert a\rvert^{-1}$ could be too large near the zero locus of $a$. To improve the bound, we make use of the higher-order derivatives of $a$.
	
	\bigskip
	
	We use the following notation for higher-order covariant derivatives:
	\begin{equation}\label{eqHigherCovDerivative}
		\nabla^{n\vphantom{n-1}}_{\!X_1,\ldots,X_n}Y := \nabla^{\vphantom{n-1}}_{\!X_1}\bigl(\nabla^{n-1}_{\!X_2,\ldots,X_n}Y\bigr) - \sum_{j=2}^{n}\nabla^{n-1}_{\!X_2,\ldots,X_{j-1},\nabla_{\!X_1}X_j,X_{j+1},\ldots,X_n}Y
	\end{equation}
	
	\begin{EWzJhvProposition}\label{propNablasBarS0}
		For any\/ $(0,1)$-vector fields\/ $X_1,\ldots,X_n$ over\/ $U$ and any\/~$b \in \Omega^{s,0}(U)$, ${\displaystyle \nabla^n_{\!X_1,\ldots,X_n}}b \in \Omega^{s,0}$.
	\end{EWzJhvProposition}
	\begin{proof}
		For~$n=1$, this follows from \autoref{corSTFormDer}. Use \EWzJhvEqRef{eqHigherCovDerivative} to do an induction, noting $\nabla_{\!X_1}X_j$ is also of type~$(0,1)$.
	\end{proof}
	
	Recall the notation of curvatures $F_A, R$ and the notation $\EWzJhvCDotInd\circledast\EWzJhvCDotInd$ from the description before \autoref{propWeitzen2plus}. In addition, we use $\displaystyle \nabla^n_{n\EWzJhvMathTimes X}$ to denote $\displaystyle \nabla^n_{\!X,\ldots,X}$ where $X$ is repeated $n$~times. Recall from the beginning of \autoref{subsecIntegrability} that $p_{\oldstylenums{0}}$ is just some arbitrarily chosen point in $M$ with $a\rvert_{p_{\oldstylenums{0}}} \neq 0$.
	
	\begin{EWzJhvLemma}\label{lmaHigherDer}
		Suppose\EWzJhvFootnote{When $N \geqslant 2$, due to \autoref{propCurvQK1304} and \autoref{corSTFormDer}, we can simply let $s_n=0$.} $\max\limits_{0 \leqslant l \leqslant n}{}\bigl|(\nabla^l_{l\EWzJhvMathTimes\bar{e}_{1,2}}R)(e_{1,1},\bar{e}_{1,2},\bar{e}_{1,1},\bar{e}_{1,2})\bigr| \leqslant s_n$ at\/ $p_{\oldstylenums{0}}$ for some\/ $s_n \geqslant 0$ for each integer\/~$n \geqslant 0$. For any integers\/ $n,m$ satisfying\/ $0 \leqslant m \leqslant n$, \ifEWzJhvPDF\vspace{1pt}\fi the following three inequalities\EWzJhvFootnote{$C({\cdots})$ without subscript denote possibly different constants from place to place.} hold at\/ $p_{\oldstylenums{0}}$:\\[\glueexpr-\baselineskip/4\relax]
		\begingroup
			\ifEWzJhvPDF
			\newcommand*{\EWzJhvLemmaHigherDerEquationContent}[2]{
				\makebox[0pt][l]{\hspace*{-0.5\linewidth}$\displaystyle\begin{aligned}\mbox{\normalfont\Large\bfseries\textbullet}&\phantom{#1}\\&\phantom{#2}\end{aligned}$}
				\makebox[0pt]{\hspace*{-2em}$\displaystyle\begin{aligned}&#1\\\leqslant{}&#2\end{aligned}$}
				\makebox[0pt][r]{$\displaystyle\left.\rule{0pt}{8.2ex}\right\}$\hspace*{1.8em}\hspace*{-0.5\linewidth}}}
			\hbadness=99999\relax
			\else
			\newcommand*{\EWzJhvLemmaHigherDerEquationContent}[2]{
				\displaystyle\left.\begin{aligned}&#1\\\leqslant{}&#2\end{aligned}\qquad\right\}}
			\fi
			\begin{equation}\label{eqLmaHigherDer1}
				\EWzJhvLemmaHigherDerEquationContent{\biggl|(\nabla^A)^n_{n\EWzJhvMathTimes\bar{e}_{1,2}}a^{2,0} - n!\, \biggl(\EWzJhvMathFracVCentered{i}{2}\,\overline{\iota_{\bar{e}_{1,1}}\phi}\biggr)^{\mskip-4mu\relax n} a^{2,0}\biggr|}{C(n,s_{n-2}) \Biggl(\sum_{l=0}^{n-2}{}\bigl|\iota_{\bar{e}_{1,1}}\phi\bigr|^l\Biggr) \bigl|a^{2,0}\bigr|}
			\end{equation}
			\begin{equation}\label{eqLmaHigherDer2}
				\EWzJhvLemmaHigherDerEquationContent{\biggl|\iota_{\bar{e}_{1,1}}(\nabla^A)^{n+1}_{e_{1,1},n\EWzJhvMathTimes\bar{e}_{1,2}}a^{2,0} + (n+1)!\,\biggl(\EWzJhvMathFracVCentered{i}{2}\,\overline{\iota_{\bar{e}_{1,1}}\phi}\biggr)^{\mskip-4mu\relax n+1} \iota_{e_{1,2}}a^{2,0}\biggr|}{C(n,s_{n-2}) \Biggl(\sum_{l=1}^{n-1}{}\bigl|\iota_{\bar{e}_{1,1}}\phi\bigr|^l\Biggr) \bigl|a^{2,0}\bigr|}
			\end{equation}
			\begin{equation}\label{eqLmaHigherDer3}
				\EWzJhvLemmaHigherDerEquationContent{\biggl|\iota_{\bar{e}_{1,1}}(\nabla^A)^{n+1}_{m\EWzJhvMathTimes\bar{e}_{1,2},e_{1,1},(n-m)\EWzJhvMathTimes\bar{e}_{1,2}}a^{2,0} + (n+1)!\,\biggl(\EWzJhvMathFracVCentered{i}{2}\,\overline{\iota_{\bar{e}_{1,1}}\phi}\biggr)^{\mskip-4mu\relax n+1} \iota_{e_{1,2}}a^{2,0}\biggr|}{C(n,s_{n-1}) \Biggl(\sum_{l=0}^{n-1}{}\bigl|\iota_{\bar{e}_{1,1}}\phi\bigr|^l\Biggr) \bigl|a^{2,0}\bigr|}
			\end{equation}\\[\glueexpr-\baselineskip/4\relax]
		\endgroup
		where\EWzJhvFootnote{\label{ftnCIndepN}We can make $C(n,s_{n-j})$ independent of $N$, because when $N \geqslant 2$, we can simply let $s_{n-j}=0$, in which case $C(n,s_{n-j})$ can be taken to be zero, so the constants for~$N=1$ work for all $N$.} for\/~$j=1,2$, $C(n,s_{n-j})$ can be taken to be zero if\/\hspace{-0.08em} $s_{n-j}=0$.\ifEWzJhvPageBreak\pagebreak\fi
	\end{EWzJhvLemma}
	\begin{proof}
		We induce on $n$.\vspace{-0.5ex}
		\begin{enumerate}[align=left, leftmargin=0pt, labelindent=0pt, itemindent=\parindent, listparindent=\parindent, labelsep=*]
			\item \EWzJhvEqRef{eqLmaHigherDer1} for~$n=0$ is trivial.\vspace{-0.25ex}
			
			\item Let $n_{\oldstylenums{0}} \geqslant 0$. Assume \EWzJhvEqRef{eqLmaHigherDer1} holds for~$n = n_{\oldstylenums{0}}$, and \EWzJhvEqRef{eqLmaHigherDer3} holds for~$n = n_{\oldstylenums{0}} - 1$ and any~$m$ with $0 \leqslant m \leqslant n_{\oldstylenums{0}} - 1$ (if $n_{\oldstylenums{0}} \geqslant 1$). Let's prove \EWzJhvEqRef{eqLmaHigherDer2} for~$n=n_{\oldstylenums{0}}$.
			
			First,
			\begin{align*}
				&(\nabla^A)^{n_{\oldstylenums{0}}+1}_{e_{1,1},n_{\oldstylenums{0}}\EWzJhvMathTimes\bar{e}_{1,2}}a^{2,0} \\={}& \nabla^A_{e_{1,1}}(\nabla^A)^{n_{\oldstylenums{0}}}_{n_{\oldstylenums{0}}\EWzJhvMathTimes\bar{e}_{1,2}}a^{2,0} - \sum_{l=1}^{n_{\oldstylenums{0}}}{}(\nabla^A)^{n_{\oldstylenums{0}}}_{(l-1)\EWzJhvMathTimes\bar{e}_{1,2},\nabla_{\!e_{1,1}}\bar{e}_{1,2},(n_{\oldstylenums{0}}-l)\EWzJhvMathTimes\bar{e}_{1,2}}a^{2,0} \EWzJhvDisplayMathPeriod.
			\end{align*}
			When applied $\iota_{\bar{e}_{1,1}}$, the $(2,0)$-part gives zero. Due to \autoref{propNablasBarS0}, the $(0,1)$-part of\/ $\nabla_{\!e_{1,1}}\bar{e}_{1,2}$ can be ignored, that is,
			\begin{align*}
				&\iota_{\bar{e}_{1,1}}(\nabla^A)^{n_{\oldstylenums{0}}+1}_{e_{1,1},n_{\oldstylenums{0}}\EWzJhvMathTimes\bar{e}_{1,2}}a^{2,0} \\={}& \iota_{\bar{e}_{1,1}}\Biggl(\nabla^A_{e_{1,1}}(\nabla^A)^{n_{\oldstylenums{0}}}_{n_{\oldstylenums{0}}\EWzJhvMathTimes\bar{e}_{1,2}}a^{2,0} - \sum_{l=1}^{n_{\oldstylenums{0}}}{}(\nabla^A)^{n_{\oldstylenums{0}}}_{(l-1)\EWzJhvMathTimes\bar{e}_{1,2},\raisebox{0pt}[1ex]{$\scriptstyle (\nabla_{\!e_{1,1}}\bar{e}_{1,2})\ifEWzJhvPDF{\rule{0pt}{1.25ex}}\fi^{1,0}$},(n_{\oldstylenums{0}}-l)\EWzJhvMathTimes\bar{e}_{1,2}}a^{2,0}\Biggr) \EWzJhvDisplayMathPeriodOverfull.
			\end{align*}
			Note $\bigl(\nabla_{\!e_{1,1}}\bar{e}_{1,2}\bigr)\ifEWzJhvPDF{\rule{0pt}{1.88ex}}\fi^{\mskip-2mu\relax 1,0} = \frac{\mathrm{id}-iJ_a}{2}\nabla_{\!e_{1,1}}\bar{e}_{1,2} = -\bigl(\nabla_{\!e_{1,1}}\frac{\mathrm{id}-iJ_a}{2}\bigr)\bar{e}_{1,2} = \frac{i}{2}\,\bigl(\iota_{\bar{e}_{1,2}}\nabla_{\!e_{1,1}}\omega^{J_a}\bigr)\ifEWzJhvPDF{\rule{0pt}{1.79ex}}\fi^{\mskip-2mu\relax \sharp} \ifEWzJhvPDF\\\fi= -\frac{i}{2}\,(\overline{\iota_{\bar{e}_{1,1}}\phi})\,e_{1,1}$. Use \autoref{lmaIotaNabla}, \EWzJhvEqRef{eqLmaHigherDer1} for~$n = n_{\oldstylenums{0}}$, and \EWzJhvEqRef{eqLmaHigherDer3} for~$n = n_{\oldstylenums{0}} - 1$ (if $n_{\oldstylenums{0}} \geqslant 1$).
			
			\item Let $n_{\oldstylenums{0}} \geqslant 0$. Assume \EWzJhvEqRef{eqLmaHigherDer2} holds for~$n = n_{\oldstylenums{0}}$, \EWzJhvEqRef{eqLmaHigherDer1} holds for any~$n < n_{\oldstylenums{0}}$, and \EWzJhvEqRef{eqLmaHigherDer3} holds for any~$n,m$ with $0 \leqslant m \leqslant n \leqslant n_{\oldstylenums{0}} - 2$. Let's prove \EWzJhvEqRef{eqLmaHigherDer3} for~$n = n_{\oldstylenums{0}}$ and any~$m$ with $0 \leqslant m \leqslant n_{\oldstylenums{0}}$. Comparing \EWzJhvEqRef{eqLmaHigherDer2} for~$n = n_{\oldstylenums{0}}$ and \EWzJhvEqRef{eqLmaHigherDer3} for~$n = n_{\oldstylenums{0}}$, we only need to prove for any~$m$ with $0 \leqslant m < n_{\oldstylenums{0}}$,
			\begin{align*}
				&\Bigl|\iota_{\bar{e}_{1,1}}\bigl((\nabla^A)^{n_{\oldstylenums{0}}+1}_{m\EWzJhvMathTimes\bar{e}_{1,2},e_{1,1},(n_{\oldstylenums{0}}-m)\EWzJhvMathTimes\bar{e}_{1,2}} - (\nabla^A)^{n_{\oldstylenums{0}}+1}_{(m+1)\EWzJhvMathTimes\bar{e}_{1,2},e_{1,1},(n_{\oldstylenums{0}}-m-1)\EWzJhvMathTimes\bar{e}_{1,2}}\bigr)a^{2,0}\Bigr| \\&\quad\leqslant C(n_{\oldstylenums{0}},s_{n_{\oldstylenums{0}}-1}) \Biggl(\sum_{l=0}^{n_{\oldstylenums{0}}-1}\bigl|\iota_{\bar{e}_{1,1}}\phi\bigr|^l\Biggr) \bigl|a^{2,0}\bigr| \EWzJhvDisplayMathPeriod.
			\end{align*}
			\indent We must be careful with the curvature. For any (complex) vector fields $X,Y$ and $Z_1,\ldots,Z_{n_{\oldstylenums{0}}-m-1}$, we have\EWzJhvFootnote{The minus or plus sign depends on your convention of the sign of $R$.}
			\begin{align*}
				&\bigl((\nabla^A)^{n_{\oldstylenums{0}}-m+1}_{\mskip-2mu\relax X,Y,Z_1,\ldots,Z_{n_{\oldstylenums{0}}-m-1}} - (\nabla^A)^{n_{\oldstylenums{0}}-m+1}_{Y,X,Z_1,\ldots,Z_{n_{\oldstylenums{0}}-m-1}}\bigr)a^{2,0}\\
				={}& \bigl((\nabla^A)^2_{\mskip-2mu\relax X,Y} - (\nabla^A)^2_{Y,X}\bigr)(\nabla^A)^{n_{\oldstylenums{0}}-m-1}_{Z_1,\ldots,Z_{n_{\oldstylenums{0}}-m-1}}a^{2,0} \\&- \sum_{t=1}^{n_{\oldstylenums{0}}-m-1}(\nabla^A)^{n_{\oldstylenums{0}}-m-1}_{Z_1,\ldots,Z_{t-1},(\nabla^2_{\!X,Y} - \nabla^2_{\mskip-1mu\relax Y,X})Z_t,Z_{t+1},\ldots,Z_{n_{\oldstylenums{0}}-m-1}}a^{2,0}\\
				={}& \bigl((F_A \oplus R)(X,Y)\bigr)\circledast(\nabla^A)^{n_{\oldstylenums{0}}-m-1}_{Z_1,\ldots,Z_{n_{\oldstylenums{0}}-m-1}}a^{2,0} \\&\mp \sum_{t=1}^{n_{\oldstylenums{0}}-m-1}(\nabla^A)^{n_{\oldstylenums{0}}-m-1}_{Z_1,\ldots,Z_{t-1},R(X,Y)Z_t,Z_{t+1},\ldots,Z_{n_{\oldstylenums{0}}-m-1}}a^{2,0} \EWzJhvDisplayMathPeriod.
			\end{align*}
			We apply $\iota_{\bar{e}_{1,1}}{\displaystyle ({\textstyle \nabla^A})^m_{m\EWzJhvMathTimes\bar{e}_{1,2}}}$ to this equation of tensors with parameters $X,Y$ and $Z_1,\ldots,Z_{n_{\oldstylenums{0}}-m-1}$, and then plug in $X:=e_{1,1}$ and $Y,Z_1,\ldots,Z_{n_{\oldstylenums{0}}-m-1}:=\bar{e}_{1,2}$. Note that $(2,0)$-forms vanish when applied $\iota_{\bar{e}_{1,1}}$. We get\ifEWzJhvPageBreak\pagebreak\fi
			\begin{align*}
				&\iota_{\bar{e}_{1,1}}\bigl((\nabla^A)^{n_{\oldstylenums{0}}+1}_{m\EWzJhvMathTimes\bar{e}_{1,2},e_{1,1},(n_{\oldstylenums{0}}-m)\EWzJhvMathTimes\bar{e}_{1,2}} - (\nabla^A)^{n_{\oldstylenums{0}}+1}_{(m+1)\EWzJhvMathTimes\bar{e}_{1,2},e_{1,1},(n_{\oldstylenums{0}}-m-1)\EWzJhvMathTimes\bar{e}_{1,2}}\bigr)a^{2,0}\\
				&\quad
				\begin{aligned}
					={}&\iota_{\bar{e}_{1,1}}\sum_{l=0}^m\binom{m}{l}\Biggl(\bigl((\nabla^l_{l\EWzJhvMathTimes\bar{e}_{1,2}} R)(e_{1,1},\bar{e}_{1,2})\bigr)^{\mskip-2mu\relax 0,2}\circledast(\nabla^A)^{n_{\oldstylenums{0}}-l-1}_{(n_{\oldstylenums{0}}-l-1)\EWzJhvMathTimes\bar{e}_{1,2}}a^{2,0} \\&\!\mp\ifEWzJhvPDF\hspace*{-2.6pt}\fi\sum_{t=1}^{n_{\oldstylenums{0}}-m-1}\ifEWzJhvPDF\hspace*{-2.6pt}\fi(\nabla^A)^{n_{\oldstylenums{0}}-l-1}_{(m-l+t-1)\EWzJhvMathTimes\bar{e}_{1,2},\raisebox{0pt}[1ex]{$\scriptstyle \mathopen{\text{$\bigl($}}(\nabla^l_{l\EWzJhvMathTimes\bar{e}_{1,2}} R)(e_{1,1},\bar{e}_{1,2})\bar{e}_{1,2}\mathclose{\text{$\bigr)$}}^{\!1,0}$},(n_{\oldstylenums{0}}-m-t-1)\EWzJhvMathTimes\bar{e}_{1,2}}a^{2,0}\Biggr)
				\end{aligned}
			\end{align*}
			where $\bigl((\nabla^l_{l\EWzJhvMathTimes\bar{e}_{1,2}} R)(e_{1,1},\bar{e}_{1,2})\bigr)\ifEWzJhvPDF{\rule{0pt}{1.88ex}}\fi^{\mskip-2mu\relax 0,2}$ is the $(0,2)$-part of $(\nabla^l_{l\EWzJhvMathTimes\bar{e}_{1,2}} R)(e_{1,1},\bar{e}_{1,2})$ regarded as a $2$-form, whose norm is either zero (if $N \geqslant 2$, due to \autoref{propCurvQK1304} and \autoref{corSTFormDer}) or bounded by $s_{n_{\oldstylenums{0}}-1}$ (if $N = 1$). Similarly, $\bigl((\nabla^l_{l\EWzJhvMathTimes\bar{e}_{1,2}} R)(e_{1,1},\bar{e}_{1,2})\bar{e}_{1,2}\bigr)\ifEWzJhvPDF{\rule{0pt}{1.88ex}}\fi^{\!1,0}$ is either zero (if $N \geqslant 2$) or some complex number times $e_{1,1}$ with norm bounded by $s_{n_{\oldstylenums{0}}-1}$ (if $N = 1$). Use \EWzJhvEqRef{eqLmaHigherDer1} for~$n=n_{\oldstylenums{0}}-l-1$ and \EWzJhvEqRef{eqLmaHigherDer3} for~$n=n_{\oldstylenums{0}}-l-2$ (if\/ $l \leqslant n_{\oldstylenums{0}} - 2$) to obtain the bound that we need.
			
			\item Let $n_{\oldstylenums{0}} \geqslant 1$. Assume \EWzJhvEqRef{eqLmaHigherDer3} holds for~$n=n_{\oldstylenums{0}}-1$ and any~$m$ with $0 \leqslant m \leqslant n_{\oldstylenums{0}} - 1$. Let's prove \EWzJhvEqRef{eqLmaHigherDer1} for~$n=n_{\oldstylenums{0}}$.
			
			We use the equation ${\displaystyle \mathrm{d}_A^*}a^{2,0}=0$. Because $\displaystyle -\mathrm{d}_A^*$ is a contraction of the total covariant derivative, applying $(\nabla^A)^{n_{\oldstylenums{0}}-1}_{(n_{\oldstylenums{0}}-1)\EWzJhvMathTimes\bar{e}_{1,2}}$, we get
			\[\sum_{j=1}^N\sum_{k=1}^2{}\bigl(\iota_{e_{j,k}}(\nabla^A)^{n_{\oldstylenums{0}}}_{(n_{\oldstylenums{0}}-1)\EWzJhvMathTimes\bar{e}_{1,2},\bar{e}_{j,k}} + \iota_{\bar{e}_{j,k}}(\nabla^A)^{n_{\oldstylenums{0}}}_{(n_{\oldstylenums{0}}-1)\EWzJhvMathTimes\bar{e}_{1,2},e_{j,k}}\bigr)a^{2,0}=0 \EWzJhvDisplayMathPeriod.\]
			Note that any (high-order) covariant derivative of $a^{2,0}$ is a section of\/ $\Omega^{2,+}(U,V \otimes_{\mathbb{R}} \mathbb{C})$. Therefore, when applied $\iota_{e_{1,1}}$, terms not beginning with $\iota_{e_{1,2}}$ or $\iota_{\bar{e}_{1,1}}$ vanish, so
			\[\bigl(\iota_{e_{1,1}}\iota_{e_{1,2}}(\nabla^A)^{n_{\oldstylenums{0}}}_{n_{\oldstylenums{0}}\EWzJhvMathTimes\bar{e}_{1,2}} + \iota_{e_{1,1}}\iota_{\bar{e}_{1,1}}(\nabla^A)^{n_{\oldstylenums{0}}}_{(n_{\oldstylenums{0}}-1)\EWzJhvMathTimes\bar{e}_{1,2},e_{1,1}}\bigr)a^{2,0}=0 \EWzJhvDisplayMathPeriod.\]
			Use \EWzJhvEqRef{eqLmaHigherDer3} for~$n=n_{\oldstylenums{0}}-1$ to get
			\begin{align*}
				&\biggl|\iota_{e_{1,1}}\iota_{e_{1,2}}(\nabla^A)^{n_{\oldstylenums{0}}}_{n_{\oldstylenums{0}}\EWzJhvMathTimes\bar{e}_{1,2}}a^{2,0} - n_{\oldstylenums{0}}!\,\biggl(\EWzJhvMathFracVCentered{i}{2}\,\overline{\iota_{\bar{e}_{1,1}}\phi}\biggr)^{\mskip-4mu\relax n_{\oldstylenums{0}}} \iota_{e_{1,1}}\iota_{e_{1,2}}a^{2,0}\biggr| \\\leqslant{}& C(n_{\oldstylenums{0}},s_{n_{\oldstylenums{0}}-2}) \Biggl(\sum_{l=0}^{n_{\oldstylenums{0}}-2}\bigl|\iota_{\bar{e}_{1,1}}\phi\bigr|^l\Biggr) \bigl|a^{2,0}\bigr| \EWzJhvDisplayMathPeriod.
			\end{align*}
			Note that both $(\nabla^A)^{n_{\oldstylenums{0}}}_{n_{\oldstylenums{0}}\EWzJhvMathTimes\bar{e}_{1,2}}a^{2,0}$ and $a^{2,0}$ are sections of\/ $\Omega^{2,0,+}(U,V \otimes_{\mathbb{R}} \mathbb{C})$, so we can remove both $\iota_{e_{1,1}}\iota_{e_{1,2}}$ and get the conclusion.\qedhere
		\end{enumerate}
	\end{proof}
	
	We finally get the bound:
	\begin{EWzJhvCorollary}\label{corHigherDer}
		Let\/ $s_n$ be as in \autoref{lmaHigherDer}. For any integer\/~$n \geqslant 1$, we have at\/ $p_{\oldstylenums{0}}$
		\[\bigl|\iota_{\bar{e}_{1,1}}\phi\bigr| \leqslant \max\left\{C(n,s_{n-3}),\: C(n)\Biggl(\frac{\bigl|\iota_{\bar{e}_{1,1}}{\displaystyle ({\textstyle \nabla^A})^n_{e_{1,1},(n-1)\EWzJhvMathTimes\bar{e}_{1,2}}}a^{2,0}\bigr|}{\bigl|\iota_{e_{1,2}}a^{2,0}\bigr|}\Biggr)^{\mskip-6mu\relax \frac{1}{n}}\right\}\]
		where\/ $C(n,s_{n-3})$ can be taken to be zero if\/\hspace{-0.08em} $s_{n-3}=0$.
		
		\medskip
		
		To put it more simply, it is possible to choose the local frame\/ $(e_{j,k})_{1 \leqslant j \leqslant N,1 \leqslant k \leqslant 2}$ so that\/ $\phi$ equals some complex number times\/ $\bar{e}^{1,1}$ at\/ $p_{\oldstylenums{0}}$, in which case we have at\/ $p_{\oldstylenums{0}}$
		\[\bigl|\nabla J_a\bigr| \leqslant \max\left\{C(n,s_{n-3}),\: C(N,n)\Biggl(\frac{\bigl|{\displaystyle ({\textstyle \nabla^A})^n} a^{2,0}\bigr|}{\bigl|a^{2,0}\bigr|}\Biggr)^{\mskip-6mu\relax \frac{1}{n}}\right\}\]
		where\EWzJhvFootnote{See the footnote\EWzJhvFootRef{ftnCIndepN} to \autoref{lmaHigherDer} on why unlike $C(N,n)$, we can make $C(n,s_{n-3})$ independent \ifEWzJhvPDF\nolinebreak\fi of \ifEWzJhvPDF\nolinebreak\fi $N$.} $C(n,s_{n-3})$ can be taken to be zero if\/\hspace{-0.08em} $s_{n-3}=0$.\ifEWzJhvPageBreak\pagebreak\fi
	\end{EWzJhvCorollary}
	\begin{proof}
		\autoref{lmaHigherDer}~\EWzJhvEqRef{eqLmaHigherDer2} for $n-1$ implies
		\[\bigl|\iota_{\bar{e}_{1,1}}\phi\bigr|^n \leqslant C_1(n)\,\frac{\bigl|\iota_{\bar{e}_{1,1}}{\displaystyle ({\textstyle \nabla^A})^n_{e_{1,1},(n-1)\EWzJhvMathTimes\bar{e}_{1,2}}}a^{2,0}\bigr|}{\bigl|\iota_{e_{1,2}}a^{2,0}\bigr|} + C_2(n,s_{n-3})\sum_{l=1}^{n-2}{}\bigl|\iota_{\bar{e}_{1,1}}\phi\bigr|^l \EWzJhvDisplayMathPeriod.\]
		Let $C_3(n,s_{n-3}) \geqslant 0$ satisfy for any~$t \geqslant C_3(n,s_{n-3})$, $C_2(n,s_{n-3})\sum_{l=1}^{n-2}t^l \leqslant \frac{1}{2}t^n$. Then
		\[\bigl|\iota_{\bar{e}_{1,1}}\phi\bigr| \leqslant \max\left\{C_3(n,s_{n-3}),\: \Biggl(2\,C_1(n)\,\frac{\bigl|\iota_{\bar{e}_{1,1}}{\displaystyle ({\textstyle \nabla^A})^n_{e_{1,1},(n-1)\EWzJhvMathTimes\bar{e}_{1,2}}}a^{2,0}\bigr|}{\bigl|\iota_{e_{1,2}}a^{2,0}\bigr|}\Biggr)^{\mskip-6mu\relax \frac{1}{n}}\right\} \EWzJhvDisplayMathPeriod.\qedhere\]
	\end{proof}
	
	\subsubsection*{Using the bound to prove \autoref{thm2}}
	
	The term
	\begingroup
	\ifEWzJhvPDF
	\newcommand*{\EWzJhvPDFSmallHeight}[1]{\raisebox{0pt}[1ex]{#1}}%
	\else
	\newcommand*{\EWzJhvPDFSmallHeight}[1]{#1}%
	\fi
	\EWzJhvPDFSmallHeight{$\bigl|a^{2,0}\bigr|^{-\frac{1}{n}}$}%
	\endgroup
	\space can be estimated with tools from \autoref{subsecFMinus1n}.
	
	To estimate the term $\bigl|{\displaystyle ({\textstyle \nabla^A})^n} a^{2,0}\bigr|$ near the zero locus of $a$, it might be tempting to try to relate $\bigl|{\displaystyle ({\textstyle \nabla^A})^n} a^{2,0}\bigr|$ to $\bigl|{\displaystyle ({\textstyle \nabla^A})^n} a\bigr|$. This works for~$n=1$ (\autoref{corNablaAA20}). This would also work if\/ $V$ is of rank~$2$, because in this case, the inner product on $V$ and a (local) orientation on $V$ determines a (local) linear complex structure on $V$ denoted by $J_V$, $\nabla^A J_V = 0$, and $a^{2,0}=\frac{\mathrm{id} \pm iJ_V}{2}a$. In general, however, it is not clear how to bound $\bigl|{\displaystyle ({\textstyle \nabla^A})^n} a^{2,0}\bigr|$ with $\bigl|{\displaystyle ({\textstyle \nabla^A})^n} a\bigr|$. Our strategy is instead to use elliptic regularity to bound $\bigl|{\displaystyle ({\textstyle \nabla^A})^n} a^{2,0}\bigr|$, as ${\displaystyle \mathrm{d}_A^*}a^{2,0}=0$. However, to do this near the zero locus of $a$, we must first prove that $a^{2,0}$ can be defined near the zero locus of $a$. This is the topic of \autoref{subsecA20ZeroLocus}.
	
	Afterwards, we prove that $\tilde{J}_a$ is a $W^{1,p}_{\mathrm{loc}}$ tensor over $M$ and then the final conclusion in \autoref{subsecFinishUp}.
	
	\subsection[\texorpdfstring{$a^{2,0}$}{𝑎²ʼ⁰} near zero locus]{$a^{2,0}$ near zero locus\EWzJhvFootnote{Jump to \autoref{lmaZeroLocusDim},~\ref{lmaJaZ2Choice} for the main results of this subsection.\EWzJhvPar Somehow some proof steps in this subsection require smoothness. For \autoref{lmaZeroLocusDim}, we might expect to be able to prove a variant with weaker assumption (e.g.\ $C^{100}$) using some more sophisticated techniques. For \autoref{lmaJaZ2Choice}, on the other hand, the reliance on the Taylor expansion in the proof feels quite essential.}}\label{subsecA20ZeroLocus}
	
	Under the assumptions of \autoref{thm2}, we would like to define $a^{2,0}$ (see \EWzJhvEqRef{eqA20}) locally on $M$, even near the zero locus of $a$. The major difficulty is that $\tilde{J}_a$ is only a $\mathbb{Z}_2$-complex structure on the complement of the zero locus of $a$, so we must make a choice to get a complex structure $J_a$ to define $a^{2,0}$, and it is not obvious whether this choice can be made consistently near the zero locus.
	
	It is helpful to study the zero locus of $a$ first. Because $a$ satisfies $\displaystyle \mathrm{d}_A^*a=0$ and is not identically zero, presu\EWzJhvText{mab}ly the results from \cite{ZeroSetsSolutionsElliptic} should apply and indicate that the zero locus of $a$ is ``countably $(4N-2)$-$C^\infty$-rectifiable''. Here, however, we will not apply those results directly, as we want to avoid the subtle problem of attempting to draw an equivalence between the existence of a common \emph{real} root of two polynomials and the vanishing of their resultant, which \cite{ZeroSetsSolutionsElliptic} seems to do in its ``Proof of Main Theorem'' on page~197. Instead, we will try to apply the proof techniques from \cite{ZeroSetsSolutionsElliptic} to our specific problem. We will see after some work that, in our specific setting, there is a way to work around this subtle problem with \cite{ZeroSetsSolutionsElliptic} and prove that the zero locus of $a$ is countably $(4N-2)$-$C^\infty$-rectifiable.\EWzJhvFootnote{Our \autoref{lmaResultant} below that shows the resultant is not identically zero only applies in our specific setting. It would be interesting to know how to fix this subtle problem with \cite{ZeroSetsSolutionsElliptic} in its original general setting. For common Laplacian-related first-order elliptic operators like ours, the proof of \cite[Proposition~2.2]{CritSetSolElli} can help to rule out nonconstant nonnegative common divisors. This should provide a mostly satisfactory solution, but may not be appli\EWzJhvText{cab}le to general ``uncommon'' first-order elliptic operators.} By working along the lines of \cite{ZeroSetsSolutionsElliptic}, we will also be able to prove that the ``tangent cone'' of the zero locus of $a$ is countably $(8N-4)$-$C^\infty$-rectifiable, which will be useful to us.
	
	We first do some preparation. The following lemma can be regarded as a special case of \autoref{thm2}.
	\begin{EWzJhvLemma}\label{lmaEuclideanQK4}
		Let\/ $(M,g,Q)$ be the standard Euclidean quaternion-Kähler space\/ $\mathbb{R}^4$. Let\/ $V \rightarrow \mathbb{R}^4$ be a trivial real vector bundle with inner product. Identify\/ $V$ with\/ $\mathtt{V}\times\mathbb{R}^4$ where\/ $\mathtt{V}$ is a finite-dimensional real inner product space. Denote by\/ $\Lambda^{2,+}$ the fibre of\/\hspace{0em plus -0.08em} $\Omega^{2,+}$ over\/ $\mathbb{R}^4$. Denote by\/ $\operatorname{Sym}_{\mathbb{R}}^d\bigl({\displaystyle ({\textstyle \mathbb{R}^4})^*}\bigr)$ the space of real homogeneous polynomial functions \ifEWzJhvPDF\linebreak\\*[-\baselineskip]\fi of degree\/~$d$ on\/ $\mathbb{R}^4$ for\/~$d \geqslant 0$. Let\/ $a \in \operatorname{Sym}_{\mathbb{R}}^d\bigl({\displaystyle ({\textstyle \mathbb{R}^4})^*}\bigr)\otimes_{\mathbb{R}}\mathtt{V}\otimes_{\mathbb{R}}\Lambda^{2,+} \subset \Omega^{2,+}(V)$ be a parascalar\/ $(2,0)$-form. If\/\hspace{0em plus -0.08em} $\displaystyle \mathrm{d}^*a=0$ and\/ $a$ is not the zero section, then the\/ $\mathbb{Z}_2$-complex structure\/ $\tilde{J}_a$ defined on the complement of the zero locus of\/\hspace{-0.08em} $a$ must come from some constant section of\/\hspace{0em plus -0.08em} $Q$ over\/ $\mathbb{R}^4$.
	\end{EWzJhvLemma}
	\begin{proof}
		The case $d=0$ is trivial.\vspace{-0.5pt}
		
		\begingroup
		\linespread{1.08}\selectfont
		Suppose $d>0$. Let $p_{\oldstylenums{0}} \in \mathbb{R}^4$ satisfy $a\rvert_{p_{\oldstylenums{0}}} \neq 0$. Then $p_{\oldstylenums{0}} \neq 0$. We work near $p_{\oldstylenums{0}}$. We use the notation from \autoref{subsecIntegrability} and make a choice of $J_a$ and the local frame $(e_{1,1}\,\,e_{1,2})$ such that \ifEWzJhvPDF\vspace{-2pt}\fi $e_{1,2}\rvert_p=\frac{\mathrm{id}-iJ_a}{\sqrt{2}\mskip2mu\relax\lvert p\rvert\,}p$ for $p$ near $p_{\oldstylenums{0}}$. Let $X$ be the vector field $X\rvert_p:=p$ for $p$ near $p_{\oldstylenums{0}}$. Because $a$ is a homogeneous function, \ifEWzJhvPDF\vspace{1pt}\fi $\nabla_{\!X} J_a=0$ and thus $\nabla_{\!e_{1,2}+\bar{e}_{1,2}} \omega^{J_a} = 0$. We know from \hyperlink{nablaOmegaJ}{the discussions} following \autoref{corSTFormDer} that $\nabla_{\!e_{1,2}} \omega^{J_a} \in \Omega^{0,2}$ and $\nabla_{\!\bar{e}_{1,2}} \omega^{J_a} \in \Omega^{2,0}$, so $\nabla_{\!e_{1,2}} \omega^{J_a} = \nabla_{\!\bar{e}_{1,2}} \omega^{J_a} = 0$. Noting $\nabla_{\!\bar{e}_{1,2}}J_a = 0$, we have
		\begin{align*}
			\nabla_{\!\bar{e}_{1,2}}\bar{e}_{1,2} &= \nabla_{\!\bar{e}_{1,2}}\biggl(\frac{\EWzJhvMathVCenter{\mathrm{id}+iJ_a}}{\sqrt{2}\,}\frac{\EWzJhvMathVCenter{X\vphantom{J_a}}}{\EWzJhvlvertRaised X\EWzJhvrvertRaised\vphantom{\sqrt{2}}}\biggr) = \frac{\EWzJhvMathVCenter{\mathrm{id}+iJ_a}}{\sqrt{2}\,}\,\nabla_{\!\bar{e}_{1,2}}\frac{\EWzJhvMathVCenter{X\vphantom{J_a}}}{\EWzJhvlvertRaised X\EWzJhvrvertRaised\vphantom{\sqrt{2}}}\\
			&= \frac{\mathrm{id}+iJ_a}{\sqrt{2}\,}\biggl(\frac{\nabla_{\!\bar{e}_{1,2}}X}{\EWzJhvlvertRaised X\EWzJhvrvertRaised\vphantom{\sqrt{2}}}+\Bigl(\nabla_{\!\bar{e}_{1,2}}\frac{\EWzJhvMathVCenter 1}{\EWzJhvlvertRaised X\EWzJhvrvertRaised\vphantom{\sqrt{2}}}\Bigr)X\biggr)\\
			&= \frac{\mathrm{id}+iJ_a}{\sqrt{2}\,}\biggl(\frac{\bar{e}_{1,2}}{\EWzJhvlvertRaised X\EWzJhvrvertRaised\vphantom{\sqrt{2}}}-\frac{\langle X, \bar{e}_{1,2} \rangle}{\EWzJhvlvertRaised X\EWzJhvrvertRaised{\vphantom{X}}^3\vphantom{\sqrt{2}}}X\biggr)\\
			&= \frac{\mathrm{id}+iJ_a}{\sqrt{2}\,}\biggl(\frac{\bar{e}_{1,2}}{\EWzJhvlvertRaised X\EWzJhvrvertRaised\vphantom{\sqrt{2}}}-\frac{\EWzJhvMathVCenter X}{\sqrt{2}\,\EWzJhvlvertRaised X\EWzJhvrvertRaised{\vphantom{X}}^2}\biggr)\\
			&= \frac{2\,\bar{e}_{1,2}}{\sqrt{2}\,\EWzJhvlvertRaised X\EWzJhvrvertRaised}-\frac{\bar{e}_{1,2}}{\sqrt{2}\,\EWzJhvlvertRaised X\EWzJhvrvertRaised}\\
			&= \frac{\bar{e}_{1,2}}{\sqrt{2}\,\EWzJhvlvertRaised X\EWzJhvrvertRaised \EWzJhvDisplayMathPeriod{\,.}}
		\end{align*}
		Therefore,\ifEWzJhvPDF\vspace{\glueexpr-\baselineskip/2\relax}\fi
		\begin{align*}
			&\nabla^{n+1}_{\makebox[-\width]{\hphantom{$\scriptstyle($}}(n+1)\EWzJhvMathTimes\bar{e}_{1,2}}J_a\\
			={}& \nabla_{\!\bar{e}_{1,2}}\nabla^n_{n\EWzJhvMathTimes\bar{e}_{1,2}}J_a - \sum_{k=1}^n \nabla^n_{\!(k-1)\EWzJhvMathTimes\bar{e}_{1,2},\nabla_{\!\bar{e}_{1,2}}\bar{e}_{1,2},(n-k)\EWzJhvMathTimes\bar{e}_{1,2}}J_a\\
			={}& \nabla_{\!\bar{e}_{1,2}}\nabla^n_{n\EWzJhvMathTimes\bar{e}_{1,2}}J_a - \EWzJhvMathFracVCentered{n}{\sqrt{2}\,\EWzJhvlvertRaised X\EWzJhvrvertRaised}\nabla^n_{n\EWzJhvMathTimes\bar{e}_{1,2}}J_a \EWzJhvDisplayMathPeriod.
		\end{align*}
		An argument by induction then shows $\nabla^n_{n\EWzJhvMathTimes\bar{e}_{1,2}}J_a=0$ for any~$n \geqslant 1$. Because $a$ is \ifEWzJhvPDF\linebreak\fi a vector-valued polynomial function of degree~$d$, \ifEWzJhvPDF\vspace{-1pt}\fi $\nabla^{d+1}_{\makebox[-\width]{\hphantom{$\scriptstyle($}}(d+1)\EWzJhvMathTimes\bar{e}_{1,2}}a=0$. Using \EWzJhvEqRef{eqA20}, we deduce $\nabla^{d+1}_{\makebox[-\width]{\hphantom{$\scriptstyle($}}(d+1)\EWzJhvMathTimes\bar{e}_{1,2}}a^{2,0}=0$. \autoref{lmaHigherDer}~\EWzJhvEqRef{eqLmaHigherDer1} then tells us $\iota_{\bar{e}_{1,1}}\phi=0$ and thus $\nabla_{\!\bar{e}_{1,1}} \omega^{J_a}=0$. We have already shown \ifEWzJhvPDF\vspace{0.5pt}\fi $\nabla_{\!\bar{e}_{1,2}}J_a = 0$, so $\nabla J_a=0$ near $p_{\oldstylenums{0}}$. This implies the vector-valued polynomial function $\bigl\langle a,\: \omega^{J_a}\rvert_{p_{\oldstylenums{0}}} \bigr\rangle$ vanishes near $p_{\oldstylenums{0}}$, so $\bigl\langle a,\: \omega^{J_a}\rvert_{p_{\oldstylenums{0}}} \bigr\rangle$ must vanish over the whole $\mathbb{R}^4$. As a result, $\tilde{J}_a$ comes from $J_a\rvert_{p_{\oldstylenums{0}}}$ over the whole complement of the zero locus of $a$ in $\mathbb{R}^4$.\qedhere\par
		\endgroup\ifEWzJhvPageBreak\pagebreak\fi
	\end{proof}
	
	Back to the setting of \autoref{thm2}, let $p \in M$ be a zero point of $a$. Due to the \ifEWzJhvPDF\linebreak\fi Weitzenböck formula {\small (\autoref{propWeitzen2plus})}, the equation $\displaystyle \mathrm{d}_A^* a = 0$ implies the equation \ifEWzJhvPDF\linebreak\fi $\frac{1}{2N}\bigl(\nabla^A\bigr)^{\mskip-5mu\relax *} \nabla^A a + (F_A \oplus R) \circledast a =0$, which then implies $a$ vanishes at $p$ of some finite order~$d$, that is, the $d$-th order Taylor expansion of $a$ at $p$ is nonzero, and any lower-order expansion vanishes (see e.g.\ \cite{UniqueContinuationGeometry}). Denote the $d$-th order Taylor expansion of $a$ at $p$ by $a_p \in \operatorname{Sym}_{\mathbb{R}}^d(\EWzJhvTpStar Tp M) \otimes_{\mathbb{R}} \Omega^{2,+}(V)\rvert_p$. We now study the resultant, as \cite{ZeroSetsSolutionsElliptic} suggests.
	
	\begin{EWzJhvLemma}\label{lmaResultant}
		Let\/ $a$ vanish at\/ $p \in M$ of some finite order\/ $d \geqslant 1$. Then there \ifEWzJhvPDF\linebreak\fi exist\/ $\nu_1,\nu_2 \in \Omega^{2,+}(V)\rvert_p$ \ifEWzJhvPDF\vspace{-1pt}\fi and an orthonormal basis\/ $(\mathsf{e}_j)_{j=1}^{4N}$ of\/ $T_pM$ such that, letting\/ $f_k(x_1,\ldots,x_{4N}):=\bigl\langle a_p\bigl(\sum_{j=1}^{4N} x_j\mathsf{e}_j\bigr),\: \nu_k \bigr\rangle$ \ifEWzJhvPDF\vspace{1pt}\fi for\/~$k=1,2$, regarding\/ $f_1,f_2$ as univariate polynomials of degree\/~$d$ with indeterminate\/~$x_1$ and coefficients in\/ $\mathbb{R}[x_2,\ldots,x_{4N}]$, the resultant\/ $\operatorname{res}_{x_1}^d(f_1,f_2) \in \mathbb{R}[x_2,\ldots,x_{4N}]$ is a nonzero polynomial.
	\end{EWzJhvLemma}
	\begin{proof}
		An appropriate identification of\/ $Q\rvert_p$ with the space of imaginary quaternions yields a left $\mathbb{H}$-module structure on $T_pM$, which makes $T_pM$ an ``Euclidean quaternion-Kähler space'', over which $a_p \in \operatorname{Sym}_{\mathbb{R}}^d(\EWzJhvTpStar Tp M) \otimes_{\mathbb{R}} \Omega^{2,+}(V)\rvert_p \subset \Omega^{2,+}(T_pM, V\rvert_p)$ is a parascalar $(2,0)$-form and satisfies $\displaystyle \mathrm{d}^*a_p=0$ as implied by the Taylor expansion of the equation $\displaystyle \mathrm{d}_A^* a = 0$ at $p$. Let $E \subset T_pM$ be an $\mathbb{H}$-invariant subspace of real dimension~$4$. Denote by $a_p\rvert_E$ the restriction of $a_p$ as a $V\rvert_p$-valued $2$-form on $T_pM$ to $E$. By looking at \EWzJhvEqRef{eqOrthonBasis2plus}, we see $a_p\rvert_E \in \operatorname{Sym}_{\mathbb{R}}^d({\displaystyle E^*}) \otimes_{\mathbb{R}} V\rvert_p \otimes_{\mathbb{R}} \Lambda^{2,+}_E \subset \Omega^{2,+}(E, V\rvert_p)$ is also a parascalar $(2,0)$-form and satisfies $\displaystyle \mathrm{d}^*(a_p\rvert_E)=0$ over $E$. \autoref{lmaEuclideanQK4} applies to $a_p\rvert_E$ and implies that there exists some compatible complex structure $J_p \in Q\rvert_p$ such that the $\mathbb{Z}_2$-complex structure induced by $a_p\rvert_E$ on the complement of the zero locus of $a_p\rvert_E$ in $E$ comes from $J_p\rvert_E$. Let $(z^1\,\,z^2)$ be some complex coordinates on $(E,J_p\rvert_E)$ that identify $(E,J_p\rvert_E,g\rvert_E)$ with $(\mathbb{C}^2,g_{\mathbb{C}^{2\mathstrut}})$. Similar to \EWzJhvEqRef{eqACoord}, write
		\begin{align*}
			a_p\rvert_E &= \operatorname{Re}(v)\otimes_{\mathbb{R}} \bigl(\mathrm{d}z^1 \wedge \mathrm{d}z^2 + \mathrm{d}\bar{z}^1 \wedge \mathrm{d}\bar{z}^2\bigr) + \operatorname{Im}(v)\otimes_{\mathbb{R}} i\hspace{0.0625em}\bigl(\mathrm{d}z^1 \wedge \mathrm{d}z^2 - \mathrm{d}\bar{z}^1 \wedge \mathrm{d}\bar{z}^2\bigr)\\&=v\otimes_{\mathbb{C}}(\mathrm{d}z^1 \wedge \mathrm{d}z^2) + \bar{v}\otimes_{\mathbb{C}}(\mathrm{d}\bar{z}^1 \wedge \mathrm{d}\bar{z}^2)
		\end{align*}
		where $v \in \operatorname{Sym}_{\mathbb{R}}^d({\displaystyle E^*}) \otimes_{\mathbb{R}} V\rvert_p \otimes_{\mathbb{R}} \mathbb{C}$ and $\operatorname{Re}(v), \operatorname{Im}(v)$ have equal norm everywhere on $E$. The equation $\displaystyle \mathrm{d}^*(a_p\rvert_E)=0$ now means that $v$ is holomorphic over $(E,J_p\rvert_E)$, and thus $v \in \operatorname{Sym}_{\mathbb{C}}^d\bigl({\displaystyle (E,J_p\rvert_E)^*}\bigr) \otimes_{\mathbb{R}} V\rvert_p$. Choose some $\nu \in V\rvert_p$ satisfying $\langle v, \nu \rangle \in \operatorname{Sym}_{\mathbb{C}}^d\bigl({\displaystyle (E,J_p\rvert_E)^*}\bigr)$ is a nonzero polynomial. Now, we let
		\[
		\nu_1:=\nu\otimes_{\mathbb{R}} \sigma_{\mskip-2mu\relax E}\bigl(\mathrm{d}z^1 \wedge \mathrm{d}z^2 + \mathrm{d}\bar{z}^1 \wedge \mathrm{d}\bar{z}^2\bigr),\quad
		\nu_2:=\nu\otimes_{\mathbb{R}} \sigma_{\mskip-2mu\relax E}\bigl(i\hspace{0.0625em}\bigl(\mathrm{d}z^1 \wedge \mathrm{d}z^2 - \mathrm{d}\bar{z}^1 \wedge \mathrm{d}\bar{z}^2\bigr)\bigr)
		\]
		where $\sigma_{\mskip-2mu\relax E}$ is the unique linear map from $\Lambda^{2,+}_E$ to $\Omega^{2,+}(M)\bigr|_p$ \ifEWzJhvPDF\vspace{-1pt}\fi satisfying $\sigma_{\mskip-2mu\relax E}(\omega)\bigr|_E=\omega$ for any~$\omega\in\Lambda^{2,+}_E$. Choose \ifEWzJhvPDF\vspace{1pt}\fi an orthonormal basis $(\mathsf{e}_j)_{j=1}^{4N}$ of\/ $T_pM$ satisfying $(\mathsf{e}_1\,\,\mathsf{e}_2\,\,\mathsf{e}_3\,\,\mathsf{e}_4)$ is an orthonormal basis of $E$ and
		\[
		\mathsf{e}_2=J_p\mathsf{e}_1,\quad
		\mathsf{e}_4=J_p\mathsf{e}_3,\quad
		\operatorname{Re}\bigl(\langle v(\mathsf{e}_1), \nu \rangle\bigr)\neq 0,\quad
		\operatorname{Im}\bigl(\langle v(\mathsf{e}_1), \nu \rangle\bigr)\neq 0 \EWzJhvDisplayMathPeriod.
		\]
		Then $f_1,f_2$ as defined in the statement of the lemma above satisfy
		\begin{gather*}
			\textstyle f_1(x_1,x_2,x_3,x_4,0,\ldots,0)=8\operatorname{Re}\bigl(\bigl\langle v\bigl(\sum_{j=1}^4 x_j\mathsf{e}_j\bigr),\: \nu \bigr\rangle\bigr)\\
			\textstyle f_2(x_1,x_2,x_3,x_4,0,\ldots,0)=8\operatorname{Im}\bigl(\bigl\langle v\bigl(\sum_{j=1}^4 x_j\mathsf{e}_j\bigr),\: \nu \bigr\rangle\bigr) \EWzJhvDisplayMathPeriod.
		\end{gather*}
		The condition $\operatorname{Re}\bigl(\langle v(\mathsf{e}_1), \nu \rangle\bigr)\neq 0$ guarantees the coefficient of term~$x_1^d$ in $f_1$ is nonzero. Similarly, the coefficient of term~$x_1^d$ in $f_2$ is also nonzero. As a consequence, $f_1$ and $f_2$ can indeed be regarded as univariate polynomials of degree~$d$ with indeterminate~$x_1$ and coefficients in $\mathbb{R}[x_2,\ldots,x_{4N}]$.
		
		Suppose the resultant $\operatorname{res}_{x_1}^d(f_1,f_2) = 0 \in \mathbb{R}[x_2,\ldots,x_{4N}]$. Denote by $f\rvert_{\scriptscriptstyle 1234}$ the polynomial $f(x_1,x_2,x_3,x_4,0,\ldots,0)$ in $x_1,x_2,x_3,x_4$ for a polynomial $f(x_1,\ldots,x_{4N})$. Then $\operatorname{res}_{x_1}^d(f_1\rvert_{\scriptscriptstyle 1234},f_2\rvert_{\scriptscriptstyle 1234}) = \operatorname{res}_{x_1}^d(f_1,f_2)\bigr|_{\scriptscriptstyle 234} = 0 \in \mathbb{R}[x_2,x_3,x_4]$, so $f_1\rvert_{\scriptscriptstyle 1234}$ and $f_2\rvert_{\scriptscriptstyle 1234}$ have a nonconstant common divisor $\gcd(f_1\rvert_{\scriptscriptstyle 1234},f_2\rvert_{\scriptscriptstyle 1234})$ in $\mathbb{R}[x_1,x_2,x_3,x_4]$. Note
		\[(f_1+if_2)\rvert_{\scriptscriptstyle 1234}=8\,\bigl\langle v\bigl((x_1+ix_2)\mathsf{e}_1+(x_3+ix_4)\mathsf{e}_3\bigr),\: \nu \bigr\rangle \in \mathbb{C}[x_1+ix_2,\: x_3+ix_4] \EWzJhvDisplayMathPeriod.\]
		As $\gcd(f_1\rvert_{\scriptscriptstyle 1234},f_2\rvert_{\scriptscriptstyle 1234})\bigm|(f_1+if_2)\rvert_{\scriptscriptstyle 1234}$, also $\gcd(f_1\rvert_{\scriptscriptstyle 1234},f_2\rvert_{\scriptscriptstyle 1234}) \in \mathbb{C}[x_1+ix_2,\: x_3+ix_4]$, but then $\gcd(f_1\rvert_{\scriptscriptstyle 1234},f_2\rvert_{\scriptscriptstyle 1234}) = \overline{\gcd(f_1\rvert_{\scriptscriptstyle 1234},f_2\rvert_{\scriptscriptstyle 1234})} \in \mathbb{C}[x_1-ix_2,\: x_3-ix_4]$. As a consequence, $\gcd(f_1\rvert_{\scriptscriptstyle 1234},f_2\rvert_{\scriptscriptstyle 1234})$ must be a constant, which is a contradiction.
	\end{proof}
	
	We use the definition that for integer~$n \geqslant 0$, a subset of a smooth manifold (of dimension~$\geqslant n$) is said to be \emph{countably\/ $n$-$C^\infty$-rectifiable}\EWzJhvInfixFootnote{Note this definition does not allow an extra part with zero $n$-dimensional Hausdorff measure, whereas those definitions used by geometric measure theory usually allow so. Without this flexibility, it seems the \ifEWzJhvPDF\linebreak\fi image of a countably $n$-$C^\infty$-rectifiable subset under a smooth map (into a smooth manifold of dimension~$\geqslant n$) is not necessarily countably $n$-$C^\infty$-rectifiable. For example, \ifEWzJhvPDF\vspace{0.4pt}\fi starting with some $f_1 \in C^\infty(\mathbb{R}, \mathbb{R}^2)$ satisfying \ifEWzJhvPDF\vspace{0.8pt}\fi $f_1([0,\frac{1}{8}])=\{(0,0)\}$, $f_1([\frac{1}{4},\frac{3}{8}])=\{(0,1)\}$, $f_1([\frac{1}{2},\frac{5}{8}])=\{(1,1)\}$, and $f_1([\frac{3}{4},\frac{7}{8}])=\{(1,0)\}$, one may try to construct \ifEWzJhvPDF\vspace{0.4pt}\fi $f \in C^\infty([0,1], \mathbb{R}^2)$ satisfying $f([0,1]) \supset \mathcal{C}^2$ where $\mathcal{C} \subset \mathbb{R}$ is some Cantor-like set. For any $C^1$ embedding $\gamma \EWzJhvMathSepColon [0,1] \hookrightarrow \mathbb{R}^2$, $\mathcal{C}^2 \cap \gamma([0,1])$ is a closed nowhere dense subset of\/ $\mathcal{C}^2$, so as a consequence of the Baire category theorem, $\mathcal{C}^2$ and thus $f([0,1])$ cannot be countably $1$-$C^\infty$-rectifiable.}, if it is contained in a countable union of $n$-dimensional smooth submanifolds \cite{ZeroSetsSolutionsElliptic}. In addition, we denote the zero locus of $a$ by $\mathcal{N}(a)$.
	
	\begin{EWzJhvLemma}\label{lmaZeroLocusDim}
		$\mathcal{N}(a)$ is countably $(4N-2)$-$C^\infty$-rectifiable.
	\end{EWzJhvLemma}
	
	We shall only present a proof sketch of this lemma, as our proof is essentially the ``Proof of Main Theorem'' of \cite{ZeroSetsSolutionsElliptic}, with the only modification being that, to avoid the subtle problem of attempting to draw an equivalence between the existence of a common \emph{real} root of two polynomials and the vanishing of their resultant (seen on \cite[page~197]{ZeroSetsSolutionsElliptic}), we use \autoref{lmaResultant} instead to show that the resultant cannot be identically zero. For more details of the proof, please refer to \cite{ZeroSetsSolutionsElliptic}.
	
	\begin{proof}[Proof sketch]
		\autoref{lmaResultant} e\EWzJhvText{nab}les us to reuse the argument of the ``Proof of Main Theorem'' of \cite{ZeroSetsSolutionsElliptic} in our setting. Basically, one applies the Malgrange preparation theorem to $\langle a, \nu_1 \rangle, \langle a, \nu_2 \rangle$ at some zero point of $a$, where $\nu_1,\nu_2$ are from \autoref{lmaResultant}, smoothly extended to a neighbourhood of the zero point, and the local coordinates are chosen in accordance with $(\mathsf{e}_j)_{j=1}^{4N}$ from \autoref{lmaResultant}. Modulo a factor of a smooth nowhere \ifEWzJhvPDF\linebreak\fi vanishing function, $\langle a, \nu_1 \rangle, \langle a, \nu_2 \rangle$ are locally represented as univariate polynomials $a_1,a_2$ with indeterminate~$x_1$ whose coefficients are smooth functions in $x_2,\ldots,x_{4N}$, approximated by $f_1,f_2$ from \autoref{lmaResultant}. The resultant $\operatorname{res}_{x_1}^d(a_1,a_2)$ is also approximated by $\operatorname{res}_{x_1}^d(f_1,f_2)$. Using the fact that, to put it simply, the roots of a real polynomial form a subset of a certain countable set of smooth functions in the coefficients of the polynomial (\cite[Lemma~6]{SetOfZerosAlmostPolynomial}, used by \cite{ZeroSetsSolutionsElliptic} on page~194), we now have locally
		\begin{equation}\label{eqCodim2Resultant}
			\ifEWzJhvPDF
			\newcommand*{\EWzJhvEqCodimTwoResultantContent}[1]{\raisebox{0pt}[\height][1.5ex]{$\displaystyle#1$}}
			\else
			\newcommand*{\EWzJhvEqCodimTwoResultantContent}[1]{#1}
			\fi
			\EWzJhvEqCodimTwoResultantContent{\underbrace{\mathcal{N}\bigl(\operatorname{res}_{x_1}^d(a_1,a_2)\bigr)}_{\text{countably }(4N-1)\text{-}C^\infty\text{-rectifiable}}\supset
			\underbrace{\mathcal{N}\bigl(\operatorname{res}_{x_1}^d(a_1,a_2)\bigr)\cap\mathcal{N}(a_2)}_{\text{countably }(4N-2)\text{-}C^\infty\text{-rectifiable}}\supset\:
			\mathcal{N}(a_1)\cap\mathcal{N}(a_2)\:\supset\:
			\mathcal{N}(a) \EWzJhvDisplayMathPeriod.}\qedhere
		\end{equation}
	\end{proof}
	
	For our purposes, we need to prove slightly more about the zero locus of $a$.
	\begin{EWzJhvLemma}\label{lmaTZeroLocusDim}
		Let ``the tangent cone of\/\hspace{-0.08em} $\mathcal{N}(a)$''\EWzJhvFootnote{Presu\EWzJhvText{mab}ly it is related to some version of tangent space in the sense of geometric measure theory.} be
		\[T\mathcal{N}(a) := \Bigl\{(p,\xi) \,\Big|\: p \in M,\ \xi \in T_p M,\ a\rvert_p=a_p(\xi)=0 \in \Omega^{2,+}(V)\bigr|_p\Bigr\} \subset TM \EWzJhvDisplayMathPeriod.\]
		Then\/ $T\mathcal{N}(a)$ is countably $(8N-4)$-$C^\infty$-rectifiable.\ifEWzJhvPageBreak\pagebreak\fi
	\end{EWzJhvLemma}
	\begin{proof}[Proof sketch]
		We make the following decomposition:
		\[T\mathcal{N}(a) = \bigcup_{d=1}^\infty T\mathcal{N}^d(a),\quad T\mathcal{N}^d(a) := \bigl\{(p,\xi) \in T\mathcal{N}(a) \,\big|\: a\text{ vanishes at }p\text{ of order }d\,\bigr\}\]
		We only need to prove that each $T\mathcal{N}^d(a)$ is countably $(8N-4)$-$C^\infty$-rectifiable. Due to \autoref{lmaZeroLocusDim}, we only need to prove that for any $(4N-2)$-dimensional smooth submanifold $\varSigma \subset M$, $T\mathcal{N}^d(a) \cap TM\rvert_\varSigma$ is countably $(8N-4)$-$C^\infty$-rectifiable. The map $p \mapsto a_p$ from the order-$d$ zero locus of $a$ to a certain space of bundle-valued homogeneous polynomials on tangent spaces can be smoothly extended to a neighbourhood of the order-$d$ zero locus. Over a possibly smaller neighbourhood, the resultant from \autoref{lmaResultant} is everywhere a nonzero polynomial. In the manner of \EWzJhvEqRef{eqCodim2Resultant}, we get $T\mathcal{N}^d(a) \cap TM\rvert_\varSigma$ is a countably $(8N-4)$-$C^\infty$-rectifiable subset of the $(8N-2)$-dimensional smooth manifold $TM\rvert_\varSigma$.
	\end{proof}
	
	We can now make a choice of complex structure $J_a$ locally in accordance with the $\mathbb{Z}_2$-complex structure $\tilde{J}_a$, even near the zero locus of $a$.
	
	\begin{EWzJhvLemma}\label{lmaJaZ2Choice}
		For any open subset\/ $U \subset M$ with\/ $U$ diffeomorphic to\/ $\mathbb{R}^{4N}$, the\/ $\mathbb{Z}_2$-complex structure\/ $\tilde{J}_a$ can be lifted to a complex structure on\/ $U \setminus \mathcal{N}(a)$.
	\end{EWzJhvLemma}
	\begin{proof}
		For the convenience of notation, we simply assume $U=M=\mathbb{R}^{4N}$ endowed with metric $g$ that is possibly different from the standard Euclidean metric. Let $S^1 \subset \mathbb{R}^2$ be the unit circle. We only need to prove that for any smooth loop $\mskip1mu\relax \gamma \mskip1mu\relax\EWzJhvMathSepColon\mskip1mu\relax S^1 \mskip1mu\relax\rightarrow\mskip1mu\relax U \setminus \mathcal{N}(a) \mskip1mu\relax$, the $\mathbb{Z}_2$-complex structure $\tilde{J}_a$ can be lifted to a smooth section of linear complex structures along $\gamma$. Let $\gamma$ be such a loop. We may construct a smooth map \hspace{0.05em}$\mathtt{F} \EWzJhvMathSepColon \mathbb{R}^2 \rightarrow U$ satisfying $\mathtt{F}\rvert_{S^1}=\gamma$. We now make a transversality argument regarding $\mathtt{F}$ and $\mathcal{N}(a)$. Define
		\begin{align*}
			\mathtt{F}_{\xi,\xi_x,\xi_y}(x,y) :={}& \mathtt{F}(x,y) + \xi + x\hspace{0.0625em}\xi_x + y\hspace{0.0625em}\xi_y \quad \text{for \hspace{0.05em}}\xi,\xi_x,\xi_y \in \mathbb{R}^{4N}\ \text{\hspace{0.01em}and \hspace{0.01em}}(x,y) \in \mathbb{R}^2\\[1.5ex]
			\mathtt{TF}_{\xi,\xi_x,\xi_y}(x,y;\theta) :={}&\bigl(\mathtt{F}_{\xi,\xi_x,\xi_y}(x,y),\: (\cos\theta)\hspace{0.0625em}\partial_x\mathtt{F}_{\xi,\xi_x,\xi_y}(x,y) + (\sin\theta)\hspace{0.0625em}\partial_y\mathtt{F}_{\xi,\xi_x,\xi_y}(x,y)\bigr)\\ ={}&\bigl(\mathtt{F}_{\xi,\xi_x,\xi_y}(x,y),\: (\cos\theta)\bigl(\partial_x\mathtt{F}(x,y) + \xi_x\bigr) + (\sin\theta)\bigl(\partial_y\mathtt{F}(x,y) + \xi_y\bigr)\bigr)\\ \in{}& (\mathbb{R}^{4N})^2 \quad \text{for \hspace{0.05em}}\xi,\xi_x,\xi_y \in \mathbb{R}^{4N},\;(x,y) \in \mathbb{R}^2,\;\text{and \hspace{0.05em}}\theta \in \mathbb{R} \EWzJhvDisplayMathPeriod.
		\end{align*}
		We easily check that $\mskip2mu\relax \mathtt{TF}_{\EWzJhvCDotInd,\EWzJhvCDotInd,\EWzJhvCDotInd}(\EWzJhvCDotInd,\EWzJhvCDotInd;\EWzJhvCDotInd) \mskip2mu\relax\EWzJhvMathSepColon\mskip2mu\relax (\mathbb{R}^{4N})^3 \times \mathbb{R}^2 \times \mathbb{R} \mskip2mu\relax\rightarrow\mskip2mu\relax (\mathbb{R}^{4N})^2 \mskip2mu\relax$ is a submersion by differentiating in $\xi,\xi_x,\xi_y$. For the convenience of notation, we identify $T\mathbb{R}^{4N}$ with $(\mathbb{R}^{4N})^2$ where the second component is the tangent vector. $T\mathcal{N}(a)$ from \autoref{lmaTZeroLocusDim} is a subset of\/ $T\mathbb{R}^{4N}$, that is, of $(\mathbb{R}^{4N})^2$. \autoref{lmaTZeroLocusDim} says $T\mathcal{N}(a)$ is countably $(8N-4)$-$C^\infty$-rectifiable. Because $\mathtt{TF}_{\EWzJhvCDotInd,\EWzJhvCDotInd,\EWzJhvCDotInd}(\EWzJhvCDotInd,\EWzJhvCDotInd;\EWzJhvCDotInd)$ is a submersion, $\bigl(\mathtt{TF}_{\EWzJhvCDotInd,\EWzJhvCDotInd,\EWzJhvCDotInd}(\EWzJhvCDotInd,\EWzJhvCDotInd;\EWzJhvCDotInd)\bigr)\ifEWzJhvPDF{\rule{0pt}{1.88ex}}\fi^{\mskip-2mu\relax -1}\bigl(T\mathcal{N}(a)\bigr)$ must be countably $(12N-1)$-$C^\infty$-rectifiable, so its projection onto the $(\mathbb{R}^{4N})^3$ factor of $(\mathbb{R}^{4N})^3 \times \mathbb{R}^2 \times \mathbb{R}$ must have zero Lebesgue measure. Therefore, it is possible to choose small $\xi,\xi_x,\xi_y$ such that for any~$(x,y;\theta) \in \mathbb{R}^2 \times \mathbb{R}$, $\mathtt{TF}_{\xi,\xi_x,\xi_y}(x,y;\theta) \notin T\mathcal{N}(a)$. Because $\xi,\xi_x,\xi_y$ can be chosen to be arbitrarily small, we may assume the loop $\mathtt{F}_{\xi,\xi_x,\xi_y}\rvert_{S^1}$ is homotopic to $\gamma$ in $U \setminus \mathcal{N}(a)$. As it is equivalent to lift along homotopic loops, we only need to prove that the $\mathbb{Z}_2$-complex structure $\tilde{J}_a$ can be lifted to a smooth section of linear complex structures along $\mathtt{F}_{\xi,\xi_x,\xi_y}\rvert_{S^1}$.
		
		Suppose $(x,y) \in \mathbb{R}^2$ satisfies $\mathtt{F}_{\xi,\xi_x,\xi_y}(x,y) \in \mathcal{N}(a)$. Write $p:=\mathtt{F}_{\xi,\xi_x,\xi_y}(x,y)$. Let $a$ vanish at $p$ of order~$d$. The condition $\mathtt{TF}_{\xi,\xi_x,\xi_y}(x,y;\theta) \notin T\mathcal{N}(a)$ for any~$\theta \in \mathbb{R}$ guarantees $a \circ \mathtt{F}_{\xi,\xi_x,\xi_y}$ also vanishes at $(x,y)$ of order~$d$, and its $d$-th order Taylor \ifEWzJhvPDF\linebreak\fi expansion at $(x,y)$, denoted by $(a \circ \mathtt{F}_{\xi,\xi_x,\xi_y})_{x,y}$, satisfies \ifEWzJhvPDF\vspace{0.5pt}\fi $(a \circ \mathtt{F}_{\xi,\xi_x,\xi_y})_{x,y}(x_1,y_1) \neq 0$ for any~$(x_1,y_1)\in\mathbb{R}^2\setminus\{(0,0)\}$. This in turn guarantees \ifEWzJhvPDF\vspace{0.5pt}\fi $(\mathtt{F}_{\xi,\xi_x,\xi_y})^{-1}\bigl(\mathcal{N}(a)\bigr)$ is discrete. \ifEWzJhvPDF\linebreak\fi Let $D\subset\mathbb{R}^2$ be the closed unit disc with boundary~$S^1$. Choose a small circle centred at each point of $D \cap (\mathtt{F}_{\xi,\xi_x,\xi_y})^{-1}\bigl(\mathcal{N}(a)\bigr)$. The total of these small circles can be joined together to be homotopic to $S^1$ in $D \setminus (\mathtt{F}_{\xi,\xi_x,\xi_y})^{-1}\bigl(\mathcal{N}(a)\bigr)$, \ifEWzJhvPageBreak\else\ifEWzJhvPDF\vspace{-1pt}\fi\fi so we only need to prove \ifEWzJhvPageBreak\pagebreak\fi the $\mathbb{Z}_2$-complex structure $\tilde{J}_a$ can be lifted to a smooth section of linear complex structures along every such small circle. Due to the Taylor expansion, taking the limit that the radius of the small circle tends to zero, we only need to prove the ``$\mathbb{Z}_2$-linear\EWzJhvFootnote{not ``linear with respect to the finite field $\mathbb{Z}_2$'', just linear complex structures modulo $\{\pm 1\}$} complex structures'' $\tilde{J}(x_1,y_1) \in \widetilde{\mathcal{J}}(M)\rvert_p$ parametrised by $(x_1,y_1) \in S^1$ given by the ``parascalar \ifEWzJhvPDF\linebreak\fi $(2,0)$-form at a point\EWzJhvFootnote{that is, $(a \circ \mathtt{F}_{\xi,\xi_x,\xi_y})_{x,y}(x_1,y_1) \in \Omega^{2,+}(V)\bigr|_p$ in the form of $\lambda\hspace{0.0625em}\bigl(v_1 \otimes \omega^1 + v_2 \otimes \omega^2\bigr)$ as in \autoref{defn2plusParascalar}\ifEWzJhvPDF\vspace{-1pt}\fi}'' \ifEWzJhvPDF\vspace{-3pt}\fi $(a \circ \mathtt{F}_{\xi,\xi_x,\xi_y})_{x,y}(x_1,y_1) \in \Omega^{2,+}(V)\bigr|_p$ can be lifted to linear \ifEWzJhvPDF\linebreak\fi complex structures parametrised by $S^1$. Let $J(\theta) \in \mathcal{J}(M)\rvert_p$ be a lift of $\tilde{J}(\cos\theta,\sin\theta)$ parametrised by~$\theta\in\mathbb{R}$. Because \ifEWzJhvPDF\vspace{-1pt}\fi $(a \circ \mathtt{F}_{\xi,\xi_x,\xi_y})_{x,y}(x_1,y_1) = (-1)^d(a \circ \mathtt{F}_{\xi,\xi_x,\xi_y})_{x,y}(-x_1,-y_1)$, $\tilde{J}(x_1,y_1)=\tilde{J}(-x_1,-y_1)$, so $J(\theta+\pi) = \pm J(\theta)$ for~$\theta\in\mathbb{R}$.\EWzJhvFootnote{Here, the reliance on the Taylor expansion and thus smoothness (up to the order of zero) feels quite essential. We do not know whether this lemma would still hold if things were only assumed to be, say, $C^{100}$, for $a$ that possibly has zeros of order higher than, say, $1000$.} In both cases of the sign, we must have $J(\theta+2\pi) = J(\theta)$, so $J([\theta])$ is well-defined for~$[\theta]\in\mathbb{R}/2\pi\mathbb{Z}$, which means $\tilde{J}(x_1,y_1)$ can be lifted to linear complex structures parametrised by $(x_1,y_1) \in S^1$.
	\end{proof}
	
	\subsection{Bounding \texorpdfstring{$\bigl\|\lvert f\rvert^{-\frac{1}{n}}\bigr\|_{L^1}$}{‖|𝑓|⁻¹∕\unichar{"202F}ⁿ‖\_𝐿¹}}\label{subsecFMinus1n}
	
	In this subsection, we shall present \autoref{lmaFMinus1n}, which will be used to estimate the term $\bigl|a^{2,0}\bigr|^{-\frac{1}{n}}$ in the context of \autoref{corHigherDer}. Although this is a subsection of ``Case of $\mathrm{Hol}(g) \subset \mathrm{Sp}(N)\hspace{0.0625em}\mathrm{Sp}(1)$'', here, we do not really need anything ``$\mathrm{Hol}(g) \subset \mathrm{Sp}(N)\hspace{0.0625em}\mathrm{Sp}(1)$''. The results in this subsection are in fact rather general properties of (smooth) functions.
	
	We first prove some preparatory lemmas. For this purpose, we need the knowledge of the Chebyshev polynomials $T_n(x)\ (n=0,1,2,\ldots)$ given by $T_n(\cos\theta) = \cos n\theta$. From the recurrence relation $T_{n+2}(x) = 2x\,T_{n+1}(x) - T_n(x)$, we see that for~$n \geqslant 1$, $T_n(x)$ is a polynomial with leading term~$2^{n-1}x^n$. Note $T_n(x) \in [-1,1]$ for~$x \in [-1,1]$. We shall see that, among those functions $f$ satisfying\rule[1.2ex]{0.1em}{0.1em} \hspace{0.15em}$\inf{}\bigl|f^{(n)}\bigr| \geqslant 2^{n-1}n!$, $T_n$ maximises the Lebesgue measure $\EWzJhvMathOP\lambda\bigl(f^{-1}([-1,1])\bigr)$. Here, our main tool is the Lagrange interpolation formula.
	
	\begin{EWzJhvLemma}\label{lmaChebyshevLagrange}
		For any positive integer\/~$n$, \EWzJhvFootnote{
			The value of each summand on the left-hand side can actually be calculated.
			
			In particular, letting $L_{n,j}(\theta) := \prod_{\substack{k \in \mathbb{Z},\\0 \leqslant k \leqslant n,\,\,k \neq j}} \bigl(\cos\theta - \cos\frac{k\pi}{n}\bigr)$, we may verify $L_{n,0}(\theta) = 2^{1-n} \sin n\theta \cot\frac{\theta}{2}$ and $2 (-1)^j L_{n,j}(\theta) = L_{n,0}\bigl(\theta - \frac{j\pi}{n}\bigr) + L_{n,0}\bigl(\theta + \frac{j\pi}{n}\bigr)$ for any~$j$ satisfying $0 \leqslant j \leqslant n$ by checking that both-hand sides are trigonometric polynomials of period~$2\pi$, of equal degree, with equal (multiple) roots as many as the degree permits, and with equal leading coefficients.
			
			Alternatively, one may realise $\prod_{j=0}^n \bigl(x-\cos\frac{j\pi}{n}\bigr)$ as some variant of the Chebyshev polynomial, and then calculate its derivative at $\cos\frac{j\pi}{n}$ (see e.g.\ \cite{RieszBernstein} or \cite{LagrangianInterpolation}).\ifEWzJhvPDF\vspace{-0.8pt}\fi
			
			As a result, the summand equals to $\frac{2^{n-2}}{n}$~if~$j \in \{0,n\}$, or $\frac{2^{n-1}}{n}$~if~$0<j<n$.}
		\[\sum_{j=0}^n \prod_{\substack{k \in \mathbb{Z},\\0 \leqslant k \leqslant n,\,\,k \neq j}} \biggl|\cos\frac{j\pi}{n} - \cos\frac{k\pi}{n}\biggr|^{-1} = 2^{n-1} \EWzJhvDisplayMathPeriod.\]
	\end{EWzJhvLemma}
	\begin{proof}
		For any integer~$j$, $T_n\bigl(\cos \frac{j\pi}{n}\bigr) = \cos j\pi = (-1)^j$. The Lagrange interpolation formula tells us
		\[T_n(x) = \sum_{j=0}^n{}(-1)^j \prod_{\substack{k \in \mathbb{Z},\\0 \leqslant k \leqslant n,\,\,k \neq j}} \frac{x - \cos\frac{k\pi}{n}}{\cos\frac{j\pi}{n} - \cos\frac{k\pi}{n}}\]
		whose leading coefficients are
		\[2^{n-1} = \sum_{j=0}^n{}(-1)^j \prod_{\substack{k \in \mathbb{Z},\\0 \leqslant k \leqslant n,\,\,k \neq j}} \biggl(\cos\frac{j\pi}{n} - \cos\frac{k\pi}{n}\biggr)^{\mskip-4mu\relax -1} \EWzJhvDisplayMathPeriod.\]
		The summands on the right-hand side are positive for all~$j$, so we get the conclusion.\ifEWzJhvPageBreak\pagebreak\fi
	\end{proof}
	
	\begin{EWzJhvLemma}\label{lmaFMinus1n1D}
		Let\/ $n$ be a positive integer. Let\/ $f\EWzJhvMathSepColon(a,b)\rightarrow\mathbb{R}$ be\/ $n$-times differentiable where\/ $-\infty \leqslant a < b \leqslant +\infty$. Denote the Lebesgue measure by\/ $\EWzJhvMathOP\lambda$. For\/~$\delta>0$, we have
		\[\EWzJhvMathOP\lambda\bigl(f^{-1}([-\delta,\delta])\bigr) \leqslant 4\!\left(\frac{n!\,\delta}{\:2\inf\limits_{x\in(a,b)}{}\bigl|f^{(n)}(x)\bigr|}\right)^{\mskip-7mu\relax \frac{1}{n}} \EWzJhvDisplayMathPeriod.\]
	\end{EWzJhvLemma}
	\begin{EWzJhvRemark}
		When $f = \delta\,T_n$, the equality holds.
	\end{EWzJhvRemark}
	\begin{proof}
		It is sufficient to prove $\EWzJhvMathOP\lambda(K) \leqslant 4\mskip2mu\relax({\cdots})^{\frac{1}{n}}$ for any compact~$K \subset f^{-1}([-\delta,\delta])$. Let $K$ be such a compact subset. We assume $\EWzJhvMathOP\lambda(K) > 0$. Note that $x \mapsto \EWzJhvMathOP\lambda\bigl(K \cap (-\infty, x]\bigr)$ is a nondecreasing continuous function over $\mathbb{R}$. Let
		\[x_0 := \inf K,\quad x_j := \inf{}\Biggl\{x \in \mathbb{R} \Biggm| \EWzJhvMathOP\lambda\bigl(K \cap (-\infty, x]\bigr) \geqslant \EWzJhvMathFracVCentered{1-\cos\frac{j\pi}{n}}{2}\EWzJhvMathOP\lambda(K)\Biggr\}\text{ for }j=1,2,\ldots,n \EWzJhvDisplayMathPeriodOverfull.\]
		Then $x_0<x_1<\cdots<x_n$, $x_j \in K$ for any~$j$, $\displaystyle \EWzJhvMathOP\lambda\bigl(K \cap (-\infty, x_j]\bigr) = \EWzJhvMathFracVCentered{1-\cos\frac{j\pi}{n}}{2}\EWzJhvMathOP\lambda(K)$ for any~$j$, and thus $\displaystyle \lvert x_j-x_k\rvert \geqslant \EWzJhvMathFracVCentered{\bigl|\cos\frac{j\pi}{n} - \cos\frac{k\pi}{n}\bigr|}{2}\EWzJhvMathOP\lambda(K)$ for any~$j,k$. Let
		\[g(x) := f(x) - \sum_{j=0}^n f(x_j) \prod_{k \neq j} \EWzJhvMathFracVCentered{x-x_k}{x_j-x_k \EWzJhvDisplayMathPeriod{\,.}}\]
		Then $g(x_j)=0$ for~$j=0,1,\ldots,n$. We use the mean value theorem to show that $g'$ has $n$ distinct zeros, and thus $g''$ has $n-1$ distinct zeros, \ldots, and thus $g^{(n)}$ has a zero. \ifEWzJhvPDF\linebreak\fi We then derive as follows:
		\begin{align*}
			0=\inf_{(a,b)}{}\bigl|g^{(n)}\bigr| &= \inf_{(a,b)}{}\Biggl|f^{(n)} - \sum_{j=0}^n \frac{n!\,f(x_j)}{\prod_{k \neq j} (x_j-x_k)}\Biggr|\\
			&\geqslant \inf_{(a,b)}{}\bigl|f^{(n)}\bigr|\, - \sum_{j=0}^n \frac{n!\,\bigl|f(x_j)\bigr|}{\prod_{k \neq j} \lvert x_j-x_k\rvert}\\
			&\geqslant \inf_{(a,b)}{}\bigl|f^{(n)}\bigr|\, - \sum_{j=0}^n \frac{n!\,\delta}{\displaystyle \,2{\vphantom{\bigr)}}^{-n}\,\bigl(\EWzJhvMathOP\lambda(K)\bigr)^{\mskip-2mu\relax n}\,\textstyle\prod_{k \neq j} \bigl|\cos\frac{j\pi}{n} - \cos\frac{k\pi}{n}\bigr|}\\
			&= \inf_{(a,b)}{}\bigl|f^{(n)}\bigr|\, - 2{\vphantom{\bigr)}}^{2n-1}\,n!\,\delta\,\bigl(\EWzJhvMathOP\lambda(K)\bigr)^{\mskip-2mu\relax -n} \qquad\text{(using \autoref{lmaChebyshevLagrange})}
		\end{align*}
		Rearrange the inequality to get the conclusion.
	\end{proof}
	
	Finally, we return to the world of smooth manifolds, where the bound is:
	
	\begin{EWzJhvLemma}\label{lmaFMinus1n}
		Let\/ $(M,g)$ be a smooth Riemannian manifold of dimension\/~$n$. Let\/ $m$ be a positive integer. Let\/ \ifEWzJhvPDF\vspace{0.5pt}\fi $f_0 \in C^m(M)$. Suppose\/ $f_0$ only vanishes at zeros of order\/~$\leqslant m$.\EWzJhvFootnote{i.e.\ $\sum_{j=0}^m{}\bigl|\nabla^j f_0\bigr| >0$} Let\/ $K \subset M$ be a compact subset. Then there exists\/ $\varepsilon>0$ satisfying for any\/~$f_1 \in C^m(M)$ with\/ $\sum_{j=0}^m{}\bigl\|\nabla^j (f_0-f_1)\bigr\|_{\text{$L^\infty(M)$}}\mskip-1mu\relax<\varepsilon$ and any\/~$\delta>0$, $\operatorname{Vol}_g\bigl(K \cap f_1^{-1}([-\delta,\delta])\bigr) \leqslant \varepsilon^{-1}\delta^{\frac{1}{m}}$. In particular, $\int_K{}\lvert f_1\rvert^{-\alpha}\,\mathrm{d}\mathrm{Vol}_g \leqslant C\bigl(\operatorname{Vol}_g(K),\varepsilon,\alpha m\bigr)$ for\/~$\alpha\in\bigl[0,\frac{1}{m}\bigr)$.
	\end{EWzJhvLemma}
	\begin{proof}
		Because $K$ is compact, we only need to prove that for any~$p \in K$, there exist an open neighbourhood $U \ni p$ and $\varepsilon>0$ satisfying for any~$f_1 \in C^m(M)$ with $\sum_{j=0}^m{}\bigl\|\nabla^j (f_0-f_1)\bigr\|_{\text{$L^\infty(M)$}}\mskip-1mu\relax<\varepsilon$ and any~$\delta>0$, $\operatorname{Vol}_g\bigl(U \cap f_1^{-1}([-\delta,\delta])\bigr) \leqslant \varepsilon^{-1}\delta^{\frac{1}{m}}$. WLOG, we simply assume that $(M,g)$ is the unit ball in the Euclidean space $(\mathbb{R}^n,g_{\mathbb{R}^{n\mathstrut}})$ and $p$ is the origin. Because $f_0$ only vanishes at zeros of order~$\leqslant m$, there exists an \ifEWzJhvPDF\linebreak\fi integer~$m_1$ satisfying $0 \leqslant m_1 \leqslant m$ and $\nabla^{m_1}f_0(0) \neq 0$. By applying an orthogonal transform to the unit ball, we may assume \ifEWzJhvPDF\vspace{1pt}\fi $\partial_{x_1}^{m_1}f_0(0) \neq 0$. Choose $\varepsilon>0$ satisfying for any~$f_1 \in C^m(M)$ with \ifEWzJhvPDF\vspace{-1pt}\fi $\sum_{j=0}^m{}\bigl\|\nabla^j (f_0-f_1)\bigr\|_{\text{$L^\infty(M)$}}\mskip-1mu\relax<\varepsilon$ and any~$x \in B^n_\varepsilon(0) \subset \mathbb{R}^n$,\EWzJhvFootnote{$B^n_\varepsilon(0)$ denotes the open ball of radius~$\varepsilon$ centred at the origin in $\mathbb{R}^n$.} $\bigl|\partial_{x_1}^{m_1}f_1(x)\bigr| \geqslant \frac{1}{2}\hspace{0.0625em}\bigl|\partial_{x_1}^{m_1}f_0(0)\bigr|$. For any~$\delta>0$, \ifEWzJhvPDF\vspace{1pt}\fi \autoref{lmaFMinus1n1D} bounds the $1$-dimensional Lebesgue measure \ifEWzJhvPDF\vspace{-1pt}\fi of the intersection of $B^n_\varepsilon(0) \cap f_1^{-1}([-\delta,\delta])$ with any line parallel to the $x_1$-axis. We then obtain \ifEWzJhvPDF\vspace{1pt}\fi $\operatorname{Vol}\bigl(B^n_\varepsilon(0) \cap f_1^{-1}([-\delta,\delta])\bigr) \leqslant C\bigl(m_1,\bigl|\partial_{x_1}^{m_1}f_0(0)\bigr|,n,\varepsilon\bigr)\hspace{0.0625em}\delta^{\frac{1}{m_1}}$.\EWzJhvFootnote{In case $m_1=0$, we have $\operatorname{Vol}\bigl(B^n_\varepsilon(0) \cap f_1^{-1}([-\delta,\delta])\bigr) =0$ for~$\delta < \frac{1}{2}\hspace{0.0625em}\lvert f_0(0)\rvert$ instead.} As we also have \ifEWzJhvPDF\vspace{-1pt}\fi $\operatorname{Vol}\bigl(B^n_\varepsilon(0) \cap f_1^{-1}([-\delta,\delta])\bigr) \leqslant \operatorname{Vol}\bigl(B^n_\varepsilon(0)\bigr) = C(n,\varepsilon)$, we can therefore conclude $\operatorname{Vol}\bigl(B^n_\varepsilon(0) \cap f_1^{-1}([-\delta,\delta])\bigr) \leqslant C\bigl(m,\bigl|\partial_{x_1}^{m_1}f_0(0)\bigr|,n,\varepsilon\bigr)\hspace{0.0625em}\delta^{\frac{1}{m}}$.\ifEWzJhvPDF\vspace{1pt}\fi
		
		To prove $\int_K{}\lvert f_1\rvert^{-\alpha}\,\mathrm{d}\mathrm{Vol}_g \leqslant C\bigl(\operatorname{Vol}_g(K),\varepsilon,\alpha m\bigr)$ for~$\alpha\in\bigl(0,\frac{1}{m}\bigr)$, just note
		\begin{align*}
			\int_K{}\lvert f_1\rvert^{-\alpha}\,\mathrm{d}\mathrm{Vol}_g &= \int_0^\infty \operatorname{Vol}_g\bigl(K \cap f_1^{-1}([-t^{-\frac{1}{\alpha}},t^{-\frac{1}{\alpha}}])\bigr)\,\mathrm{d}t\\ &\leqslant \int_0^\infty \min\bigl\{\operatorname{Vol}_g(K),\:\varepsilon^{-1}t^{-\frac{1}{\alpha m}}\bigr\}\,\mathrm{d}t \EWzJhvDisplayMathPeriod.\qedhere
		\end{align*}
	\end{proof}
	
	\begin{EWzJhvRemark}
		Later in \autoref{subsecFinishUp}, we shall apply \autoref{lmaFMinus1n} to $\lvert a\rvert^2$ in the context of \autoref{thm2} and \autoref{corHigherDer}. Note that \autoref{lmaFMinus1n} works for rather general smooth functions, but $a$ satisfies a certain elliptic PDE, and its zero locus is countably $(4N-2)$-$C^\infty$-rectifiable by \autoref{lmaZeroLocusDim}. Suppose $a$ only vanishes at zeros of order~$\leqslant m$. \autoref{lmaFMinus1n} can only give us a bound on $\int_K{}\lvert a\rvert^{-\alpha}\,\mathrm{d}\mathrm{Vol}_g$ for~$\alpha\in\bigl[0,\frac{1}{m}\bigr)$, but as $a$ has these rather special properties, presu\EWzJhvText{mab}ly we should be able to derive a bound on $\int_K{}\lvert a\rvert^{-\alpha}\,\mathrm{d}\mathrm{Vol}_g$ for~$\alpha$ in the wider range $\bigl[0,\frac{2}{m}\bigr)$ with some more specialised method. As an illustration, $\lvert z\rvert^{-\alpha}$ is integ\EWzJhvText{rab}le over the unit disc $D \subset \mathbb{C}$ for~$\alpha \in [0,2)$.
	\end{EWzJhvRemark}
	
	\subsection{Finish up}\label{subsecFinishUp}
	
	In this subsection, we use the theory of\/ Sobolev spaces and elliptic PDEs to complete the proof of \autoref{thm2} (and of \autoref{thm2Continuity}). Notably, we take advantage of an elliptic PDE satisfied by $J_a$ {\small (\autoref{lmaPdeJa})}.
	
	\begin{EWzJhvLemma}\label{lmaW1pSingularCodim1Plus}
		Let\/ $M$ be a smooth manifold of dimension\/~$n$ (\/$n \geqslant 2$). Let\/ $V \rightarrow M$ be a smooth vector bundle. Let\/ $K \subset M$ be a closed subset with zero\/ $(n-1)$-dimensional Hausdorff measure. Let\/ $p\in[1,\infty]$. Let\/ $f$ be a\/ $W^{1,p}_{\mathrm{loc}}$ section of\/ $V$ over\/ $M \setminus K$. If the zero extensions of\/\hspace{-0.08em} $f$ and\/ $\nabla f$ (given some connection on\/ $V \rightarrow M$) from\/ $M \setminus K$ to\/ $M$ are actually\/ $L^p_{\mathrm{loc}}$ over\/ $M$, then\/ $f$ is\/ $W^{1,p}_{\mathrm{loc}}$ over\/ $M$.\vspace{-1pt}
	\end{EWzJhvLemma}
	\begin{proof}
		\linespread{1.1}\selectfont
		
		WLOG, we simply assume $M$ is $\mathbb{R}^n$. We also identify $V \rightarrow M$ with some trivial vector bundle. To prove that $f$ is $W^{1,p}_{\mathrm{loc}}$ over $M$, we only need to verify for any smooth compactly supported function $\varphi \in C_{\mathrm{c}}^\infty(\mathbb{R}^n)$ and any integer~$j$ satisfying $1 \leqslant j \leqslant n$, $\int_{\mathbb{R}^n \setminus K} f\,\partial_{x_j}\varphi\,\mathrm{d}x=-\int_{\mathbb{R}^n \setminus K} \varphi\,\partial_{x_j}f\,\mathrm{d}x$. WLOG, we assume $j=1$. Let $\pi \EWzJhvMathSepColon \mathbb{R}^n \rightarrow \mathbb{R}^{n-1}$ be the projection onto the $x_2x_3 \cdots x_n$-hyperplane. Let $\{\psi_k\}_{k=1,2,\ldots} \subset C_{\mathrm{c}}^\infty(\mathbb{R}^{n-1})$ be a sequence of functions satisfying $0 \leqslant \psi_k \leqslant 1$, $\overline{\operatorname{supp}(\psi_k)} \cap \pi\bigl(\overline{\operatorname{supp}(\varphi)} \cap K\bigr) = \varnothing$, and $\psi_k \rightarrow 1$ pointwise over $\mathbb{R}^{n-1} \setminus \pi\bigl(\overline{\operatorname{supp}(\varphi)} \cap K\bigr)$ as $k\rightarrow\infty$. Using $(\psi_k\circ\pi)\,\varphi$ as a test function, we get $\int_{\mathbb{R}^n \setminus K}{}(\psi_k\circ\pi)\hspace{0.0625em}\bigl(f\,\partial_{x_1}\varphi+\varphi\,\partial_{x_1}f\bigr)\,\mathrm{d}x=0$. Lebesgue's dominated convergence theorem implies $\int_{\mathbb{R}^n \setminus K}{}\bigl(\mathbf{1}_{\mathbb{R}^{n-1} \setminus \pi(\overline{\operatorname{supp}(\varphi)} \cap K)}\circ\pi\bigr)\hspace{0.0625em}\bigl(f\,\partial_{x_1}\varphi+\varphi\,\partial_{x_1}f\bigr)\,\mathrm{d}x=0$ by taking $k\rightarrow\infty$. Since $\pi\bigl(\overline{\operatorname{supp}(\varphi)} \cap K\bigr)$ has zero $(n-1)$-dimensional Lebesgue measure, we arrive at $\int_{\mathbb{R}^n \setminus K} f\,\partial_{x_1}\varphi\,\mathrm{d}x=-\int_{\mathbb{R}^n \setminus K} \varphi\,\partial_{x_1}f\,\mathrm{d}x$.\qedhere\par
	\end{proof}
	
	To prove \autoref{thm2}, we work over any open $U \subset M$ satisfying that $\overline{U}$ is compact and $U$ is diffeomorphic to $\mathbb{R}^{4N}$. Due to \autoref{lmaJaZ2Choice}, we can choose a complex structure $J_a$ from the $\mathbb{Z}_2$-complex structure $\tilde{J}_a$ over $U \setminus \mathcal{N}(a)$. We can then define $a^{2,0}$ given by \EWzJhvEqRef{eqA20} over $U \setminus \mathcal{N}(a)$.
	
	\begin{EWzJhvLemma}\label{lmaJaW1p}
		$J_a$ as a smooth section of\/ $TM \otimes_{\mathbb{R}} \EWzJhvTpStar T{\hphantom{p}} M$ over\/ $U \setminus \mathcal{N}(a)$ extends uniquely to a\/ $W^{1,p}_{\mathrm{loc}}$ section over\/ $U$ for any\/~$p\in[1,\infty)$.
	\end{EWzJhvLemma}
	\begin{proof}
		Due to $\lvert a\rvert=\sqrt{2}\mskip2mu\relax\bigl|a^{2,0}\bigr|$, $\bigl|\nabla^A a\bigr| = \sqrt{2}\mskip2mu\relax\bigl|\nabla^A a^{2,0}\bigr|$ (\autoref{corNablaAA20}), \autoref{lmaZeroLocusDim} and \autoref{lmaW1pSingularCodim1Plus}, $a^{2,0}$ is $W^{1,\infty}_{\mathrm{loc}}$ over $U$. Furthermore, $a^{2,0}$ satisfies ${\displaystyle \mathrm{d}_A^*}a^{2,0}=0$ (\autoref{lmaA20eq}) weakly over $U$. The Weitzenböck formula {\small (\autoref{propWeitzen2plus})} and the regularity theory of weak solutions to second-order linear elliptic PDEs now imply $a^{2,0}$ is smooth over $U$.
		
		As $\overline{U}$ is assumed to be compact, there exists some positive integer~$m$ satisfying $a\rvert_U$ only vanishes at zeros of order~$\leqslant m$ {\small \cite{UniqueContinuationGeometry}}. \autoref{lmaFMinus1n} now applies to $\lvert a\rvert^2$ over $U$ and implies $\lvert a\rvert^{-\alpha} \in L^1_{\mathrm{loc}}(U)$ for~$\alpha>0$ small. Note $\lvert a\rvert=\sqrt{2}\mskip2mu\relax\bigl|a^{2,0}\bigr|$. Then \autoref{corHigherDer} with $n$ large\EWzJhvFootnote{Curiously, depending on how high $N$ and the orders of zeros of $a$ are (we only need $J_a \in W^{1,p}_{\mathrm{loc}}$ for a certain value of $p$ that only depends on $N$ to proceed further), the existence of derivatives of accordingly arbitrarily high order is needed here, due to the nature of \autoref{corHigherDer}. It would be interesting to know whether the results would still hold if things were only assumed to be, say, $C^{100}$, for $a$ that possibly has zeros of order higher than, say, $1000$.} implies that the zero extension of\/ $\nabla J_a$ from $U \setminus \mathcal{N}(a)$ to $U$ is $L^p_{\mathrm{loc}}$ \ifEWzJhvPDF\vspace{-1pt}\fi over $U$ for any~$p\in[1,\infty)$. Finally, \autoref{lmaW1pSingularCodim1Plus} implies $J_a$ is $W^{1,p}_{\mathrm{loc}}$ over $U$ for any~$p\in[1,\infty)$.
	\end{proof}
	
	Recall the notation $R$ and $\EWzJhvCDotInd\circledast\EWzJhvCDotInd$ from the description before \autoref{propWeitzen2plus}.
	
	\begin{EWzJhvLemma}[A PDE satisfied by $J_a$]\label{lmaPdeJa}\hspace*{-0.5em}\EWzJhvFootnote{There are \hyperref[subsubsecLmaPdeJaNotes]{additional notes} to this lemma at the end of this subsection.}
		Let\/ $\omega^{J_a} := g(J_a\EWzJhvCDotInd,\EWzJhvCDotInd)$ be regarded as a\/ $W^{1,p}_{\mathrm{loc}}$ section of\/\hspace{0em plus -0.08em} $\Omega^{2,+}(U)$, where\/ $p\in[1,\infty)$. If\/\hspace{0em plus -0.08em} $b \in \Omega^{2,+}(U)$ has compact support in\/ $U$, then
		\[\int_U{}\bigl\langle \nabla{\textstyle \omega^{J_a}}, \nabla b \bigr\rangle\,\mathrm{d}\mathrm{Vol}_g = \int_U{}\biggl(\EWzJhvMathFracVCentered{1}{2N}\hspace{0.0625em}\bigl|\nabla{\textstyle \omega^{J_a}}\bigr|^2\langle {\textstyle \omega^{J_a}}, b \rangle + R \circledast \bigl({\textstyle \omega^{J_a}} \oplus ({\textstyle \omega^{J_a}})^{\circledast\EWzJhvMathVCenter 3}\bigr) \circledast b\biggr)\,\mathrm{d}\mathrm{Vol}_g\]
		where the operators\EWzJhvFootnote{$R \circledast \bigl(\omega^{J_a} \oplus (\omega^{J_a})^{\circledast\EWzJhvMathVCenter 3}\bigr) \circledast b$ actually vanishes if $N \geqslant 2$ (or $N = 1$ and anti-self-dual). To see this, we can use the facts from the footnote\EWzJhvFootRef{ftnPropWeitzen2plus} to \autoref{propWeitzen2plus}, or alternatively, use \EWzJhvEqRef{eqDBarD2n0plus} and the fact that $(\Omega^{2,0,+}, \bar{\partial}^{\mathtt{D}}) \otimes_{\mathbb{C}} {\displaystyle ({\textstyle \Omega^{2,0,+}, \bar{\partial}^{\mathtt{LC}}})^*}$ is holomorphic to get $\mathrm{d}{\displaystyle \mathrm{d}^*}\omega^{J_a} \in \Omega^{1,1}$ and thus \EWzJhvEqRef{eqPfPdeJa} vanishes. In particular, when $(M,g,Q)$ is hyper\EWzJhvText{käh}ler, this PDE is just the harmonic map equation for \hspace{0.1em}$\omega^{J_a} \EWzJhvMathSepColon (M,g) \rightarrow S^2$.} $\EWzJhvCDotInd\circledast\EWzJhvCDotInd$ only depend on\/ $N$.
	\end{EWzJhvLemma}
	\begin{proof}
		Over $U \setminus \mathcal{N}(a)$, we have $b = \bigl(b - \frac{1}{2N}\langle \omega^{J_a}, b \rangle\,\omega^{J_a}\bigr) + \frac{1}{2N}\langle \omega^{J_a}, b \rangle\,\omega^{J_a}$. The first part satisfies $\bigl\langle \omega^{J_a},\: b - \frac{1}{2N}\langle \omega^{J_a}, b \rangle\,\omega^{J_a}\bigr\rangle=0$. $b - \frac{1}{2N}\langle \omega^{J_a}, b \rangle\,\omega^{J_a}$ is also $W^{1,p}_{\ifEWzJhvPDF\raisebox{0pt}[\dimexpr\height-1pt\relax]{$\scriptstyle\mathrm{loc}$}\else\mathrm{loc}\fi}$ over $U$ for any~$p\in[1,\infty)$ and has compact support. \autoref{lmaJaEqPre} applied to $b - \frac{1}{2N}\langle \omega^{J_a}, b \rangle\,\omega^{J_a}$ over $U \setminus \mathcal{N}(a)$ together with the integral form of the Weitzenböck formula {\small (\autoref{propWeitzen2plus})} implies
		\begin{align}
			&\int_U{}\bigl\langle \nabla{\textstyle \omega^{J_a}},\: \nabla \bigl(b - \EWzJhvMathFracVCentered{1}{2N}\langle {\textstyle \omega^{J_a}}, b \rangle\,{\textstyle \omega^{J_a}}\bigr) \bigr\rangle\,\mathrm{d}\mathrm{Vol}_g\notag\\
			={}& N\int_U{}\bigl\langle \mathrm{d}^*{\textstyle \omega^{J_a}},\: \mathrm{d}^* \bigl(b - \EWzJhvMathFracVCentered{1}{2N}\langle {\textstyle \omega^{J_a}}, b \rangle\,{\textstyle \omega^{J_a}}\bigr) \bigr\rangle\,\mathrm{d}\mathrm{Vol}_g\label{eqPfPdeJa}\\
			={}& \EWzJhvMathFracVCentered{1}{2} \int_U{}\Bigl(\bigl\langle \nabla{\textstyle \omega^{J_a}},\: \nabla \bigl(b - \EWzJhvMathFracVCentered{1}{2N}\langle {\textstyle \omega^{J_a}}, b \rangle\,{\textstyle \omega^{J_a}}\bigr) \bigr\rangle + R \circledast {\textstyle \omega^{J_a}} \circledast \bigl(b - \EWzJhvMathFracVCentered{1}{2N}\langle {\textstyle \omega^{J_a}}, b \rangle\,{\textstyle \omega^{J_a}}\bigr)\Bigr)\,\mathrm{d}\mathrm{Vol}_g \EWzJhvDisplayMathPeriod.\notag
		\end{align}
		Then
		\[\int_U{}\bigl\langle \nabla{\textstyle \omega^{J_a}},\: \nabla \bigl(b - \EWzJhvMathFracVCentered{1}{2N}\langle {\textstyle \omega^{J_a}}, b \rangle\,{\textstyle \omega^{J_a}}\bigr) \bigr\rangle\,\mathrm{d}\mathrm{Vol}_g = \int_U{}\Bigl(R \circledast \bigl({\textstyle \omega^{J_a}} \oplus ({\textstyle \omega^{J_a}})^{\circledast\EWzJhvMathVCenter 3}\bigr) \circledast b\Bigr)\,\mathrm{d}\mathrm{Vol}_g \EWzJhvDisplayMathPeriod.\]
		
		\ifEWzJhvPageBreak\pagebreak\fi
		
		Over $U \setminus \mathcal{N}(a)$, because $\langle \omega^{J_a}, \omega^{J_a} \rangle$ is constant, $\bigl\langle \nabla\omega^{J_a}, \omega^{J_a} \bigr\rangle=0$. Therefore,
		\[\int_U{}\bigl\langle \nabla{\textstyle \omega^{J_a}},\: \nabla \bigl(\EWzJhvMathFracVCentered{1}{2N}\langle {\textstyle \omega^{J_a}}, b \rangle\,{\textstyle \omega^{J_a}}\bigr) \bigr\rangle\,\mathrm{d}\mathrm{Vol}_g = \EWzJhvMathFracVCentered{1}{2N}\int_U{}\bigl|\nabla{\textstyle \omega^{J_a}}\bigr|^2 \langle {\textstyle \omega^{J_a}}, b \rangle\,\mathrm{d}\mathrm{Vol}_g \EWzJhvDisplayMathPeriod.\]
		
		We conclude by adding the two equations above.
	\end{proof}
	
	\begin{proof}[\bfseries Proof of \autoref{thm2}]
		We have already proved in \autoref{lmaJaIntegrability} that $J_a$ is integ\EWzJhvText{rab}le over $U \setminus \mathcal{N}(a)$, so we only need to verify that $J_a$ is smooth over $U$. We consider the PDE from \autoref{lmaPdeJa} and invoke elliptic regularity. More specifically, $\omega^{J_a}$ is a $W^{1,2}_{\mathrm{loc}}$ weak solution to the second-order linear elliptic PDE in divergence form over $U$ with term on the right-hand side $\frac{1}{2N}\hspace{0.0625em}\bigl|\nabla\omega^{J_a}\bigr|\ifEWzJhvPDF{\rule{0pt}{1.88ex}}\fi^2\,\omega^{J_a} + R \circledast \bigl(\omega^{J_a} \oplus (\omega^{J_a})^{\circledast\EWzJhvMathVCenter 3}\bigr) \in L^2_{\mathrm{loc}}$, which implies \ifEWzJhvPDF\linebreak\\*[-\baselineskip]\fi $\omega^{J_a}$ is $W^{2,2}_{\mathrm{loc}}$ over $U$. We can now rewrite the PDE as follows:
		\begin{equation}\label{eqCompatibleCplxStr}
			\nabla^*\nabla{\textstyle \omega^{J_a}} = \EWzJhvMathFracVCentered{1}{2N}\hspace{0.0625em}\bigl|\nabla{\textstyle \omega^{J_a}}\bigr|^2\,{\textstyle \omega^{J_a}} + R \circledast \bigl({\textstyle \omega^{J_a}} \oplus ({\textstyle \omega^{J_a}})^{\circledast\EWzJhvMathVCenter 3}\bigr)
		\end{equation}
		We already know $\omega^{J_a}$ is $W^{1,p}_{\mathrm{loc}}$ over $U$ for any~$p\in[1,\infty)$. If\/ $\omega^{J_a}$ is $W^{k,p}_{\mathrm{loc}}$ \ifEWzJhvPDF\vspace{-1pt}\fi over $U$ for any~$p\in(1,\infty)$ for some integer~$k \geqslant 1$, then the right-hand side of \EWzJhvEqRef{eqCompatibleCplxStr} is $W^{k-1,p}_{\mathrm{loc}}$ \ifEWzJhvPDF\vspace{-1pt}\fi over $U$ for any~$p\in(1,\infty)$, so by elliptic regularity, $\omega^{J_a}$ is $W^{k+1,p}_{\mathrm{loc}}$ \ifEWzJhvPDF\vspace{1pt}\fi over $U$ for any~$p\in(1,\infty)$. The smoothness of\/ $\omega^{J_a}$ now follows from induction and Sobolev embedding.
	\end{proof}
	
	The continuity statement can be proved with similar methods:\vspace{-1pt}
	\begin{proof}[\bfseries Proof of \autoref{thm2Continuity}]
		\linespread{1.1}\selectfont
		
		We still work over any open $U \subset M$ satisfying that $\overline{U}$ is compact and $U$ is diffeomorphic to $\mathbb{R}^{4N}$. We choose a complex structure $J_{a_j}$ from the $\mathbb{Z}_2$-complex structure $\tilde{J}_{a_j}$ over $U$ for each~$j$. We can then define $a_j^{2,0}$ given by \EWzJhvEqRef{eqA20} over $U$. Because $\lvert a_j\rvert_{g_j}=\sqrt{2}\mskip2mu\relax\bigl|a_j^{2,0}\bigr|_{g_j}$ \ifEWzJhvPDF\vspace{-2pt}\fi and ${\displaystyle \mathrm{d}_{A_j,g_j}^*}a_j^{2,0}=0$ (\autoref{lmaA20eq}), elliptic regularity implies $a_j^{2,0}$ is bounded\EWzJhvFootnote{The $C^k,L^p,W^{k,p}$ norms over $U'$ are defined with respect to some fixed Riemannian metric on $M$ that is independent of $j$, e.g.\ $g_\infty$.} in $C^k$ over $U'$ for any integer~$k$ and any open\EWzJhvFootnote{$U' \subset\subset U$ means $\overline{U'}$ is a compact subset of\/ $U$.} $U' \subset\subset U$.
		
		As $\overline{U}$ is assumed to be compact, there exists some positive integer~$m$ satisfying $a_\infty\rvert_U$ only vanishes at zeros of order~$\leqslant m$ {\small \cite{UniqueContinuationGeometry}}. \autoref{lmaFMinus1n} now applies to \ifEWzJhvPDF\vspace{-3pt}\fi $\lvert a_j\rvert_{g_j}^2$ over $U$ and implies that for any open $U' \subset\subset U$, $\bigl\{\lvert a_j\rvert_{g_j}^{-\frac{1}{2m}}\bigr\}_{j \geqslant j_{\oldstylenums{0}}}$ \ifEWzJhvPDF\vspace{-1pt}\fi is bounded in $L^1$ over $U'$ for some $j_{\oldstylenums{0}} \in \mathbb{N}$. Note $\lvert a_j\rvert_{g_j}=\sqrt{2}\mskip2mu\relax\bigl|a_j^{2,0}\bigr|_{g_j}$. \autoref{corHigherDer} with $n$ large now implies $\{J_{a_j}\}_j$ is bounded in $W^{1,p}$ over $U'$ for any~$p \in [1,\infty)$ and any open $U' \subset\subset U$. If for some positive integer~$k$, $\{J_{a_j}\}_j$ is bounded in $W^{k,p}$ over $U'$ for any~$p \in (1,\infty)$ and any open $U' \subset\subset U$, then the right-hand side of \EWzJhvEqRef{eqCompatibleCplxStr} for $J_{a_j}$ with $j\in\mathbb{N}\cup\{\infty\}$ is bounded in $W^{k-1,p}$ over $U'$ for any~$p \in (1,\infty)$ and any open $U' \subset\subset U$, so by elliptic regularity, $\{J_{a_j}\}_j$ is bounded in $W^{k+1,p}$ over $U'$ for any~$p \in (1,\infty)$ and any open $U' \subset\subset U$. By induction and Sobolev embedding, $\{J_{a_j}\}_j$ is compact in $C^\infty$ over $U'$ for any open $U' \subset\subset U$.
		
		By the definition of $\tilde{J}_{a_j}$, $\tilde{J}_{a_j}$ converges to $\tilde{J}_{a_\infty}$ in $\displaystyle C^\infty_{\mathrm{loc}}$ over $M \setminus \mathcal{N}(a_\infty)$. Due to the compactness in $C^\infty$ over $U'$ for any open $U' \subset\subset U$, $\tilde{J}_{a_j}$ must converge to $\tilde{J}_{a_\infty}$ in $C^\infty$ over $U'$ for any open $U' \subset\subset U$.\qedhere\par
	\end{proof}
	
	\hypersetup{next-anchor=notesPdeJa}\subsubsection*{Additional notes to \autoref{lmaPdeJa}}\label{subsubsecLmaPdeJaNotes}
	
	\autoref{lmaPdeJa} only describes a PDE \emph{satisfied by} $J_a$. It is not a \emph{sufficient} condition for a compatible almost complex structure on $(M,g,Q)$ with $\mathrm{Hol}(g) \subset \mathrm{Sp}(N)\hspace{0.0625em}\mathrm{Sp}(1)$ to be integ\EWzJhvText{rab}le, even if $M$ is assumed to be compact. The logic behind this is that, basically, this PDE follows from the fact implied by \autoref{lmaJaEqPre} that $\omega^{J_a}$ is a critical point \ifEWzJhvPageBreak\pagebreak\fi of the functional given by the integral of $\displaystyle \bigl|\nabla \omega\bigr|\ifEWzJhvPDF{\rule{0pt}{1.88ex}}\fi^2 - N\hspace{0.0625em}\bigl|\mathrm{d}^* \omega\bigr|\ifEWzJhvPDF{\rule{0pt}{1.88ex}}\fi^2$ for~$\omega \in \Omega^{2,+}$ with $\lvert\omega\rvert=\sqrt{2N}$, but the integ\EWzJhvText{rab}ility of some $J$, in light of \autoref{propIntegrabEq}, means $\omega^J$ is a zero (and minimum) point of this functional. For a concrete example, on a flat $4N$\hspace*{-0.03em}-torus denoted by $T^{4N}$, this PDE is just the harmonic map equation from $T^{4N}$ to $S^2$, and nonconstant harmonic maps from $T^{4N}$ to $S^2$ can be given by nonconstant meromorphic functions with respect to a Kähler structure on $T^{4N}$, so nonparallel solutions to this PDE may exist on a suitable $T^{4N}$, but as the remark following \autoref{propIntegrabEq} indicates, $T^{4N}$ as a compact hyper\EWzJhvText{käh}ler manifold admits no nonparallel compatible complex structure.
	
	On the other hand, by working backwards, we can easily see that, when $N = 1$, if a compatible almost complex structure $J$ on $(M,g,Q)$ satisfies this PDE \emph{for any Riemannian metric conformal to\/ $g$}, then $\omega^J$ must be\EWzJhvFootnote{Here, we may use \EWzJhvTextNormal{(23)} of \cite[Remark~2]{GauduchonCplxStrCConfMNegT} to express $\bigl|\nabla \omega^J\bigr|\ifEWzJhvPDF{\rule{0pt}{1.88ex}}\fi^2 - \bigl|{\displaystyle \mathrm{d}^*}\omega^J\bigr|\ifEWzJhvPDF{\rule{0pt}{1.88ex}}\fi^2$ in terms of the norm of the Nijenhuis tensor of $J$.} a critical point of the functional given by the integral of $\displaystyle \bigl|\nabla \omega\bigr|\ifEWzJhvPDF{\rule{0pt}{1.88ex}}\fi^2 - \bigl|\mathrm{d}^* \omega\bigr|\ifEWzJhvPDF{\rule{0pt}{1.88ex}}\fi^2$ times any positive function over $M$, so the conclusion of \autoref{lmaJaEqPre} must hold at every point in $M$, and thus $J$ must be integ\EWzJhvText{rab}le.
	
	As a side note, this method of deriving an elliptic PDE\EWzJhvFootnote{If we forget about \autoref{lmaJaEqPre} for a moment, the integ\EWzJhvText{rab}ility condition of a compatible almost complex structure over a smooth Riemannian manifold itself is also an (injectively) elliptic PDE (almost-holomorphicity as a map from the almost hermitian manifold to the twistor space). It is possible that such a PDE can replace \autoref{lmaPdeJa} in the proof of \autoref{thm2},~\ref{thm2Continuity}.} satisfied by a compatible complex structure through \autoref{lmaJaEqPre} does not seem to generalise directly to general smooth Riemannian manifolds. Although \autoref{lmaJaEqPre} can be generalised to show that $\omega^J$ is a critical point of the functional given by the integral of $\bigl|\nabla \omega\bigr|\ifEWzJhvPDF{\rule{0pt}{1.88ex}}\fi^2-\bigl|\mathrm{d} \omega\bigr|\ifEWzJhvPDF{\rule{0pt}{1.88ex}}\fi^2$ (see the footnote\EWzJhvFootRef{ftnLmaJaEqPre} to \autoref{lmaJaEqPre}), it is not clear how this leads to an elliptic PDE---we do not expect something like \autoref{propWeitzen2plus} that implies the ellipticity of\/ $\displaystyle \nabla^*\nabla - \mathrm{d}^*\mathrm{d} \approx \mathrm{d}\mathrm{d}^*$ to hold in general. We note as a consequence of \EWzJhvTextNormal{(23)} of \cite[Remark~2]{GauduchonCplxStrCConfMNegT} that, in the case of $\mathrm{Hol}(g) \subset \mathrm{Sp}(N)\hspace{0.0625em}\mathrm{Sp}(1)$, $\bigl|\nabla \omega^J\bigr|\ifEWzJhvPDF{\rule{0pt}{1.88ex}}\fi^2 - \bigl|\mathrm{d} \omega^J\bigr|\ifEWzJhvPDF{\rule{0pt}{1.88ex}}\fi^2$ is proportional to the square of the norm \ifEWzJhvPDF\linebreak\\*[-\baselineskip]\fi of the Nijenhuis tensor, but in general, the relation between $\bigl|\nabla \omega^J\bigr|\ifEWzJhvPDF{\rule{0pt}{1.88ex}}\fi^2 - \bigl|\mathrm{d} \omega^J\bigr|\ifEWzJhvPDF{\rule{0pt}{1.88ex}}\fi^2$ and the Nijenhuis tensor is more convoluted. In light of this, to generalise \autoref{lmaPdeJa}, we may try calculating the variation of the square of the norm of the Nijenhuis tensor instead. Note how remar\EWzJhvText{kab}le the relation between $\bigl|\nabla \omega^J\bigr|\ifEWzJhvPDF{\rule{0pt}{1.88ex}}\fi^2 - \bigl|\mathrm{d} \omega^J\bigr|\ifEWzJhvPDF{\rule{0pt}{1.88ex}}\fi^2$ and the Nijenhuis tensor is: the operators $\nabla, \mathrm{d}$ are linear, but the Nijenhuis tensor is rather nonlinear in $J$.
	
	\section{Case of general even-dimensional manifolds}\label{secGeneral}
	
	Contrary to the conclusive results in the previous section, our study of the case of general even-dimensional smooth manifolds in this section is rather rudimentary. Let $M$ be a $2N$\hspace*{-0.03em}-dimensional smooth manifold. Let $V \rightarrow M$ be a smooth real vector bundle with inner product. Let $a \in \Omega^n(V)$ be a smooth $V$\hspace*{-0.03em}-valued $n$-form. We make the following definition, which is a generalisation\EWzJhvFootnote{The exact relation between \autoref{defn2plusParascalar},~\ref{defnParascalarN} is described in \autoref{corParascalar2plusVsN}, which also confirms the consistency between the two definitions. Still, to be rigorous, we declare that, in this section, the phrase ``parascalar $({\cdots},0)$-form'' refers to \autoref{defnParascalarN}, unless otherwise specified.} of \autoref{defn2plusParascalar}:
	\begin{EWzJhvDefinition}\label{defnParascalarN}
		Differential form\/ $a \in \Omega^n(V)$ is said to be a \emph{parascalar $(n,0)$-form}, if for every point\/ $p \in M$ with\/ $a\rvert_p \neq 0$, there exist a linear complex structure on\/ $T_p M$, \ifEWzJhvPDF\linebreak\fi a real\/ $2$-dimensional subspace\/ $\hat{V}_p$ of\/ $V\rvert_p$, and a linear complex structure on\/ $\hat{V}_p$ compatible with the restricted inner product, such that\/ $a\rvert_p \in \Lambda^n(\EWzJhvTpStar Tp M) \otimes_{\mathbb{R}} \hat{V}_p$ and furthermore, $a\rvert_p \in \Lambda^{n,0}(\EWzJhvTpStar Tp M \otimes_{\mathbb{R}} \mathbb{C}) \otimes_{\mathbb{C}} \hat{V}_p \subset \Lambda^n_{\mathbb{C}}(\EWzJhvTpStar Tp M \otimes_{\mathbb{R}} \mathbb{C}) \otimes_{\mathbb{C}} \hat{V}_p = \Lambda^n(\EWzJhvTpStar Tp M) \otimes_{\mathbb{R}} \hat{V}_p$.
		
		In case\/ $M$ is endowed with a Riemannian metric {\small (or\/ $(g,Q)$ with\/ $\mathrm{Hol}(g) \subset \mathrm{Sp}(\frac{N}{2})\hspace{0.0625em}\mathrm{Sp}(1)$, etc.)}, the parascalar\/ $(n,0)$-form is said to be compatible with the metric {\small (or \ldots)}, if for every point\/ $p \in M$ with\/ $a\rvert_p \neq 0$, the linear complex structure on\/ $T_p M$ satisfying the above condition can be chosen to be compatible with the metric {\small (or \ldots)}.\ifEWzJhvPageBreak\pagebreak\fi
	\end{EWzJhvDefinition}
	
	\begin{EWzJhvProposition}\label{propParascalarNInducesCplxStr}
		Let\/ $a$ be a parascalar\/ $(n,0)$-form. Let\/ $p \in M$. If the map from\/ $T_p M$ to\/ $\Lambda^{n-1}(\EWzJhvTpStar Tp M) \otimes_{\mathbb{R}} V\rvert_p$ given by\/ $v \mapsto \iota_v a\rvert_p$ is injective, then the linear complex structure on\/ $T_p M$ from \autoref{defnParascalarN} is unique up to sign.
	\end{EWzJhvProposition}
	Then we say this linear complex structure on $T_p M$ is \emph{induced} by $a$ (up to sign).\ifEWzJhvPageBreak\vspace{-0pt plus -6pt}\fi
	\begin{proof}
		For any linear complex structure on $T_p M$ from \autoref{defnParascalarN}, the injection from $T_p M$ to $\Lambda^{n-1}(\EWzJhvTpStar Tp M) \otimes_{\mathbb{R}} V\rvert_p$ given by $v \mapsto \iota_v a\rvert_p$ is actually a $\mathbb{C}$-linear embedding from $T_p M$ as a complex vector space to $\Lambda^{n-1}(\EWzJhvTpStar Tp M) \otimes_{\mathbb{R}} \hat{V}_p$ whose linear complex structure comes from the one on $\hat{V}_p$ (not the one on $\EWzJhvTpStar Tp M$). The real $2$-dimensional subspace $\hat{V}_p \subset V\rvert_p$ is the image of the map $(v_1,v_2,\ldots,v_n) \mapsto \iota_{v_1}\iota_{v_2}\cdots\iota_{v_n} a\rvert_p$ where $v_j \in T_p M$. The linear complex structure on $\hat{V}_p$ compatible with the restricted inner product is unique up to sign. Therefore, the linear complex structure on $T_p M$ from \autoref{defnParascalarN} is unique up to sign.
	\end{proof}
	
	Now, we present a sufficient condition for the integ\EWzJhvText{rab}ility of the induced $\mathbb{Z}_2$-almost complex structure. Let $V \rightarrow M$ be endowed with a connection $A$, compatible with the inner product on $V$.
	
	\begin{EWzJhvProposition}\label{propParascalarNClosedToIntegrab}
		Let\/ $a$ be a parascalar\/ $(n,0)$-form with\/ $\mathrm{d}_A a = 0$. Let\/ $p \in M$. If the map from\/ $T_p M$ to\/ $\Lambda^{n-1}(\EWzJhvTpStar Tp M) \otimes_{\mathbb{R}} V\rvert_p$ given by\/ $v \mapsto \iota_v a\rvert_p$ is injective, then the\/ $\mathbb{Z}_2$-almost complex structure induced by\/ $a$ over some open neighbourhood of\/\hspace{-0.08em} $p$ in\/ $M$ is integrable.
	\end{EWzJhvProposition}
	\begin{proof}
		Fix an almost complex structure induced by $a$ over a small open neighbourhood $U \ni p$. Note $a\rvert_{p'} \in \Lambda^{n,0}(\EWzJhvTpStar T{p'} \mskip-2mu\relax M \otimes_{\mathbb{R}} \mathbb{C}) \otimes_{\mathbb{C}} \hat{V}_{p'}$ in \autoref{defnParascalarN} for~$p' \in U$. Denote by $\hat{V}$ the real rank-$2$ subbundle formed by $\hat{V}_{p'}$ for~$p' \in U$. For clarity, let $\hat{a}$ be $a$, but $\hat{a}$ is regarded as an $(n,0)$-form taking value in the complex line bundle $\hat{V}$. Through the orthogonal decomposition $V = \hat{V} \oplus \hat{V}^\perp$, the connection $A$ restricts to a $\mathbb{C}$-linear connection $\hat{A}$ over $\hat{V}$, and $\mathrm{d}_{\hat{A}} \hat{a} = 0$. For any $(0,1)$-vector fields $X,Y$ over $U$, we compute with the help of Cartan's magic formula:
		\[0 = \iota_X \iota_Y \mathrm{d}_{\hat{A}} \hat{a} = \iota_X \bigl(\mathscr{L}^{\hat{A}}_Y \hat{a} - \mathrm{d}_{\hat{A}} \iota_Y \hat{a}\bigr) = \iota_X \mathscr{L}^{\hat{A}}_Y \hat{a} = \mathscr{L}^{\hat{A}}_Y \iota_X \hat{a} - \iota_{[Y,X]} \hat{a} = \iota_{[X,Y]} \hat{a}\]
		Because for~$p' \in U$, $v \mapsto \iota_v a\rvert_{p'}$ is a $\mathbb{C}$-linear embedding from $T_{p'} \mskip-2mu\relax M$ as a complex vector space to $\Lambda^{n-1}(\EWzJhvTpStar T{p'} \mskip-2mu\relax M) \otimes_{\mathbb{R}} \hat{V}_{p'}$, $[X,Y]$ must be of type~$(0,1)$ over $U$. This shows that the Lie bracket of any two $(0,1)$-vector fields over $U$ is still of type~$(0,1)$, then we can apply the Newlander--Nirenberg theorem to show the integ\EWzJhvText{rab}ility of the almost complex structure over $U$.
	\end{proof}
	
	When $a$ is a nonvanishing parascalar $(N,0)$-form (recall $N$ is the ``complex \ifEWzJhvPDF\linebreak\fi dimension'' of the even-dimensional manifold), the injectivity of the map from $T_p M$ to $\Lambda^{N-1}(\EWzJhvTpStar Tp M) \otimes_{\mathbb{R}} V\rvert_p$ given by $v \mapsto \iota_v a\rvert_p$ is automatic. As a result, we get
	
	\begin{EWzJhvCorollary}\label{corParaVolumeForm}
		Let\/ $a$ be a parascalar\/ $(N,0)$-form. Let\/ $p \in M$. If\/\hspace{-0.08em} $a\rvert_p \neq 0$, then the linear complex structure on\/ $T_p M$ from \autoref{defnParascalarN} is unique up to sign. It then forms a\/ $\mathbb{Z}_2$-almost complex structure on the complement of the zero locus of\/\hspace{-0.08em} $a$. If furthermore, $\mathrm{d}_A a = 0$, then this\/ $\mathbb{Z}_2$-almost complex structure is integrable.
	\end{EWzJhvCorollary}
	
	In real dimension~$4$, this is basically \autoref{lmaJaIntegrability}.
	
	\ifEWzJhvPageBreak\flushbottom\enlargethispage{-1.5cm}\pagebreak\raggedbottom\fi
	
	\subsection{A counterexample to smooth extensibility}\label{subsecCounterexampleSmoothExtensib}
	
	We have just obtained a sufficient condition for the integ\EWzJhvText{rab}ility of the $\mathbb{Z}_2$-almost complex structure defined wherever the injectivity holds. However, assuming $\mathrm{d}_A a = 0$, or even assuming compatibility with a Riemannian metric\EWzJhvFootnote{Without the assumption of compatibility with a Riemannian metric, wilder counterexamples exist. For example, let $a := \exp\bigl(-x^{-1}-iy\bigr)\,\mathrm{d}\bigl(x^{-1}+iy\bigr)$ for~$(x,y) \in (0,\infty)\times\mathbb{R}$, with $V \rightarrow \mathbb{R}^2$ being the trivial rank-$2$ real vector bundle with fibres identified with $\mathbb{C}$. We easily check $a$ is a closed parascalar $(1,0)$-form that zero extends to $\mathbb{R}^2$ smoothly, but its induced $\mathbb{Z}_2$-complex structure does not extend \ifEWzJhvPDF\nolinebreak\fi to \ifEWzJhvPDF\nolinebreak\fi $\mathbb{R}^2$.} in addition to $\mathrm{d}_A a = 0$, it is not true in general that this integ\EWzJhvText{rab}le $\mathbb{Z}_2$-(almost) complex structure extends smoothly to the whole $M$. For example, let $M$ be $\mathbb{C}^2 \times \mathbb{R}^4$ with the standard metric. We use the coordinates $(z^1,z^2,x^1,x^2,x^3,x^4) \in \mathbb{C}^2 \times \mathbb{R}^4$. Similar to \EWzJhvEqRef{eqOrthonBasis2plus}, write
	\begin{gather*}
		\omega^1 := \mathrm{d}x^2 \wedge \mathrm{d}x^3 + \mathrm{d}x^1 \wedge \mathrm{d}x^4\\
		\omega^2 := \mathrm{d}x^3 \wedge \mathrm{d}x^1 + \mathrm{d}x^2 \wedge \mathrm{d}x^4\\
		\omega^3 := \mathrm{d}x^1 \wedge \mathrm{d}x^2 + \mathrm{d}x^3 \wedge \mathrm{d}x^4 \EWzJhvDisplayMathPeriod.
	\end{gather*}
	Our counterexample is given by the following expression:
	\begin{equation}
		a := \mathrm{d}z^1 \wedge \mathrm{d}z^2 \wedge \Bigl(i\bigl((z^1)^2+(z^2)^2\bigr)\,\omega^1 - \bigl((z^1)^2-(z^2)^2\bigr)\,\omega^2 + 2z^1z^2\,\omega^3\Bigr)
	\end{equation}
	Note that the sum of the squares of the coefficients of\/ $\omega^1, \omega^2, \omega^3$ is zero, so the real parts and the imaginary parts of these three coefficients form two orthogonal vectors in $\mathbb{R}^3$ with equal norm. This guarantees that for any~$(z^1,z^2) \in \mathbb{C}^2$, there are $\lambda \in [0,\infty)$ and some positively-oriented orthonormal coordinates of\/ $\mathbb{R}^4$ denoted by $(x'^1\,\,x'^2\,\,x'^3\,\,x'^4)$ such that for these particular values of $z_1,z_2$,
	\begin{align*}
		a\rvert_{z^1,z^2} &= \lambda\,\mathrm{d}z^1 \wedge \mathrm{d}z^2 \wedge \bigl(i\,(\mathrm{d}x'^2 \wedge \mathrm{d}x'^3 + \mathrm{d}x'^1 \wedge \mathrm{d}x'^4) - (\mathrm{d}x'^3 \wedge \mathrm{d}x'^1 + \mathrm{d}x'^2 \wedge \mathrm{d}x'^4)\bigr)\\&= \lambda\,\mathrm{d}z^1 \wedge \mathrm{d}z^2 \wedge (\mathrm{d}x'^1 + i\,\mathrm{d}x'^2) \wedge (\mathrm{d}x'^3 + i\,\mathrm{d}x'^4) \EWzJhvDisplayMathPeriod.
	\end{align*}
	Therefore, $a$ is a parascalar $(4,0)$-form compatible with the standard metric. We easily verify $\mathrm{d}a=0$.
	
	For~$z^2=0$,\ifEWzJhvPDF\vspace{\glueexpr-\baselineskip/4\relax}\fi
	\[a\rvert_{z^1,z^2:=0} = (z^1)^2\,\mathrm{d}z^1 \wedge \mathrm{d}z^2 \wedge (\mathrm{d}x^1 + i\,\mathrm{d}x^2) \wedge (\mathrm{d}x^3 + i\,\mathrm{d}x^4) \EWzJhvDisplayMathPeriod.\]
	On the other hand, for~$z^2=z^1$,
	\[a\rvert_{z^1,z^2:=z^1} = -2i(z^1)^2\,\mathrm{d}z^1 \wedge \mathrm{d}z^2 \wedge (\mathrm{d}x^3 + i\,\mathrm{d}x^1) \wedge (\mathrm{d}x^2 + i\,\mathrm{d}x^4) \EWzJhvDisplayMathPeriod.\]
	So the induced $\mathbb{Z}_2$-complex structure is not continuous at points with $z^1=z^2=0$.
	
	This is a counterexample in complex dimension~$4$. The zero locus of this $a$ is of complex codimension~$2$. The author does not know if there is any counterexample in complex dimension~$3$.
	
	\subsection{Criterion given a prespecified almost complex structure}
	
	In this short supplementary subsection, we prove a criterion for a bundle-valued \ifEWzJhvPDF\linebreak\fi $n$-form to be a parascalar $(n,0)$-form, assuming the form satisfies a certain condition with respect to a prespecified almost complex structure. This is sometimes useful.
	
	\begin{EWzJhvProposition}\label{propParascalarNAlCplxMfld}
		Let\/ $(M,J)$ be an almost complex manifold. Let\/ $V \rightarrow M$ be a smooth real vector bundle with inner product. Let\/ $\omega$ be a smooth nowhere vanishing section of\/\hspace{0em plus -0.08em} $\Omega^{n,0}$ with\/ $n \geqslant 2$. Let\/ $\xi_1,\xi_2 \in \Gamma(V)$. Let\/ $a := \xi_1 \otimes_{\mathbb{R}} (\omega + \overline{\omega}) + \xi_2 \otimes_{\mathbb{R}} i(\omega - \overline{\omega}) \in \Omega^n(V)$. Then\/ $a$ is a parascalar\/ $(n,0)$-form {\small (\autoref{defnParascalarN})} (not assumed to induce the almost complex structure\/ $J$) if and only if\/\hspace{-0.08em} $\lvert\xi_1\rvert=\lvert\xi_2\rvert$ and\/ $\langle \xi_1, \xi_2 \rangle=0$.\ifEWzJhvPageBreak\pagebreak\fi
	\end{EWzJhvProposition}
	\begin{proof}
		Assuming $\lvert\xi_1\rvert=\lvert\xi_2\rvert$ and $\langle \xi_1, \xi_2 \rangle=0$, it is straightforward to verify that $a$ is a parascalar $(n,0)$-form.
		
		Assume $a$ is a parascalar $(n,0)$-form. Let $p \in M$. If $a\rvert_p=0$, then $\lvert\xi_1\rvert=\lvert\xi_2\rvert$ and $\langle \xi_1, \xi_2 \rangle=0$ at $p$ are trivially true. Therefore, we assume $a\rvert_p \neq 0$. Then $\xi_1,\xi_2$ cannot both \ifEWzJhvPDF\linebreak\\*[-\baselineskip]\fi vanish at $p$. As $\omega$ is assumed to be nowhere vanishing, there exist $v_1,v_2,\ldots,v_n \in T_p^{1,0}\mskip-2mu\relax M$ satisfying $\iota_{v_1}\iota_{v_2}\cdots\iota_{v_n} \omega\rvert_p \neq 0$. Let $J_a$ be a linear complex structure on $T_pM$ and $J_V$ be a linear complex structure on $\hat{V}_p \subset V\rvert_p$ from \autoref{defnParascalarN} due to $a$ being a parascalar $(n,0)$-form. As $n \geqslant 2$, $\iota_{v_2}\cdots\iota_{v_n} a\rvert_p = (\xi_1 + i\xi_2) \otimes_{\mathbb{C}} \iota_{v_2}\cdots\iota_{v_n} \omega\rvert_p$. Then $\iota_{v_1}\iota_{v_2}\cdots\iota_{v_n} a\rvert_p \neq 0$. \autoref{defnParascalarN} implies $\iota_{J_a v_1}\iota_{v_2}\cdots\iota_{v_n} a\rvert_p = J_V \iota_{v_1}\iota_{v_2}\cdots\iota_{v_n} a\rvert_p$, so $\iota_{J_a v_1}\iota_{v_2}\cdots\iota_{v_n} a\rvert_p \neq 0$ and thus $\iota_{J_a v_1}\iota_{v_2}\cdots\iota_{v_n} \omega\rvert_p \neq 0$. Also as a consequence, $\langle \iota_{v_1}\iota_{v_2}\cdots\iota_{v_n} a,\: \iota_{J_a v_1}\iota_{v_2}\cdots\iota_{v_n} a \rangle\rvert_p =0$, where $\displaystyle \langle\EWzJhvCDotInd,\EWzJhvCDotInd\rangle \in (\hat{V}_p^*)^{\otimes\EWzJhvMathVCenter 2}\otimes_{\mathbb{R}}\mathbb{C}$ is \emph{the\/ $\mathbb{C}$-linear extension} of the inner product on $\hat{V}_p$. Then $\bigl\langle (\xi_1 + i\xi_2)\,\iota_{v_1}\iota_{v_2}\cdots\iota_{v_n} \omega,\: (\xi_1 + i\xi_2)\,\iota_{J_a v_1}\iota_{v_2}\cdots\iota_{v_n} \omega \bigr\rangle\bigr|_p =0$, so $\langle \xi_1 + i\xi_2,\: \xi_1 + i\xi_2 \rangle\rvert_p =0$, \ifEWzJhvPDF\linebreak\\*[-\baselineskip]\fi which is just $\lvert\xi_1\rvert=\lvert\xi_2\rvert$ and $\langle \xi_1, \xi_2 \rangle=0$ at $p$.
	\end{proof}
	
	As a corollary, we can now confirm the consistency between the two definitions of ``parascalar $({\cdots},0)$-forms'', \autoref{defn2plusParascalar},~\ref{defnParascalarN}:
	
	\begin{EWzJhvCorollary}\label{corParascalar2plusVsN}
		Let\/ $(M,g,Q)$ be a smooth\/ $4N$\hspace*{-0.08em}-dimensional Riemannian manifold with holonomy\/ $\mathrm{Hol}(g) \subset \mathrm{Sp}(N)\hspace{0.0625em}\mathrm{Sp}(1)$. Let\/ $V \rightarrow M$ be a smooth real vector bundle with inner product. Let\/ $a \in \Omega^{2,+}(V)$. The following are equivalent:\vspace{-\parskip}
		\begin{enumerate}[topsep=1ex plus 0.25ex minus 0.25ex, itemsep=1ex plus 0.25ex minus 0.25ex, parsep=0pt, font=\fontshape{ui}\selectfont]
			\item $a$ is a parascalar\/ $(2,0)$-form according to \autoref{defn2plusParascalar};
			\item $a$ is a parascalar\/ $(2,0)$-form according to \autoref{defnParascalarN};
			\item $a$ is a parascalar\/ $(2,0)$-form compatible with\/ $(g,Q)$ according to \autoref{defnParascalarN}.
		\end{enumerate}
	\end{EWzJhvCorollary}
	\begin{proof}
		We assume that at least one of these three conditions holds, and restrict our attention to the complement of the zero locus of $a$. Any of these three conditions implies that $a$ takes value in a rank-$2$ subbundle of\/ $V$, so $a$ also takes value in a rank-$2$ subbundle of\/ $\Omega^{2,+}$. As a consequence, locally, we can write $a = \xi_1 \otimes_{\mathbb{R}} (\omega^+ + \overline{\omega^+}) + \xi_2 \otimes_{\mathbb{R}} i(\omega^+ - \overline{\omega^+})$ for some $\xi_1,\xi_2 \in \Gamma(V)$ where $\omega^+ \in \Omega^{2,0,+}$ is as in \EWzJhvEqRef{eqACoord} with respect to some compatible almost complex structure. Then we use \autoref{propParascalarNAlCplxMfld} to conclude.
	\end{proof}
	
	\section{Implications for Vafa--Witten equation}\label{secVW}
	
	In this section, we shall observe a close relation between ``$\mathrm{U}(1)$-invariant'' or ``nilpotent'' solutions to the Vafa--Witten equation and the notion of parascalar $(N,0)$-forms described in the previous sections, and finish with some discussions.
	
	\EWzJhvMiniSubsection*{Preliminaries: The Vafa--Witten equation}
	
	Let $(M,g)$ be a smooth oriented Riemannian $4$-manifold. Let $G$ be a compact Lie group, $\mathrm{SU}(2)$ and $\mathrm{SO}(3)$ being two simple cases. Let $P \rightarrow M$ be a $G$-principal bundle over $M$. Denote the space of principal connections on $P \rightarrow M$ by $\mathscr{A}(P)$. \ifEWzJhvPDF\linebreak\fi Denote the space of self-dual $2$-forms valued in the adjoint bundle by $\Omega^{2,+}(\mathrm{ad}(P))$. The Vafa--Witten equation \cite{VAFA19943}, following \cite{MaresPhD}, is the following equation system on a pair\EWzJhvFootnote{\label{ftnEqVWIgnore}We ignore the seemingly less significant $\Gamma(\mathrm{ad}(P))$ component of a solution, which vanishes in many interesting cases anyway \cite[Remark~2.1.2]{MaresPhD}.} $(A,a) \in \mathscr{A}(P) \times \Omega^{2,+}(\mathrm{ad}(P))$
	\begin{subequations}\label{eqVW}
		\renewcommand*{\theequation}{\theparentequation\EWzJhvOptionallyFixedWidth{0.52em}{\alph{equation}}}
		\let\EWzJhvOptionallyFixedWidth=\EWzJhvOptionallyFixedWidthEnabled
		\begin{align}
			\mathrm{d}_A^* a &= 0\label{eqVWa}\\
			F_A^+ + \EWzJhvMathFracVCentered{1}{8}\hspace{0.0625em}[a \centerdot a] &= 0
		\end{align}
	\end{subequations}
	where $F_A^+ \in \Omega^{2,+}(\mathrm{ad}(P))$ is the self-dual part of the curvature of $A$, $[\EWzJhvCDotInd\centerdot\EWzJhvCDotInd]$ is the \ifEWzJhvPageBreak\pagebreak\fi following antisymmetric linear map $\mskip2mu\relax \EWzJhvCDotInd\centerdot\EWzJhvCDotInd \mskip2mu\relax\EWzJhvMathSepColon\mskip2mu\relax \Omega^{2,+}\rvert_p \otimes \Omega^{2,+}\rvert_p \mskip2mu\relax\rightarrow\mskip2mu\relax \Omega^{2,+}\rvert_p \mskip2mu\relax$ tensored with the Lie bracket on the adjoint bundle: given an oriented orthonormal basis $(e^1\,\,e^2\,\,e^3\,\,e^4)$ of\/ $\EWzJhvTpStar Tp M$, which induces the following\EWzJhvFootnote{This is a special case of \EWzJhvEqRef{eqOrthonBasis2plus}.} orthonormal basis $(\omega^1\,\,\omega^2\,\,\omega^3)$ of\/ $\Omega^{2,+}\rvert_p$
	\begin{equation}
		\begin{gathered}
			\omega^1=\EWzJhvMathFracVCentered{1}{\sqrt{2}\,}\bigl(e^2 \wedge e^3 + e^1 \wedge e^4\bigr)\\
			\omega^2=\EWzJhvMathFracVCentered{1}{\sqrt{2}\,}\bigl(e^3 \wedge e^1 + e^2 \wedge e^4\bigr)\\
			\omega^3=\EWzJhvMathFracVCentered{1}{\sqrt{2}\,}\bigl(e^1 \wedge e^2 + e^3 \wedge e^4\bigr) \EWzJhvDisplayMathPeriod,
		\end{gathered}
	\end{equation}
	the antisymmetric linear map $\EWzJhvCDotInd\centerdot\EWzJhvCDotInd$ is given by
	\[
		\omega^1 \centerdot \omega^2 = -\sqrt{2}\,\omega^3,\quad
		\omega^2 \centerdot \omega^3 = -\sqrt{2}\,\omega^1,\quad
		\omega^3 \centerdot \omega^1 = -\sqrt{2}\,\omega^2 \EWzJhvDisplayMathPeriod.
	\]
	
	In case $(M,g)$ is a Kähler surface, $a \in \Omega^{2,+}(\mathrm{ad}(P))$ splits into $a = \beta - \beta^* + \gamma \otimes_{\mathbb{R}} \omega$, where $\beta \in \Omega^{2,0}(\mathrm{ad}(P) \otimes_{\mathbb{R}} \mathbb{C})$, $\gamma \in \Gamma(\mathrm{ad}(P))$, $\omega = \frac{i}{2}(\mathrm{d}z^1\wedge\mathrm{d}\bar{z}^1 + \mathrm{d}z^2\wedge\mathrm{d}\bar{z}^2)$ is the Kähler form, and $\beta^*$ makes $\beta - \beta^* \in \Omega^{2,+}(\mathrm{ad}(P))$ and $i(\beta + \beta^*) \in \Omega^{2,+}(\mathrm{ad}(P))$.
	
	\begin{EWzJhvProposition}[{\cite[Theorem~7.1.2]{MaresPhD}}]\label{propVWequivalenceKaehler}
		If\/\hspace{-0.08em} $(M,g)$ is a compact Kähler surface, then \EWzJhvEqRef{eqVW} is equivalent to\EWzJhvFootnote{$\gamma^*$ ($=-\gamma$ in our setting) is a remnant of the ignored\EWzJhvFootRef{ftnEqVWIgnore} $\Gamma(\mathrm{ad}(P))$ component of a solution, which when considered would make $\gamma \in \Gamma(\mathrm{ad}(P) \otimes_{\mathbb{R}} \mathbb{C})$ instead of\/ $\Gamma(\mathrm{ad}(P))$.}\ifEWzJhvPDF\vspace{\glueexpr-\baselineskip/2\relax}\fi
		\begin{subequations}
			\renewcommand*{\theequation}{\theparentequation\EWzJhvOptionallyFixedWidth{0.52em}{\alph{equation}}}
			\let\EWzJhvOptionallyFixedWidth=\EWzJhvOptionallyFixedWidthEnabled
			\begin{gather}
				F_A^{0,2} = 0\label{eqVWKaehlerA}\\
				\mathrm{d}_A\beta = 0\label{eqVWKaehlerB}\\
				\omega \wedge i F_A + \EWzJhvMathFracVCentered{1}{2}\hspace{0.0625em}[\beta \wedge \beta^*] = 0\label{eqVWKaehlerC}\\
				\mathrm{d}_A\gamma = 0\\
				[\gamma,\gamma^*] = 0\\
				[\gamma, \beta + \beta^*] = 0 \EWzJhvDisplayMathPeriod.
			\end{gather}
		\end{subequations}
		where\/ $F_A \ifEWzJhvPDF\hspace*{0pt minus 1.6pt}\fi\in\ifEWzJhvPDF\hspace*{0pt minus 0.8pt}\fi \Omega^2(\mathrm{ad}(P))$ is the curvature of\/\hspace{-0.08em} $A$, and\/ $F_A^{0,2} \ifEWzJhvPDF\hspace*{0pt minus 1.6pt}\fi\in\ifEWzJhvPDF\hspace*{0pt minus 0.8pt}\fi \Omega^{0,2}(\mathrm{ad}(P) \otimes_{\mathbb{R}} \mathbb{C})$ is its\/ \ifEWzJhvPDF\nolinebreak\fi $(0,2)$-part.
	\ifEWzJhvPDF\exhyphenpenalty=10000\relax\par\fi
	\end{EWzJhvProposition}
	
	In many interesting cases $\gamma = 0$, for example if $A$ is an irreducible $\mathrm{SU}(2)$-connection.
	
	The equations \EWzJhvEqRef{eqVWKaehlerA}\EWzJhvEqRef{eqVWKaehlerB}\EWzJhvEqRef{eqVWKaehlerC} can be generalised to compact Kähler manifolds of arbitrary dimension, in which case $\beta \in \Omega^{N,0}(\mathrm{ad}(P) \otimes_{\mathbb{R}} \mathbb{C})$ where $N := \dim_{\mathbb{C}}M$. This was the setting of \cite{chen2024VWeqKaehler}.
	
	The space of solutions $(A,\beta)$ to \EWzJhvEqRef{eqVWKaehlerA}\EWzJhvEqRef{eqVWKaehlerB}\EWzJhvEqRef{eqVWKaehlerC} modulo the gauge group action \ifEWzJhvPDF\linebreak\fi admits a $\mathrm{U}(1)$-action\EWzJhvFootnote{If one considers the other side of the Kobayashi--Hitchin correspondence, it is actually a $\mathbb{C}^*$-action. More precisely, it is actually this $\mathbb{C}^*$-action that was considered in \cite{Tanaka_2019}\cite{Tanaka_2017}\cite{chen2024VWeqKaehler}.}: $\beta \mapsto e^{i\theta}\beta$, whose invariant locus is used to define Vafa--Witten invariants for projective surfaces \cite{Tanaka_2019}\cite{Tanaka_2017}.
	
	\EWzJhvMiniSubsection*{Relation between $\mathrm{U}(1)$-invariancy, nilpotency, and parascalar $(N,0)$-forms}
	
	We work with some fixed invariant inner product on the associated Lie algebra $\mathfrak{g}$, which then induces an invariant inner product on the adjoint bundle $\mathrm{ad}(P)$. Recall our notion of ``parascalar $(N,0)$-form'' (\autoref{defnParascalarN}, generalising \autoref{defn2plusParascalar}), which is defined without referring to any prespecified (almost) complex structure.
	
	\ifEWzJhvPageBreak\pagebreak\fi
	
	\begin{EWzJhvProposition}\label{propU1invToParascalar}
		Let\/ $M$ be an almost complex manifold of complex dimension\/~$N$. \ifEWzJhvPDF\linebreak\fi Let\/ $G$ be a Lie group with invariant inner product on associated Lie algebra. Let\/ $P \rightarrow M$ \ifEWzJhvPDF\linebreak\fi be a\/ $G$-principal bundle. If\/\hspace{-0.08em} $\beta \in \Omega^{N,0}(\mathrm{ad}(P) \otimes_{\mathbb{R}} \mathbb{C})$ is\/ $\mathrm{U}(1)$-invariant modulo the gauge group action (that is, for any\/~$e^{i\theta} \in \mathrm{U}(1)$, $e^{i\theta}\beta$ and\/ $\beta$ are related by some gauge transformation of\/\hspace{-0.08em} $P \rightarrow M$), then\/ $a := \beta - \beta^* \in \Omega^N(\mathrm{ad}(P))$ must be a parascalar\/ $(N,0)$-form.
	\end{EWzJhvProposition}
	\begin{proof}
		Let $p \in M$. Let $\nu$ be a smooth nowhere vanishing section of\/ $\Omega^{N,0}$ over some open neighbourhood of $p$. Write $\beta = (\xi_1 + i\xi_2) \otimes_{\mathbb{C}} \nu$ where $\xi_1,\xi_2$ are sections of\/ $\mathrm{ad}(P)$, over this open neighbourhood. Then $a = \xi_1 \otimes_{\mathbb{R}} (\nu + \bar{\nu}) + \xi_2 \otimes_{\mathbb{R}} i(\nu - \bar{\nu})$. To show that $a$ is a parascalar $(N,0)$-form, we just need to show $\lvert\xi_1\rvert=\lvert\xi_2\rvert$ and $\langle \xi_1, \xi_2 \rangle=0$ over the open neighbourhood, which together are equivalent to $\langle \xi_1 + i\xi_2,\: \xi_1 + i\xi_2 \rangle=0$ where $\displaystyle \langle\EWzJhvCDotInd,\EWzJhvCDotInd\rangle \in \Gamma\bigl(\bigl((\mathrm{ad}(P))^*\bigr)\ifEWzJhvPDF{\rule{0pt}{1.88ex}}\fi^{\mskip-2mu\relax \otimes\EWzJhvMathVCenter 2}\otimes_{\mathbb{R}}\mathbb{C}\bigr)$ is \emph{the\/ $\mathbb{C}$-linear extension} of the inner product on $\mathrm{ad}(P)$. As the inner product on $\mathrm{ad}(P)$ is invariant under gauge transformations, its $\mathbb{C}$-linear extension is also invariant under gauge transformations. We assumed for any~$e^{i\theta} \in \mathrm{U}(1)$, $e^{i\theta}\beta$ and $\beta$ are related by some gauge transformation. As a consequence, for any~$e^{i\theta}$, $\bigl\langle e^{i\theta}(\xi_1 + i\xi_2),\: e^{i\theta}(\xi_1 + i\xi_2) \bigr\rangle = \langle \xi_1 + i\xi_2,\: \xi_1 + i\xi_2 \rangle$, so $\langle \xi_1 + i\xi_2,\: \xi_1 + i\xi_2 \rangle=0$.
	\end{proof}
	
	We can also draw a connection between our notion of ``parascalar $(N,0)$-form'' and the notion of ``nilpotent solutions'' of \cite{chen2024VWeqKaehler}. Recall \cite{chen2024VWeqKaehler} showed that for~$G = \mathrm{U}(n)$, if $(A,\beta)$ is a solution to \EWzJhvEqRef{eqVWKaehlerA}\EWzJhvEqRef{eqVWKaehlerB}\EWzJhvEqRef{eqVWKaehlerC} (or their generalisation to compact Kähler manifolds of arbitrary dimension) that is $\mathrm{U}(1)$-invariant modulo the gauge group action, then $\beta$ must be pointwise nilpotent (that is, at every point in $M$, any value of $\beta$ in $\mathrm{ad}(P) \otimes_{\mathbb{R}} \mathbb{C}$ can be identified with some nilpotent matrix in $\mathfrak{gl}(n,\mathbb{C}) = \mathfrak{u}(n) \otimes_{\mathbb{R}} \mathbb{C}$ under some local trivialisation of $P$) {\small \cite[Lemma~2.12]{chen2024VWeqKaehler}}, and solutions $(A,\beta)$ where $\beta$ is pointwise nilpotent\EWzJhvFootnote{or more generally, solutions with uniformly bounded spectral covers} satisfy a $C^0$ a priori estimate and, as a consequence, a version of Uhlenbeck compactness {\small \cite[section~5]{chen2024VWeqKaehler}}.
	
	\begin{EWzJhvProposition}\label{propNilpotentParascalar}
		Let\/ $M$ be an almost complex manifold of complex dimension\/~$N$. \ifEWzJhvPDF\linebreak\fi Let\/ $G$ be a Lie group with invariant inner product on associated Lie algebra. Let\/ $P \rightarrow M$ be a\/ $G$-principal bundle. Let\/ $\beta \in \Omega^{N,0}(\mathrm{ad}(P) \otimes_{\mathbb{R}} \mathbb{C})$. Let\/ $a := \beta - \beta^* \in \Omega^N(\mathrm{ad}(P))$.
		\begin{enumerate}[topsep=\glueexpr(\topsep+\parsep)/2\relax, itemsep=\glueexpr\itemsep/2\relax, font=\fontshape{ui}\selectfont]
			\item \EWzJhvPrefixFootnote{\label{ftnPropNilpotentParascalarInvInnPd}The invariant inner product on $\mathfrak{g}$ remains arbitrary.}For\/~$\mathfrak{g} = \mathfrak{u}(n)$, if\/\hspace{-0.08em} $\beta$ is pointwise nilpotent, then\/ $a$ is a parascalar\/ $(N,0)$-form.
			\item \EWzJhvPrefixFootRef{ftnPropNilpotentParascalarInvInnPd}For\/~$\mathfrak{g} = \mathfrak{su}(2)$ and\/ $N \geqslant 2$, $\beta$ is pointwise nilpotent if and only if\/\hspace{-0.08em} $a$ is a parascalar\/ $(N,0)$-form.
		\end{enumerate}
	\end{EWzJhvProposition}
	\begin{proof}
		For~$p \in M$, let $\nu$ be a smooth nowhere vanishing section of\/ $\Omega^{N,0}$ over some open neighbourhood of $p$, and write $\beta = (\xi_1 + i\xi_2) \otimes_{\mathbb{C}} \nu$ where $\xi_1,\xi_2$ are sections of\/ $\mathrm{ad}(P)$, over this open neighbourhood. As in the proof of \autoref{propU1invToParascalar}, $a$ is a parascalar $(N,0)$-form over the open neighbourhood if $\langle \xi_1 + i\xi_2,\: \xi_1 + i\xi_2 \rangle=0$. The converse holds when $N \geqslant 2$, due to \autoref{propParascalarNAlCplxMfld}. We trivialise $P$ over the open neighbourhood, shrinking the open neighbourhood if necessary.
		
		For~$\mathfrak{g} = \mathfrak{u}(n)$, suppose $\beta$ is pointwise nilpotent over the open neighbourhood, that is, $(\xi_1 + i\xi_2)(p')$ is nilpotent as a matrix in $\mathfrak{gl}(n,\mathbb{C})$ for any point $p'$ in the open neighbourhood. Then all eigenvalues of $(\xi_1 + i\xi_2)(p')$ must be zero, so $(\xi_1 + i\xi_2)(p') \in \mathfrak{sl}(n,\mathbb{C})$ and $\xi_1(p'),\xi_2(p') \in \mathfrak{su}(n)$. As $\mathfrak{su}(n)$ is simple, the restriction of the given invariant \ifEWzJhvPDF\linebreak\fi inner product on $\mathfrak{g} = \mathfrak{u}(n)$ to $\mathfrak{su}(n)$ must be proportional to the invariant inner \ifEWzJhvPDF\linebreak\fi product $(X,Y) \mapsto -\operatorname{tr}(XY)$ for~$X,Y \in \mathfrak{su}(n)$. Because $(\xi_1 + i\xi_2)(p')$ is nilpotent, $-\operatorname{tr}\bigl(((\xi_1 + i\xi_2)(p'))^2\bigr)=0$ and thus $\bigl\langle (\xi_1 + i\xi_2)(p'),\: (\xi_1 + i\xi_2)(p') \bigr\rangle = 0$. This implies $a$ is a parascalar $(N,0)$-form.
		
		\ifEWzJhvPageBreak\pagebreak\fi
		
		To prove the converse for~$\mathfrak{g} = \mathfrak{su}(2)$ and $N \geqslant 2$, suppose $a$ is a parascalar $(N,0)$-form. Then $\bigl\langle (\xi_1 + i\xi_2)(p'),\: (\xi_1 + i\xi_2)(p') \bigr\rangle = 0$ and thus $-\operatorname{tr}\bigl(((\xi_1 + i\xi_2)(p'))^2\bigr)=0$. As in this case $(\xi_1 + i\xi_2)(p') \in \mathfrak{sl}(2,\mathbb{C})$, this implies that both eigenvalues of $(\xi_1 + i\xi_2)(p')$ are zero, so $\beta$ is pointwise nilpotent.
	\end{proof}
	
	A variant\EWzJhvFootnote{\label{ftnC0EstimateVariant}A closer imitation of \cite[Proposition~5.1]{chen2024VWeqKaehler} would be to equip $(M,g)$ with the $\mathbb{Z}_2$-complex structure $\tilde{J}_a$ induced by $a$, consider $\Omega^{2,+} \cap (\Omega^{2,0} \oplus \Omega^{0,2})$ as a holomorphic line bundle, there establish an Hermite--Einstein metric (possible on a compact hermitian manifold due to e.g.\ \cite[(2.1.6)~Corollary]{KobayashiHitchinCorrespondence}), and then compute the Weitzenböck formula for $a$ with this metric and its associated Chern connection, instead of with the Riemannian metric and the Levi-Civita connection on $\Omega^{2,+}$. It would be less straightforward to see why the estimate can be made only depend on $(M,g)$ and $C(\mathfrak{g})$ but not $\tilde{J}_a$, though.} of the $C^0$ a priori estimate of \cite[Proposition~5.1]{chen2024VWeqKaehler} also holds:
	\begin{EWzJhvProposition}\label{propParascalarC0Estimate}
		Let\/ $(M,g)$ be a compact smooth oriented Riemannian\/ $4$-manifold. Let\/ $G$ be a Lie group with\/ $\mathfrak{g} = \mathfrak{su}(2)$. Let\/ $\mathfrak{g}$ be endowed with an invariant inner product. Let\/ $P \rightarrow M$ be a\/ $G$-principal bundle over\/ $M$. Let\/ $(A,a) \in \mathscr{A}(P) \times \Omega^{2,+}(\mathrm{ad}(P))$ be a solution to \EWzJhvEqRef{eqVW}. If\/\hspace{-0.08em} $a$ is a parascalar\/ $(2,0)$-form, then\/ $\lVert a\rVert_{L^\infty}^2 \leqslant C(\mathfrak{g})\lVert R^+\rVert_{L^\infty}$,\EWzJhvFootnote{$C(\mathfrak{g})$ denotes a constant that only depends on the invariant inner product on $\mathfrak{g} = \mathfrak{su}(2)$. A more detailed calculation of the term involving the curvature of $g$ in the Weitzenböck formula would allow the $\lVert R^+\rVert_{L^\infty}$ on the right-hand side of the inequality to be replaced by the norm of the negative part of some kind of curvature tensor of $g$, compa\EWzJhvText{rab}le to \cite[Theorem~2.1.3]{MaresPhD} or \cite[Corollary~3.3]{chen2024VWeqKaehler}.} where\/ $R^+ \in \operatorname{Sym}^2(\Omega^{2,+})$ is the self-dual part of the Riemann curvature tensor.
	\end{EWzJhvProposition}
	As a consequence of such a $C^0$ a priori estimate, \cite[Theorem~3.5.2]{MaresPhD} applies and implies that such ``parascalar'' solutions also satisfy a version of Uhlenbeck compactness. In other words, our notion of ``parascalar'' is quite similar to the notion of ``nilpotent'' of \cite{chen2024VWeqKaehler}, but does not require a Kähler structure.
	\begin{proof}
		The proof is similar\EWzJhvFootnote{See the footnote\EWzJhvFootRef{ftnC0EstimateVariant} above.} to that of \cite[Proposition~5.1]{chen2024VWeqKaehler}.
		
		\EWzJhvTextNormal{(B.19)} of \cite{MaresPhD} tells us\EWzJhvFootnote{This is a consequence of the Weitzenböck formula {\scriptsize (\autoref{propWeitzen2plus})}, when the operator $\EWzJhvCDotInd\circledast\EWzJhvCDotInd$ there is written out more explicitly. The sign of $\Delta$ is the one that makes $\Delta = \nabla^*\nabla$. About the constants, $C$ or $C({\cdots})$ denote possibly different constants from place to place.}
		\[\Delta\lvert a\rvert^2 + 2\hspace{0.0625em}\bigl|\nabla^A a\bigr|^2 \leqslant 4\hspace{0.0625em}\bigl\langle a, \mathrm{d}_A\mathrm{d}_A^* a\bigr\rangle + C\hspace{0.0625em}\lvert R^+\rvert\lvert a\rvert^2 + 2\hspace{0.0625em}\bigl\langle a, [F_A^+ \centerdot a] \bigr\rangle \EWzJhvDisplayMathPeriod.\]
		\EWzJhvTextNormal{(A.25)} of \cite{MaresPhD} tells us $\bigl\langle a, [F_A^+ \centerdot a] \bigr\rangle = \bigl\langle F_A^+, [a \centerdot a] \bigr\rangle$. We then use \EWzJhvEqRef{eqVW} to get
		\[\Delta\lvert a\rvert^2 + 2\hspace{0.0625em}\bigl|\nabla^A a\bigr|^2 \leqslant C\hspace{0.0625em}\lvert R^+\rvert\lvert a\rvert^2 - \smash{\EWzJhvMathFracVCentered{1}{4}}\hspace{0.0625em}\bigl|[a \centerdot a]\bigr|^2 \EWzJhvDisplayMathPeriod.\]
		Let $p \in M$ be a maximum point of $\lvert a\rvert$. At $p$, we get $\bigl|[a \centerdot a]\bigr|\ifEWzJhvPDF{\rule{0pt}{1.88ex}}\fi^2 \leqslant C\hspace{0.0625em}\lvert R^+\rvert\lvert a\rvert^2$. Because $a$ is a parascalar $(2,0)$-form, using \autoref{defn2plusParascalar}, we easily check $\bigl|[a \centerdot a]\bigr| = C(\mathfrak{g})\lvert a\rvert^2$ for some positive $C(\mathfrak{g})$. As a consequence, $\lVert a\rVert_{L^\infty}^4 \leqslant C(\mathfrak{g})\lVert R^+\rVert_{L^\infty}\lVert a\rVert_{L^\infty}^2$, which simplifies to $\lVert a\rVert_{L^\infty}^2 \leqslant C(\mathfrak{g})\lVert R^+\rVert_{L^\infty}$.
	\end{proof}
	
	\begin{EWzJhvRemark}
		For~$\mathfrak{g} = \mathfrak{su}(2)$, $a \in \Omega^{2,+}(\mathrm{ad}(P))$ that is a (nonzero) parascalar $(2,0)$-form corresponds to the family ``$(B_1\,\,B_1\,\,0)$'' from \EWzJhvText{Tab}le~4.2 ``Stratification of\/ $3 \times 3$ matrices'' on page~39 of subsection~4.1.1 ``Matrix representations of $\Lambda^{2,+} \otimes \mathfrak{sp}(1)$'' of \cite{MaresPhD}.
	\end{EWzJhvRemark}
	
	\EWzJhvMiniSubsection*{Discussions}
	
	We have just shown a close relation between $\mathrm{U}(1)$-invariant (or nilpotent) solutions to \EWzJhvEqRef{eqVWKaehlerA}\EWzJhvEqRef{eqVWKaehlerB}\EWzJhvEqRef{eqVWKaehlerC} and the notion of parascalar $(2,0)$-forms. Note that the condition of being a parascalar $(2,0)$-form makes sense on general smooth $4$-manifolds with no Kähler structure specified. This hints at a possibility to generalise the notion of $\mathrm{U}(1)$-invariant solutions from the Kähler setting to general smooth oriented Riemannian $4$-manifolds. However, \autoref{thm2} indicates that for a nontrivial parascalar $(2,0)$-form forming part of a solution to \EWzJhvEqRef{eqVWa} to exist, it is necessary that the smooth oriented Riemannian $4$-manifold admits a compatible $\mathbb{Z}_2$-complex structure. Therefore, it rather seems that the notion of $\mathrm{U}(1)$-invariant solutions of the Kähler setting cannot be generalised to very general smooth oriented Riemannian $4$-manifolds in this seemingly quite natural manner.
	
	One may ask some further questions around the nonexistence of solutions with a nontrivial parascalar $(2,0)$-form on general smooth oriented Riemannian $4$-manifolds. If one starts with a (compact) Kähler {\small (or only $\mathbb{Z}_2$-hermitian)} surface and perturbs its Riemannian metric, it may no longer admit any compatible $\mathbb{Z}_2$-complex structure. \ifEWzJhvPDF\linebreak\fi Do those solutions with a nontrivial parascalar $(2,0)$-form just disappear, or do they deform into something else? Over a smooth oriented Riemannian $4$-manifold with no compatible $\mathbb{Z}_2$-complex structure, what happens if one tries to minimise some functional that measures how far a solution deviates from parascalar $(2,0)$-forms?
	
	Alternatively, if the existence of a compatible $\mathbb{Z}_2$-complex structure is somehow necessary, one may consider studying the Vafa--Witten equation on (compact) ($\mathbb{Z}_2$-)\ifEWzJhvPDF\linebreak\fi hermitian surfaces. However, it should be noted that in this case, while \EWzJhvEqRef{eqVWKaehlerA}\EWzJhvEqRef{eqVWKaehlerB}\EWzJhvEqRef{eqVWKaehlerC} are still equivalent to \EWzJhvEqRef{eqVW} for~$a = \beta - \beta^*$, a full, decoupling equivalence that includes $\gamma$ like \autoref{propVWequivalenceKaehler} might not hold.
	
	One may also consider higher-dimensional generalisations or variants of the Vafa--Witten equation, like in \cite{chen2024VWeqKaehler}. Without $\mathrm{Hol}(g) \subset \mathrm{Sp}(\frac{N}{2})\hspace{0.0625em}\mathrm{Sp}(1)$, our rather rudimentary results from \autoref{secGeneral} do not preclude potential singularities of the $\mathbb{Z}_2$-(almost) complex structure induced by $a$, if $a$ is assumed to be a parascalar $(N,0)$-form.
	
	\appendix\stepcounter{section}
	\hypersetup{next-anchor=appendix}\section*{Appendix}\addtocontents{toc}{\EWzJhvTocBeginGroupSectionNoPageNumber}\addcontentsline{toc}{section}{Appendix}\addtocontents{toc}{\EWzJhvTocEndGroup}
	
	\subsection{Additional results to \autoref*{subsecCplxStr}}\label{subsecCplxStrAddResults}
	
	Following \autoref{subsecCplxStr}, we prove some additional results about $\Omega^{2n,0,+}$ and the several $\bar{\partial}$-operators. The assumptions and notation from \autoref{subsecCplxStr} are kept.
	
	\subsubsection*{Normal bundle in twistor space, the $\bar{\partial}$-operator $\bar{\partial}^{\mathtt{tw}}$}
	
	Recall that given $(M,g,Q)$ with $\mathrm{Hol}(g) \subset \mathrm{Sp}(N)\hspace{0.0625em}\mathrm{Sp}(1)$, we can construct its twistor space $Z$ (\cite[section~4]{SelfDuality4dimRiemGeometry} for~$N=1$, \cite{QKManifoldsSalamon} for~$N \geqslant 2$). It is an $S^2$-fibre bundle over $M$ whose (local) sections correspond to compatible almost complex structures on $M$. \ifEWzJhvPDF\linebreak\fi $Z$ is endowed with a natural almost complex structure satisfying the property that, the graph of a (local) section of $Z$ is an almost complex submanifold (isomorphic to $M$ with the corresponding almost complex structure) if and only if the section corresponds to an integ\EWzJhvText{rab}le almost complex structure on $M$. If (and only if) $N \geqslant 2$, or $N = 1$ and the Weyl curvature tensor of $(M,g,Q)$ is anti-self-dual, then this natural almost complex structure on $Z$ is integ\EWzJhvText{rab}le, making $Z$ a complex manifold.
	
	Here\EWzJhvInfixFootnote{$\Omega^{\boldsymbol{\cdots}}$ with various superscripts without an explicitly specified base manifold continue to denote the various bundles over $M$, not $Z$.}, we have chosen the compatible complex structure $J$ on $M$. This gives a section of $Z$ whose graph is an almost complex submanifold. This almost complex submanifold of $Z$ has a normal bundle, whose (almost-holomorphic) sections should correspond to first-order variations of the (integ\EWzJhvText{rab}le) compatible almost complex structure $J$ on $M$. Let $J_t\ (t \in \mathbb{R})$ be a smooth variation of compatible almost complex structure on $M$ with $J_0 = J$. Let $\omega^{J_t} := \langle J_t\EWzJhvCDotInd,\EWzJhvCDotInd \rangle \in \Omega^{2,+}$. Let $\frac{\mathrm{d}}{\mathrm{d}t}\omega^{J_t}\bigr|_{t=0} = u\,\omega^+ + \overline{u\,\omega^+}$ where $u$ is some complex-valued function. Let $\overline{\pi}_t := \frac{\mathrm{id}+iJ_t}{2}$, then $\overline{\pi}_t\bar{e}_{j,k}$ are $(0,1)$-vector fields under $J_t$ with\ifEWzJhvPDF\vspace{-\baselineskip}\fi
	\begin{align*}
		\frac{\EWzJhvMathVCenter{\mathrm{d}\rule{0pt}{2ex}}}{\mathrm{d}t}\bigl(\overline{\pi}_t\bar{e}_{j,1}\bigr)\biggr|_{t=0} = \EWzJhvMathFracVCentered{i}{2}\,\bigl(\iota_{\bar{e}_{j,1}}(u\,\omega^+ + \overline{u\,\omega^+})\bigr)^{\mskip-2mu\relax \sharp} &= \EWzJhvMathFracVCentered{i}{2}\,\bar{u}\,e_{j,2}\\
		\frac{\EWzJhvMathVCenter{\mathrm{d}\rule{0pt}{2ex}}}{\mathrm{d}t}\bigl(\overline{\pi}_t\bar{e}_{j,2}\bigr)\biggr|_{t=0} = \EWzJhvMathFracVCentered{i}{2}\,\bigl(\iota_{\bar{e}_{j,2}}(u\,\omega^+ + \overline{u\,\omega^+})\bigr)^{\mskip-2mu\relax \sharp} &= -\EWzJhvMathFracVCentered{i}{2}\,\bar{u}\,e_{j,1} \EWzJhvDisplayMathPeriod.
	\end{align*}
	The integ\EWzJhvText{rab}ility of $J_t$ would imply $\nabla_{\!\overline{\pi}_t\bar{e}_{j,k}} \omega^{J_t}$ is of type~$(2,0)$ under $J_t$ and therefore $\iota_{\overline{\pi}_t\bar{e}_{1,2}}\iota_{\overline{\pi}_t\bar{e}_{1,1}} \nabla_{\!\overline{\pi}_t\bar{e}_{j,1}} \omega^{J_t}=0$. We then compute
	\begin{subequations}\label{eqDBarTW}
		\begin{align}
			&\frac{\EWzJhvMathVCenter{\mathrm{d}\rule{0pt}{2ex}}}{\mathrm{d}t}\Bigl(\iota_{\overline{\pi}_t\bar{e}_{1,2}}\iota_{\overline{\pi}_t\bar{e}_{1,1}} \nabla_{\!\overline{\pi}_t\bar{e}_{j,1}} \omega^{J_t}\Bigr)\biggr|_{t=0}\notag\\
			={}&\iota_{\bar{e}_{1,2}}\iota_{\bar{e}_{1,1}}\biggl(\EWzJhvMathFracVCentered{i}{2}\,\bar{u}\,\nabla_{\!e_{j,2}}\omega^J + \nabla_{\!\bar{e}_{j,1}}\bigl(u\,\omega^+ + \overline{u\,\omega^+}\bigr)\biggr)\notag\\
			={}&\partial_{\bar{e}_{j,1}}\bar{u}-\bar{u}\,\iota_{\bar{e}_{j,1}}\psi + \EWzJhvMathFracVCentered{i}{2}\,\overline{u\,\iota_{\bar{e}_{j,2}} \phi}
		\end{align}
		and\ifEWzJhvPDF\vspace{\glueexpr-\baselineskip/2\relax}\fi
		\begin{align}
			&\frac{\EWzJhvMathVCenter{\mathrm{d}\rule{0pt}{2ex}}}{\mathrm{d}t}\Bigl(\iota_{\overline{\pi}_t\bar{e}_{1,2}}\iota_{\overline{\pi}_t\bar{e}_{1,1}} \nabla_{\!\overline{\pi}_t\bar{e}_{j,2}} \omega^{J_t}\Bigr)\biggr|_{t=0}\notag\\
			={}&\iota_{\bar{e}_{1,2}}\iota_{\bar{e}_{1,1}}\biggl(-\EWzJhvMathFracVCentered{i}{2}\,\bar{u}\,\nabla_{\!e_{j,1}}\omega^J + \nabla_{\!\bar{e}_{j,2}}\bigl(u\,\omega^+ + \overline{u\,\omega^+}\bigr)\biggr)\notag\\
			={}&\partial_{\bar{e}_{j,2}}\bar{u}-\bar{u}\,\iota_{\bar{e}_{j,2}}\psi - \EWzJhvMathFracVCentered{i}{2}\,\overline{u\,\iota_{\bar{e}_{j,1}} \phi} \EWzJhvDisplayMathPeriod.
		\end{align}
	\end{subequations}
	We see that the expressions~\EWzJhvEqRef{eqDBarTW} define a $\bar{\partial}$-operator, denoted by $\bar{\partial}^{\mathtt{tw}}$, on the dual line bundle $\displaystyle ({\textstyle \Omega^{2,0,+}})^*$ (identified with $\overline{\Omega^{2,0,+}}$) such that $\bigl({\displaystyle ({\textstyle \Omega^{2,0,+}})^*},\bar{\partial}^{\mathtt{tw}}\bigr)$ is naturally isomorphic to the normal bundle of the graph in $Z$ corresponding to $J$. If $N \geqslant 2$, the natural almost complex structure on $Z$ is integ\EWzJhvText{rab}le, so in this case $\bigl({\displaystyle ({\textstyle \Omega^{2,0,+}})^*},\bar{\partial}^{\mathtt{tw}}\bigr)$ is a holomorphic line bundle.
	
	Compare \EWzJhvEqRef{eqDBarD2n0plus} and \EWzJhvEqRef{eqDBarTW}, we get
	\begin{EWzJhvProposition}\label{prop20plusStarTW}
		There is an isomorphism of complex line bundles with\/ $\bar{\partial}$-operator\/ $(\Omega^{2,0,+}, \bar{\partial}^{\mathtt{D}}) \cong {\displaystyle ({\textstyle \Omega^{2,0,+}, \bar{\partial}^{\mathtt{LC}}})^{\otimes_{\mathbb{C}}\mskip1.5mu\relax\EWzJhvMathVCenter 2}} \otimes_{\mathbb{C}} \bigl({\displaystyle ({\textstyle \Omega^{2,0,+}})^*},\bar{\partial}^{\mathtt{tw}}\bigr)$, where the underlying isomorphism of complex line bundles is the natural one:\/ $\Omega^{2,0,+} \cong {\displaystyle ({\textstyle \Omega^{2,0,+}})^{\otimes_{\mathbb{C}}\mskip1.5mu\relax\EWzJhvMathVCenter 2}} \otimes_{\mathbb{C}} {\displaystyle ({\textstyle \Omega^{2,0,+}})^*}$
	\end{EWzJhvProposition}
	\begin{EWzJhvRemark}
		The proposition above holds whether $N \geqslant 2$ or not. However, when $N=1$, we do not claim $(\Omega^{2,0,+}, \bar{\partial}^{\mathtt{LC}})$ or $\bigl({\displaystyle ({\textstyle \Omega^{2,0,+}})^*},\bar{\partial}^{\mathtt{tw}}\bigr)$ is always a holomorphic line bundle.
	\end{EWzJhvRemark}
	\begin{EWzJhvRemark}
		When $(M,g,J)$ is Kähler, $\bar{\partial}^{\mathtt{D}}$, $\bar{\partial}^{\mathtt{LC}}$ and $\bar{\partial}^{\mathtt{tw}}$ all coincide with $\bar{\partial}^{\mathtt{CM}}$, in which case the results in \autoref{subsecCplxStr} and the current subsection become mostly trivial.
	\end{EWzJhvRemark}
	\begin{EWzJhvRemark}
		The difference between the $\bar{\partial}$-operator of $(\Omega^{2,0,+}, \bar{\partial}^{\mathtt{LC}}) \otimes_{\mathbb{C}} \bigl({\displaystyle ({\textstyle \Omega^{2,0,+}})^*},\bar{\partial}^{\mathtt{tw}}\bigr)$ (or equivalently, $(\Omega^{2,0,+}, \bar{\partial}^{\mathtt{D}}) \otimes_{\mathbb{C}} {\displaystyle ({\textstyle \Omega^{2,0,+}, \bar{\partial}^{\mathtt{LC}}})^*}$) and the standard $\bar{\partial}$-operator on the trivial line bundle is related to ${\displaystyle \mathrm{d}^*}\omega^J$ or the Lee form $\theta := - {\displaystyle \mathrm{d}^*}\omega^J \circ J$. When $N \geqslant 2$, $\mathrm{d}\theta$ is known to be $Q$-hermitian (and harmonic) \cite{CompatibleAlmostCplxStrsOnQKManifolds}, which is connected to the fact that in this case $(\Omega^{2,0,+}, \bar{\partial}^{\mathtt{LC}}) \otimes_{\mathbb{C}} \bigl({\displaystyle ({\textstyle \Omega^{2,0,+}})^*},\bar{\partial}^{\mathtt{tw}}\bigr)$ is a holomorphic line bundle.
	\end{EWzJhvRemark}
	\begin{EWzJhvRemark}
		Subsection~3.3 of \cite{CompatibleAlmostCplxStrsOnQKManifolds} considered the linear connection (\hspace*{-0.08em}``the first canonical connection of the almost Hermitian structure''\hspace*{-0.08em}) $\nabla'_{\!X} := \nabla_{\!X} - \frac{1}{2}\,J\,\nabla_{\!X} J$, in particular on the orthogonal complement of $J$ in $Q$, which in our notation is basically equivalent to considering the Chern connection on $(\Omega^{2,0,+}, \bar{\partial}^{\mathtt{LC}})$. Notably, they considered its curvature and first Chern class, and then developed some consequences.
	\end{EWzJhvRemark}
	
	\ifEWzJhvPageBreak\pagebreak\fi
	
	\subsubsection*{The inclusion\hspace{0.5em plus 0.166667em minus 0.066667em}$\Omega^{2n,0,+} \subset \Omega^{2n,0}$}
	
	It is obvious that there is an inclusion of complex vector bundles $\Omega^{2n,0,+} \subset \Omega^{2n,0}$. Regarding the relation between $\Omega^{2n,0,+}$ and $(\Omega^{2n,0},\bar{\partial}^{\mathtt{CM}})$, we have
	\begin{EWzJhvProposition}\label{prop2n0plusDBar}
		For any integer\/~$n$ satisfying\/ $1 \leqslant n < N$, any\/~$b \in \Omega^{2,0,+}$ and any\/~$p \in M$ satisfying that\/ $b$ does not vanish at\/ $p$, the following are equivalent:\vspace{-\parskip}
		\begin{enumerate}[topsep=1ex plus 0.25ex minus 0.25ex, itemsep=1ex plus 0.25ex minus 0.25ex, parsep=0pt, font=\fontshape{ui}\selectfont]
			\item the\/ $(2n,1)$-part of\/\hspace{0em plus -0.08em} $\mathrm{d}(b^n)$ vanishes at\/ $p$;
			\item both\/ $\nabla J$ and\/ $\displaystyle \mathrm{d}^*b$ vanish at\/ $p$;
			\item both\/ $\nabla J$ and\/ $\bar{\partial}^{\mathtt{LC}}b$ vanish at\/ $p$.
		\end{enumerate}
	\end{EWzJhvProposition}
	
	As a consequence, for any integer~$n$ satisfying $1 \leqslant n < N$, $\Omega^{2n,0,+}$ is a holomorphic subbundle of $(\Omega^{2n,0},\bar{\partial}^{\mathtt{CM}})$ if and only if $(M,g,J)$ is Kähler. (For~$n=N$, $\Omega^{2N,0,+}$ is just $\Omega^{2N,0}$.)
	
	\medskip
	
	Before proving \autoref{prop2n0plusDBar}, we first prove a lemma that resembles a property of the Lefschetz operator:
	\begin{EWzJhvLemma}\label{lmaSTOmegaPlusIsom}
		For any integers\/ $s$ and\/ $t$ satisfying\/ $0 \leqslant s \leqslant N$, the wedge product gives an isomorphism of complex vector bundles\/ $\Omega^{s,t} \otimes_{\mathbb{C}} \Omega^{2(N-s),0,+} \cong \Omega^{2N-s,t}$.
	\end{EWzJhvLemma}
	\begin{proof}
		We induce on $s$.
		\begin{enumerate}[align=left, leftmargin=0pt, labelindent=0pt, itemindent=\parindent, listparindent=\parindent, labelsep=*, topsep=\parsep, itemsep=0pt]
			\item For~$s=0$, as $\Omega^{2N,0,+}=\Omega^{2N,0}$, we have $\Omega^{0,t} \otimes_{\mathbb{C}} \Omega^{2N,0,+} \cong \Omega^{2N,t}$.
			
			\item Assume $0 \leqslant s < N$ and $\Omega^{s,t} \otimes_{\mathbb{C}} \Omega^{2(N-s),0,+} \cong \Omega^{2N-s,t}$. To prove the bundle map given by the wedge product from $\Omega^{s+1,t} \otimes_{\mathbb{C}} \Omega^{2(N-s-1),0,+}$ to $\Omega^{2N-s-1,t}$ is an isomorphism, by counting dimensions, we only need to prove it is a fibrewise surjection. For any~$p \in M$, the fibre $\Omega^{2N-s-1,t}\rvert_p$ is spanned by $\iota_X \alpha$ where $X \in T_p^{1,0}\mskip-2mu\relax M$ and $\alpha \in \Omega^{2N-s,t}\rvert_p$, so we only need to prove that any such $\iota_X \alpha$ is in the image of the wedge product. We already know the bundle map given by the wedge product from $\Omega^{s,t} \otimes_{\mathbb{C}} \Omega^{2(N-s),0,+}$ to $\Omega^{2N-s,t}$ is a fibrewise surjection, so we can write $\alpha = \beta \wedge (\omega^+)^{N-s}$ where $\beta \in \Omega^{s,t}\rvert_p$. Now,
			\begin{align*}
				\iota_X \alpha &= (\iota_X \beta) \wedge (\omega^+)^{N-s} + (-1)^{s+t}(N-s)\,\beta \wedge (\iota_X \omega^+) \wedge (\omega^+)^{N-s-1}\\
				&= \bigl((\iota_X \beta) \wedge \omega^+ + (-1)^{s+t}(N-s)\,\beta \wedge (\iota_X \omega^+)\bigr) \wedge (\omega^+)^{N-s-1}
			\end{align*}
			which shows $\iota_X \alpha$ is in the image of the wedge product from $\bigl(\Omega^{s+1,t} \otimes_{\mathbb{C}} \Omega^{2(N-s-1),0,+}\bigr)\bigr|_p$ to $\Omega^{2N-s-1,t}\rvert_p$.\qedhere
		\end{enumerate}
	\end{proof}
	\begin{proof}[Proof of \autoref{prop2n0plusDBar}]
		\leavevmode\vspace{-0.5pt}
		\begin{enumerate}[align=left, leftmargin=0pt, labelindent=0pt, itemindent=\parindent, listparindent=\parindent, labelsep=*, topsep=\glueexpr(\topsep+\parsep)/2\relax, itemsep=\glueexpr\itemsep/2\relax]
			\linespread{1.08}\selectfont
			
			\ifEWzJhvPDF
			\newcommand*{\EWzJhvPDFZeroDepth}[1]{\raisebox{0pt}[\height][0pt]{#1}}
			\newcommand*{\EWzJhvPDFSmallHeightMath}[1]{\raisebox{0pt}[1ex]{$#1$}}
			\else
			\newcommand*{\EWzJhvPDFZeroDepth}[1]{#1}
			\newcommand*{\EWzJhvPDFSmallHeightMath}[1]{#1}
			\fi
			
			\item We first show \EWzJhvPDFZeroDepth{$\displaystyle \bigl(\mathrm{d}(b^n)\bigr)\ifEWzJhvPDF{\rule{0pt}{1.88ex}}\fi^{\mskip-2mu\relax 2n,1}\bigr|_p=0$} implies both $(\mathrm{d}b)^{2,1}\rvert_p = 0$ and $\nabla J\rvert_p=0$.
			
			We have $\displaystyle \mathrm{d}(b^n) = n\,(\mathrm{d}b) \wedge (b^{n-1})$ and therefore $\displaystyle \EWzJhvPDFSmallHeightMath{\displaystyle \bigl(\mathrm{d}(b^n)\bigr)\ifEWzJhvPDF{\rule{0pt}{1.88ex}}\fi^{\mskip-2mu\relax 2n,1}} = n\,(\mathrm{d}b)^{2,1} \wedge (b^{n-1})$. If $\displaystyle \bigl(\mathrm{d}(b^n)\bigr)\ifEWzJhvPDF{\rule{0pt}{1.88ex}}\fi^{\mskip-2mu\relax 2n,1}\bigr|_p=0$, then $\bigl((\mathrm{d}b)^{2,1} \wedge (b^{N-2})\bigr)\bigr|_p = 0$, so with the help of \autoref{lmaSTOmegaPlusIsom} we get $(\mathrm{d}b)^{2,1}\rvert_p = 0$. Then we have $\mathrm{d}b\rvert_p = (\mathrm{d}b)^{3,0}\rvert_p = \bigl(\sum_{j,k} e^{j,k} \wedge (\nabla_{\!e_{j,k}}b)^{2,0}\bigr)\bigr|_p$. \ifEWzJhvPDF\linebreak\fi Note $(\nabla_{\!e_{j,k}}b)^{2,0}\rvert_p \in \Omega^{2,0,+}\rvert_p$, so $(\nabla_{\!e_{j,k}}b)^{2,0}\rvert_p$ is a multiple of\/ $b\rvert_p$, and as a consequence, we can write $\mathrm{d}b\rvert_p = \alpha \wedge b\rvert_p$ where $\alpha \in \Omega^{1,0}\rvert_p$. Choose a smooth function $u \EWzJhvMathSepColon M \rightarrow \mathbb{C}$ satisfying $u(p)=1$ and $\mathrm{d}u\rvert_p=-\alpha$. Then $\mathrm{d}(ub)\rvert_p=0$. \autoref{propDDStar2plus} then implies $\nabla(ub)\rvert_p=0$, so $\bigl\langle \nabla_{\!\bar{e}_{j,k}}\omega^J, \overline{ub} \bigr\rangle\bigr|_p = -\bigl\langle \omega^J, \nabla_{\!\bar{e}_{j,k}}(\overline{ub}) \bigr\rangle\bigr|_p=0$. As a result, $\nabla J\rvert_p=0$.
			
			\item Assuming $\nabla J\rvert_p=0$, we show $(\mathrm{d}b)^{2,1}\rvert_p=0 \Longleftrightarrow \bar{\partial}^{\mathtt{LC}}b\rvert_p=0 \Longleftrightarrow {\displaystyle \mathrm{d}^*b}\rvert_p=0$.
			
			Because $\nabla J\rvert_p=0$, we have $(\nabla_{\!X} e_{j,k})^{0,1}\rvert_p = \bigl(\frac{\mathrm{id}+iJ}{2}\nabla_{\!X} e_{j,k}\bigr) \bigr|_p = \nabla_{\!X} \bigl(\frac{\mathrm{id}+iJ}{2}e_{j,k}\bigr) \bigr|_p = 0$ and similarly $(\nabla_{\!X} \bar{e}_{j,k})^{1,0}\rvert_p = 0$ for any (complex) vector field $X$. As a consequence, we have $(\mathrm{d}b)^{2,1}\rvert_p = \bigl(\sum_{j,k} \bar{e}^{j,k} \wedge \nabla_{\!\bar{e}_{j,k}}b\bigr)\bigr|_p$. Then $(\mathrm{d}b)^{2,1}\rvert_p=0 \Longleftrightarrow \bar{\partial}^{\mathtt{LC}}b\rvert_p=0$. To see \ifEWzJhvPDF\linebreak\\*[-\baselineskip]\fi $\bar{\partial}^{\mathtt{LC}}b\rvert_p=0 \Longleftrightarrow {\displaystyle \mathrm{d}^*b}\rvert_p=0$, check \EWzJhvEqRef{eqDBarD2n0plus}, where $\phi\rvert_p$ is now zero.\qedhere\par
		\end{enumerate}
	\end{proof}
	
	\subsubsection*{Relation between $\bar{\partial}^{\mathtt{LC}}$ and some holomorphic vector bundles over $Z$}
	
	Assume the natural almost complex structure on $Z$ is integ\EWzJhvText{rab}le ($N \geqslant 2$, or $N = 1$ and anti-self-dual). Identify $M$ with the graph in $Z$ corresponding to $J$. The adjunction formula tells us
	\[\bigl(\Omega^{2N+1,0}(Z),\bar{\partial}^{\mathtt{CM}}\bigr)\bigr|_M \cong (\Omega^{2N,0}, \bar{\partial}^{\mathtt{CM}}) \otimes_{\mathbb{C}} \bigl((\Omega^{2,0,+})^*,\bar{\partial}^{\mathtt{tw}}\bigr)^{\mskip-2mu\relax *} \EWzJhvDisplayMathPeriod.\]
	Combining this with \autoref{prop2N0plusDCanonical},~\ref{propDBarD2n0plus},~\ref{prop20plusStarTW}, we simply have
	\begin{equation}\label{eqCanonicalBdlZRestrict}
		\bigl(\Omega^{2N+1,0}(Z),\bar{\partial}^{\mathtt{CM}}\bigr)\bigr|_M \cong (\Omega^{2,0,+}, \bar{\partial}^{\mathtt{LC}})^{\otimes_{\mathbb{C}}(N+1)}
	\end{equation}
	that is, $(\Omega^{2,0,+}, \bar{\partial}^{\mathtt{LC}})^{\otimes_{\mathbb{C}}(N+1)}$ is naturally isomorphic to the restriction of the canonical bundle of $Z$.
	
	We can also construct a complex line bundle with $\bar{\partial}$-operator over $Z$ whose restriction is naturally isomorphic to $(\Omega^{2,0,+}, \bar{\partial}^{\mathtt{LC}})$, by pulling back the bundle $\Omega^{2,+} \otimes_{\mathbb{R}} \mathbb{C} \rightarrow M$ with hermitian metric and Levi-Civita connection given by $g$ along the projection $Z \rightarrow M$, then defining a subbundle that pointwise corresponds to $\Omega^{2,0,+}$ with respect to the complex structure at one point given by the base point in $Z$, and finally taking the $(0,1)$-part of the connection restricted to the subbundle. If the natural almost complex structure on $Z$ is integ\EWzJhvText{rab}le, the bundle constructed is also a holomorphic line bundle, whose $(N+1)$-th power is naturally isomorphic to the canonical bundle of $Z$.
	
	In fact, we can say a little more.
	
	The metric $g$ induces a Levi-Civita-like connection on the fibre bundle $Z \rightarrow M$. Holomorphic structure aside, this splits the tangent bundle $TZ$ into the direct sum of the vertical subbundle $T_{\mathtt{v}}Z$ and the horizontal subbundle $T_{\mathtt{h}}Z$. This splitting is compatible with the complex structure on $Z$, so we have a similar splitting into vertical and horizontal subbundles $\Omega^{1,0}_{\vphantom{\mathtt{h}}}(Z) = \Omega_{\mathtt{v}}^{1,0}(Z) \oplus \Omega_{\mathtt{h}}^{1,0}(Z)$, where $\Omega_{\mathtt{v}}^{1,0}(Z)$ vanishes on $T_{\mathtt{h}}Z$ and $\Omega_{\mathtt{h}}^{1,0}(Z)$ vanishes on $T_{\mathtt{v}}Z$. If we take into account the holomorphic structure, assuming $N \geqslant 2$ (or $N = 1$ and anti-self-dual and Einstein\EWzJhvFootnote{
		An oriented anti-self-dual Einstein $4$-dimensional Riemannian manifold is what one may call a (real-)$4$-dimensional quaternion-Kähler manifold, as the condition of being anti-self-dual and Einstein ensures that its Riemann curvature tensor behaves like that of a quaternion-Kähler manifold of higher dimension. \cite[Theorem~4.2]{QKManifoldsSalamon} is in the context of $N \geqslant 2$, but its proof also works for a $4$-dimensional quaternion-Kähler manifold (also see \cite[13.81~Theorem~(1)]{EinsteinManifolds}).
	}), we have a short exact sequence \cite[Theorem~4.2]{QKManifoldsSalamon}
	\[0 \longrightarrow \bigl(\Omega_{\mathtt{v}}^{1,0}(Z),\bar{\partial}^{\mathtt{CM}}\bigr) \longrightarrow \bigl(\Omega^{1,0}(Z),\bar{\partial}^{\mathtt{CM}}\bigr) \longrightarrow \bigl(\Omega^{1,0}(Z) / \Omega_{\mathtt{v}}^{1,0}(Z),\bar{\partial}^{\mathtt{CM}}\bigr) \longrightarrow 0\]
	where $\bigl(\Omega_{\mathtt{v}}^{1,0}(Z),\bar{\partial}^{\mathtt{CM}}\bigr)$ is a holomorphic subbundle of $\bigl(\Omega^{1,0}(Z),\bar{\partial}^{\mathtt{CM}}\bigr)$, and the fourth term $\bigl(\Omega^{1,0}(Z) / \Omega_{\mathtt{v}}^{1,0}(Z),\bar{\partial}^{\mathtt{CM}}\bigr)$ is the quotient holomorphic vector bundle. Taking determinant, we get\ifEWzJhvPDF\makebox[0.5\linewidth]{}\vspace{\glueexpr-\baselineskip/2\relax}\fi
	\[\bigl(\Omega^{2N+1,0}(Z),\bar{\partial}^{\mathtt{CM}}\bigr) \cong \bigl(\Omega_{\mathtt{v}}^{1,0}(Z),\bar{\partial}^{\mathtt{CM}}\bigr) \otimes_{\mathbb{C}} \det{}\bigl(\Omega^{1,0}(Z) / \Omega_{\mathtt{v}}^{1,0}(Z),\bar{\partial}^{\mathtt{CM}}\bigr) \EWzJhvDisplayMathPeriod.\]
	By inspecting the expression for ``$d\mskip2mu\relax\eta_i$'' on page~153 of \cite{QKManifoldsSalamon}, we see the restriction of $\bigl(\Omega^{1,0}(Z) / \Omega_{\mathtt{v}}^{1,0}(Z),\bar{\partial}^{\mathtt{CM}}\bigr)$ is naturally isomorphic to $(\Omega^{1,0}, \bar{\partial}^{\mathtt{LC}})$. The determinant line bundle $\det{}(\Omega^{1,0}, \bar{\partial}^{\mathtt{LC}})$ is naturally isomorphic to $\displaystyle ({\textstyle \Omega^{2,0,+}, \bar{\partial}^{\mathtt{LC}}})^{\otimes_{\mathbb{C}}\mskip0.5mu\relax\EWzJhvMathVCenter N}$. By comparing this to \EWzJhvEqRef{eqCanonicalBdlZRestrict}, we see the restriction of $\bigl(\Omega_{\mathtt{v}}^{1,0}(Z),\bar{\partial}^{\mathtt{CM}}\bigr)$ is naturally isomorphic to $(\Omega^{2,0,+}, \bar{\partial}^{\mathtt{LC}})$.
	
	\hypersetup{next-anchor=acknowledgements}\section*{Acknowledgements}\addcontentsline{toc}{section}{Acknowledgements}
	
	I would like to thank my doctoral supervisor Prof.~Dr.~Thomas Walpuski for introducing me to the topic of the Vafa--Witten equation, guiding me through the research, and giving me advice on the writing of this article. I would also like to thank Siqi He (AMSS, Chinese Academy of Sciences) for some comments on the results.
	
	\ifEWzJhvPageBreak\pagebreak\fi
	
	Funded by the Deutsche Forschungsgemeinschaft (DFG, German Research Foundation) under Germany's Excellence Strategy -- The Berlin Mathematics Research Center MATH+ (EXC-2046/1, EXC-2046/2, project ID: 390685689).
	
	\begingroup
	\ifEWzJhvPageBreak\linespread{0.92}\selectfont\enlargethispage*{11pt}\fi
	\makeatletter
	\EWzJhvLetNew\EWzJhvOriginalBibitem=\bibitem
	\EWzJhvProtectedDefNew\EWzJhvBibitem@C#1.{%
		\EWzJhvOriginalBibitem[Bär99]{ZeroSetsSolutionsElliptic}
		Christian Bär.}
	\EWzJhvDefNew\EWzJhvBibitem@B@Comparand{%
		\EWzJhvOriginalBibitem[B{\"a}r99]{ZeroSetsSolutionsElliptic}
		Christian B{\"a}r}
	\EWzJhvProtectedDefNew\EWzJhvBibitem@B#1.{%
		\begingroup
			\edef\@tempa{\unexpanded{#1}}%
			\ifx\@tempa\EWzJhvBibitem@B@Comparand
				\aftergroup\EWzJhvBibitem@C
			\else
				\errmessage{Unexpected \string\bibitem\space ZeroSetsSolutionsElliptic.}%
			\fi
		\endgroup
		#1.}
	\EWzJhvDefNew\EWzJhvBibitem@A@Comparand{[B{\"a}r99}
	\EWzJhvProtectedDefNew\EWzJhvBibitem@A#1]{%
		\begingroup
			\edef\@tempa{\unexpanded{#1}}%
			\ifx\@tempa\EWzJhvBibitem@A@Comparand
				\aftergroup\EWzJhvBibitem@B
			\fi
		\endgroup
		\EWzJhvOriginalBibitem#1]}
	\protected\def\bibitem{\@ifnextchar[\EWzJhvBibitem@A\EWzJhvOriginalBibitem}
	\EWzJhvLetNew\EWzJhvOriginalSection=\section
	\newcommand*{\EWzJhvSectionWithToc}[1]{\hypersetup{next-anchor=references}\EWzJhvOriginalSection*{#1}\addtocontents{toc}{\EWzJhvTocNoIndentBeginMinipage}\addcontentsline{toc}{section}{#1}\addtocontents{toc}{\EWzJhvTocEndMinipage}\small}
	\protected\def\section{\@ifstar\EWzJhvSectionWithToc\EWzJhvSectionWithToc}
	\makeatother
	\bibliography{refs}

\begin{thebibliography}{HHHN99}

\bibitem[AHS78]{SelfDuality4dimRiemGeometry}
M.~F. Atiyah, N.~J. Hitchin, and I.~M. Singer.
\newblock Self-duality in four-dimensional {Riemannian} geometry.
\newblock {\em Proceedings of the Royal Society of London. Series A,
  Mathematical and Physical Sciences}, 362(1711):425--461, 1978.

\bibitem[AMP98]{CompatibleAlmostCplxStrsOnQKManifolds}
D.~V. Alekseevsky, S.~Marchiafava, and M.~Pontecorvo.
\newblock Compatible almost complex structures on quaternion {Kähler}
  manifolds.
\newblock {\em Annals of Global Analysis and Geometry}, 16:419--444, October
  1998.

\bibitem[B{\"a}r99]{ZeroSetsSolutionsElliptic}
Christian B{\"a}r.
\newblock Zero sets of solutions to semilinear elliptic systems of first order.
\newblock {\em Inventiones mathematicae}, 138:183--202, 1999.

\bibitem[Bes87]{EinsteinManifolds}
Arthur~L. Besse.
\newblock {\em Einstein Manifolds}, volume~10 of {\em Ergebnisse der Mathematik
  und ihrer Grenzgebiete}.
\newblock Springer-Verlag Berlin Heidelberg, 1987.

\bibitem[Che24]{chen2024VWeqKaehler}
Xuemiao Chen.
\newblock On {Vafa--Witten} equations over {Kähler} manifolds.
\newblock {\em Journal für die reine und angewandte Mathematik (Crelles
  Journal)}, 2024(814):135--163, 2024.

\bibitem[DK90]{Geometry4Manifolds}
S.~K. Donaldson and P.~B. Kronheimer.
\newblock {\em The Geometry of Four-Manifolds}.
\newblock Oxford Mathematical Monographs. Oxford University Press, September
  1990.

\bibitem[Gau96]{GauduchonCplxStrCConfMNegT}
Paul Gauduchon.
\newblock Complex structures on compact conformal manifolds of negative type.
\newblock In Vincenzo Ancona, Edoardo Ballico, and Alessandro Silva, editors,
  {\em Complex Analysis and Geometry: Proceedings of the Conference at Trento},
  volume 173 of {\em Lecture notes in pure and applied mathematics}, pages
  201--212. Marcel Dekker, 1996.

\bibitem[HHHN99]{CritSetSolElli}
R.~Hardt, M.~Hoffmann{-}Ostenhof, T.~Hoffmann{-}Ostenhof, and N.~Nadirashvili.
\newblock Critical sets of solutions to elliptic equations.
\newblock {\em Journal of Differential Geometry}, 51(2):359--373, 1999.

\bibitem[Ish74]{QKManifoldsIshihara}
Shigeru Ishihara.
\newblock {Quaternion Kählerian manifolds}.
\newblock {\em Journal of Differential Geometry}, 9(4):483--500, 1974.

\bibitem[Kaz88]{UniqueContinuationGeometry}
Jerry~L. Kazdan.
\newblock Unique continuation in geometry.
\newblock {\em Communications on Pure and Applied Mathematics}, 41(5):667--681,
  1988.

\bibitem[LT95]{KobayashiHitchinCorrespondence}
Martin Lübke and Andrei Teleman.
\newblock {\em The Kobayashi--Hitchin Correspondence}.
\newblock World Scientific, 1995.

\bibitem[Mar10]{MaresPhD}
Bernard~A. Mares.
\newblock {\em Some analytic aspects of {Vafa--Witten} twisted\/ $\mathcal{N} =
  4$ supersymmetric {Yang--Mills} theory}.
\newblock PhD thesis, Massachusetts Institute of Technology, 2010.

\bibitem[Pon94]{CplxStrsOnQuarternionicManifolds}
Massimiliano Pontecorvo.
\newblock Complex structures on quaternionic manifolds.
\newblock {\em Differential Geometry and its Applications}, 4(2):163--177,
  1994.

\bibitem[Rie16]{RieszBernstein}
Marcel Riesz.
\newblock Über einen {Satz} des {Herrn} {Serge} {Bernstein}.
\newblock {\em Acta Mathematica}, 40:337--347, 1916.

\bibitem[Sal72]{LagrangianInterpolation}
H.~E. Salzer.
\newblock Lagrangian interpolation at the {Chebyshev} points $x_{n,\nu} \equiv
  \cos(\nu\pi/n)$, $\nu = 0(1)n$; some unnoted advantages.
\newblock {\em The Computer Journal}, 15(2):156--159, May 1972.

\bibitem[Sal82]{QKManifoldsSalamon}
Simon Salamon.
\newblock Quaternionic {Kähler} manifolds.
\newblock {\em Inventiones mathematicae}, 67:143--171, February 1982.

\bibitem[SV09]{OrthogonalCplxStrR4}
Simon Salamon and Jeff Viaclovsky.
\newblock Orthogonal complex structures on domains in $\mathbb{R}^4$.
\newblock {\em Mathematische Annalen}, 343:853--899, April 2009.

\bibitem[TT17]{Tanaka_2017}
Yuuji Tanaka and Richard~P. Thomas.
\newblock {Vafa--Witten} invariants for projective surfaces {II}: semistable
  case.
\newblock {\em Pure and Applied Mathematics Quarterly}, 13(3):517--562, 2017.

\bibitem[TT19]{Tanaka_2019}
Yuuji Tanaka and Richard~P. Thomas.
\newblock {Vafa--Witten} invariants for projective surfaces {I}: stable case.
\newblock {\em Journal of Algebraic Geometry}, 29(4):603--668, October 2019.

\bibitem[VW94]{VAFA19943}
Cumrun Vafa and Edward Witten.
\newblock A strong coupling test of\/ {S}-duality.
\newblock {\em Nuclear Physics B}, 431(1):3--77, 1994.

\bibitem[Yom84]{SetOfZerosAlmostPolynomial}
Y.~Yomdin.
\newblock The set of zeroes of an ``almost polynomial'' function.
\newblock {\em Proceedings of the American Mathematical Society},
  90(4):538--542, 1984.

\end{thebibliography}
	\endgroup
\end{document}